\documentclass[a4paper]{amsart}

\usepackage{a4wide}
\usepackage{enumerate}
\usepackage{mathrsfs}
\usepackage{amsfonts,amssymb,amscd,amsmath,amsthm,xypic}
\usepackage{txfonts}
\usepackage{accents}

\newif\iflineno 
\linenofalse
\iflineno \usepackage[pagewise]{lineno}\fi
\usepackage{ifpdf}

\ifpdf
    \usepackage[pdftex]{hyperref}
   \usepackage[pdftex]{graphicx}
   
\else
   \usepackage[dvips]{hyperref}
   \usepackage[dvips]{graphicx} 
   
\fi

\newtheorem{Thm}[equation]{Theorem}
\newtheorem*{Thmstar}{Theorem}
\newtheorem*{Setup}{Setup}
\newtheorem{Lem}[equation]{Lemma}
\newtheorem{Cor}[equation]{Corollary}
\theoremstyle{definition}
\newtheorem*{wrongproof}{Wrong argument}

\newtheorem{Def}[equation]{Definition}
\newtheorem{DefandClaim}[equation]{Definition and Claim}
\newtheorem{Asm}[equation]{Assumption}
\newtheorem{Rem}[equation]{Remark}
\newtheorem{Fact}[equation]{Fact}
\newtheorem*{Notation}{Notation}

\theoremstyle{remark}

\numberwithin{equation}{section}

\newcommand{\CLOPEN}{\textsc{clopen}}
\DeclareMathOperator{\lh}{lh}
\DeclareMathOperator{\stem}{stem}
\DeclareMathOperator{\suc}{succ}
\DeclareMathOperator{\Leb}{Leb}
\DeclareMathOperator{\ON}{Ord}
\DeclareMathOperator{\cf}{cf}
\DeclareMathOperator{\dom}{dom}
\DeclareMathOperator{\height}{height}
\DeclareMathOperator{\nullset}{nullset}
\DeclareMathOperator{\ordclos}{ordclos}

\newcommand{\ves}{{\varepsilon'}}
\newcommand{\wit}{\n \xi }
\newcommand{\Int}{L}
\newcommand{\ur}{\tau}
\newcommand{\tomek}{B^*}
\newcommand{\DEFEQ}{\coloneqq}
\newcommand{\bL}{\mathbb{L}}
\newcommand{\bJ}{\mathbb{J}}
\newcommand{\BP}{\mathbf{P}}
\newcommand{\BQ}{\mathbf{Q}}
\newcommand{\forces}{\Vdash}
\newcommand{\incomp}{\perp}
\newcommand{\comp}{\parallel}
\newcommand{\on}{\mathord\restriction}
\newcommand{\pure}{\mathrel{\le_0}}
\newcommand{\MA}{{\mathcal A}}
\newcommand{\powerP}{{\mathscr P}}
\newcommand{\CS}{{\textrm{\textup{CS}}}}
\newcommand{\tot}{{\textrm{\textup{tot}}}}
\newcommand{\prep}{\mathbb{R}}
\newcommand{\esm}{\prec}  
\newcommand{\gen}{\textup{\textrm{gen}}}

\newcommand{\al}[1]{\ensuremath{{\aleph_{#1}}} }          
\newcommand{\om}[1]{\ensuremath{{\omega_{#1}}} }          
\newcommand{\qemph}[1]{``\emph{#1}''}
\newcommand{\std}[1]{\ensuremath{\check{#1}}}              
\newcommand{\n}[1]{\underaccent{\tilde}{#1}}
\newcommand{\nd}[1]{\underaccent{\tilde}{\dot{#1}}}

\newcommand{\proofsection}[1]{\smallskip\noindent{\em \bf #1}.}  
\newcommand{\proofclaim}[2]{
\begin{equation}\label{#1}
  \parbox{0.8\textwidth}{#2}
\end{equation}}
\newcommand{\proofclaimnl}[1]{
\begin{equation*}
  \parbox{0.8\textwidth}{#1}
\end{equation*}}

\begin{document}
\iflineno \pagewiselinenumbers\fi

\renewcommand{\thesubsection}{\thesection.\Alph{subsection}}
\date{2011-12-27}

\title{Borel Conjecture and dual Borel Conjecture}
\author{Martin Goldstern}
\address{Institut f\"ur Diskrete Mathematik und Geometrie\\
 Technische Universit\"at Wien\\
 Wiedner Hauptstra{\ss}e 8--10/104\\
 1040 Wien, Austria}
\email{martin.goldstern@tuwien.ac.at}
\urladdr{http://www.tuwien.ac.at/goldstern/}
\author{Jakob Kellner}
\address{Kurt G\"odel Research Center for Mathematical Logic\\
 Universit\"at Wien\\
 W\"ahringer Stra\ss e 25\\
 1090 Wien, Austria}
\email{kellner@fsmat.at}
\urladdr{http://www.logic.univie.ac.at/$\sim$kellner/}
\author{Saharon Shelah}
\address{Einstein Institute of Mathematics\\
Edmond J. Safra Campus, Givat Ram\\
The Hebrew University of Jerusalem\\
Jerusalem, 91904, Israel\\
and
Department of Mathematics\\
Rutgers University\\
New Brunswick, NJ 08854, USA}
\email{shelah@math.huji.ac.il}
\urladdr{http://shelah.logic.at/}
\author{Wolfgang Wohofsky}
\address{Institut f\"ur Diskrete Mathematik und Geometrie\\
 Technische Universit\"at Wien\\
 Wiedner Hauptstra{\ss}e 8--10/104\\
 1040 Wien, Austria}
\email{wolfgang.wohofsky@gmx.at}
\urladdr{http://www.wohofsky.eu/math/}
\thanks{
We gratefully acknowledge the following partial support: US National Science
Foundation Grant No. 0600940 (all authors); US-Israel Binational Science
Foundation grant 2006108 (third author); Austrian Science Fund (FWF): P21651-N13 and P23875-N13 and EU FP7 Marie Curie grant PERG02-GA-2207-224747 (second and fourth author);
FWF grant P21968 (first and fourth author); \"OAW Doc fellowship (fourth author).
This is publication 969 of the third author.\\
 We are grateful to the anonymous referee for pointing out several
unclarities in the original version.}

\dedicatory{Dedicated to the memory of Richard Laver}

\begin{abstract}
	 We show that it is consistent that the Borel Conjecture and the dual Borel
   Conjecture hold simultaneously.
\end{abstract}

\maketitle

\section*{Introduction}
\subsection*{History}

A set  $X$ of reals\footnote{In this paper, we use $2^\omega$ as 
the set of reals. ($\omega=\{0,1,2,\ldots\}$.)  
By well-known results both the definition and the theorem 
also work for the unit interval~$ [0,1]$ or the torus $\mathbb R/\mathbb Z$.
Occasionally we also write ``$x$ is a real'' for ``$x\in \omega^\omega$''.} 
is called ``strong measure zero'' (smz), if for
all functions $f:\omega\to\omega$ there are intervals $I_n$ of measure $\leq
1/f(n)$ covering $X$. 
Obviously, a smz set is a null set (i.e., has Lebesgue measure zero), and it
is easy to see that the family of smz sets forms a $\sigma$-ideal and that
perfect sets (and therefore uncountable Borel or analytic sets) are not smz.

At the beginning of the 20th century, Borel~\cite[p.~123]{MR1504785}
conjectured:
\proofclaimnl{Every smz set is countable.}
This statement is known as the ``Borel Conjecture'' (BC). In the 1970s it was proved
that BC is  \emph{independent}, i.e., neither provable nor refutable.

Let us very briefly comment on the notion of independence: A sentence $\varphi$
is called independent of a set $T$ of axioms, if neither $\varphi$ nor $\lnot\varphi$
follows from $T$. (As a trivial example, $(\forall x)(\forall y) x\cdot y= y\cdot x$ is
independent from the group axioms.) The set theoretic (first order) axiom system
ZFC (Zermelo Fraenkel with the axiom of choice) is considered to be the
standard axiomatization of all of mathematics: A mathematical proof is
generally accepted as valid iff it can be formalized in ZFC. Therefore we just
say ``$\varphi$ is independent'' if $\varphi$ is independent of ZFC.
Several mathematical statements are independent, the earliest and most
prominent example is Hilbert's first problem, the Continuum Hypothesis
(CH).

BC is independent as well:  Sierpi\'nski~\cite{sierpinskiCH} showed
that CH implies $\lnot$BC (and, since G\"odel showed the consistency of CH, this
gives us the consistency of $\lnot$BC). Using the method of forcing,
Laver~\cite{MR0422027} showed that BC is consistent.

Galvin, Mycielski and Solovay~\cite{GMS} proved the following conjecture
of Prikry:
\proofclaimnl{$X \subseteq 2^\omega $ is smz
  if and only if  every comeager (dense $G_\delta$) set
  contains a translate of $X$.}
Prikry also defined the following dual notion:
\proofclaimnl{$X \subseteq 2^\omega$ is called ``strongly meager'' (sm) if every set
  of Lebesgue measure 1  contains a translate of $X$.}
The dual Borel Conjecture (dBC) states:
\proofclaimnl{Every sm set is countable.}
Prikry noted that CH implies $\lnot$dBC and conjectured dBC to be consistent
(and therefore independent), which was later proved by
Carlson~\cite{MR1139474}.

Numerous additional results regarding BC and dBC have been proved: The
consistency of variants of BC or of dBC, the consistency of BC or dBC together
with certain assumptions on cardinal characteristics, etc.
See~\cite[Ch.~8]{MR1350295} for several of these results.  In this paper, we
prove the consistency (and therefore independence) of BC+dBC (i.e.,
consistently BC and dBC hold simultaneously).

\subsection*{The problem}

The obvious first attempt to force BC+dBC is to somehow combine Laver's and
Carlson's constructions. However, there are strong obstacles:

Laver's construction is a countable support iteration of Laver forcing. The
crucial points are:
\begin{itemize}
  \item Adding a Laver real makes every old uncountable set $X$ non-smz.
  \item And this set $X$ remains non-smz after another forcing $P$, provided
    that $P$ has the ``Laver property''.
\end{itemize}
So we can start with CH and use a countable support iteration of Laver forcing of
length $\om2$. In the final model, every set $X$ of reals of size~$\al1$
already appeared at some stage $\alpha<\om2$ of the iteration; the next Laver real
makes $X$ non-smz, and the rest of the iteration (as it is a countable support
iteration of proper forcings with the Laver property) has the Laver property,
and therefore $X$ is still non-smz in the final model.

Carlson's construction on the other hand adds $\omega_2$ many 
Cohen reals in a finite support iteration (or
equivalently: finite support product).  The crucial points are:
\begin{itemize}
  \item A Cohen real makes every old uncountable set $X$ non-sm.
  \item And this set $X$ remains non-sm after another forcing $P$, provided
    that $P$ has precaliber~$\al1$.
\end{itemize}
So we can start with CH, and use more or less the same argument as above:
Assume that $X$ appears at $\alpha<\om2$. Then the next Cohen makes $X$ non-sm.  It
is enough to show that $X$ remains non-sm at all subsequent  stages
$\beta<\om2$.  This is guaranteed by the fact that a finite support iteration
of Cohen reals of length~$<\om2$ has precaliber~$\al1$.

So it is unclear how to combine the two proofs: A Cohen real makes all old sets
smz, and it is easy to see that whenever we add Cohen reals cofinally often in
an iteration of length, say, $\om2$, all sets of any intermediate extension
will be smz, thus violating BC.  So we have to avoid Cohen
reals,\footnote{An iteration that forces dBC without adding Cohen reals was given in 
\cite{MR2767969}, using non-Cohen oracle-cc.}
which also
implies that we cannot use finite support limits in our iterations.
So we have
a problem even if we find a replacement for Cohen forcing in Carlson's proof
that makes all old uncountable sets $X$ non-sm and that does not add Cohen
reals: Since we cannot use finite support, it seems hopeless to get
precaliber~$\aleph_1$,
an essential requirement to keep $X$ non-sm.

Note that it is the \emph{proofs} of BC and  dBC that are seemingly
irreconcilable; this is not clear for the models.  Of course Carlson's model,
i.e., the Cohen model, cannot satisfy BC, but it is not clear whether maybe
already the Laver model could satisfy dBC. (It is even still open whether a
single Laver forcing makes every old uncountable set non-sm.) Actually,
Bartoszy\'nski and Shelah~\cite{MR2020043} proved that the Laver model does
satisfy the following weaker variant of dBC (note that the continuum has size
$\al2$ in the Laver model):
\begin{quote}
  Every sm set has size less than the continuum.
\end{quote}

In any case, it turns out that one \emph{can} reconcile Laver's and Carlson's proof,
by ``mixing'' them ``generically'', resulting in the following theorem:
\begin{Thmstar}
If ZFC is consistent, then ZFC+BC+dBC is consistent. 
\end{Thmstar}

\subsection*{Prerequisites}

To understand anything of this paper, the reader
\begin{itemize}
	\item should have some experience with finite and countable support
	iteration,  proper forcing, $\al2$-cc, $\sigma$-closed, etc.,
  \item should know what a quotient forcing is,
  \item should have seen some preservation theorem for proper countable support iteration, 
  \item should have seen some tree forcings (such as Laver forcing).
\end{itemize}

To understand everything, additionally the following is required:
\begin{itemize}
	\item The ``case A'' preservation theorem from~\cite{MR1623206},
		 more specifically we build on 
		 the proof
    of~\cite{MR1234283} (or~\cite{MR2214624}).
    \item In particular, some familiarity with the property ``preservation of randoms''
    is recommended.  We will use the fact that random and Laver forcing 
    have this property.
  \item We make some claims about (a rather special case of) 
    ord-transitive models in Section~\ref{subsec:ordtrans}. 
    The readers can either believe these claims, or check them themselves
    (by some  rather straightforward proofs),
    or look up the proofs (of more general 
    settings) in~\cite{MR2115943} or~\cite{kellnernep}.
\end{itemize}

{}From the theory of  strong measure zero and strongly meager,  we only need the
following two results (which are essential for our proofs of BC and dBC,
respectively):
\begin{itemize}
  \item Pawlikowski's result from~\cite{MR1380640} (which we quote as Theorem~\ref{thm:pawlikowski} below), and
  \item Theorem 8 of Bartoszy\'nski and Shelah's~\cite{MR2767969} (which we
	quote as Lemma~\ref{lem:tomek}).
\end{itemize}
We do not need any other results of Bartoszy\'nski and Shelah's
paper~\cite{MR2767969}; in particular we do not use the notion of non-Cohen
oracle-cc (introduced in~\cite{MR2243849}); and the reader does not have to know the original proofs of Con(BC)
and Con(dBC), by Laver and Carlson, respectively.

The third author claims that our construction is more or less the same as a
non-Cohen oracle-cc construction, and that the
extended version presented in~\cite{MR2610747} is even closer to our
preparatory forcing.

\subsection*{Notation and some basic facts on forcing, strongly meager (sm) and 
strong measure zero (smz) sets}

We call a lemma ``Fact'' if we think that no proof is necessary --- either because 
it is trivial, or because it is well known (even without a reference),
or because we give an explicit reference to the literature.

Stronger conditions in forcing notions are smaller, i.e., $q\leq p$ means that
 $q$ is stronger than $p$.

Let $P\subseteq Q$ be forcing notions. (As usual, we abuse notation by not
distinguishing between the underlying set and the quasiorder on it.)
\begin{itemize}
  \item For $p_1,p_2\in P$ we write $p_1\incomp_P p_2$ for
    ``$p_1$ and $p_2$ are incompatible''. Otherwise we write $p_1 \comp_P p_2$.
    (We may just write $\incomp$ or $\comp$ if $P$ is understood.)
  \item\label{def:starorder} $q\leq^* p$ (or: $q\leq^*_P p$) means that $q$ forces that $p$ is in the generic
    filter, or equivalently that 
    every $q'\leq q$ is compatible with $p$. 
    And $q=^* p$ means $q\leq^* p\ \wedge\ p\leq^* q$.
  \item\label{def:separative} $P$ is separative, if $\leq$ is the same as $\leq^*$,
    or equivalently, if for all $q\leq p$ with $q\neq p$ there is an $r\leq p$
    incompatible with $q$. Given any $P$, we can define its ``separative
    quotient'' $Q$ by first replacing (in $P$) $\leq$ by $\leq^*$ and then
    identifying elements $p,q$ whenever $p=^*q$. Then $Q$ is separative
    and forcing equivalent to $P$.
  \item \qemph{$P$ is a subforcing of $Q$}  means that the relation $\le_P$
     is the restriction of $\le_Q$ to~$P$.
  \item \qemph{$P$ is an incompatibility-preserving subforcing of $Q$} 
    means that $P$ is a subforcing of
    $Q$  and that 
    $p_1\incomp_P p_2$ iff $p_1\incomp_Q p_2$ for all $p_1,p_2\in P$.
\end{itemize}
Let additionally $M$ be a countable 
transitive\footnote{We will also use so-called
ord-transitive models, as defined in Section~\ref{subsec:ordtrans}.}
model (of a sufficiently large subset of ZFC) containing~$P$.
  \begin{itemize}
  \item ``$P$ is an $M$-complete subforcing of
    $Q$'' (or: $P\lessdot_M Q$) 
    means that $P$ is a subforcing of $Q$ and: if $A\subseteq P$ is in
    $M$ a maximal antichain, then it is a maximal antichain of~$Q$ as
    well. (Or equivalently: $P$ is an incompatibility-preserving
    subforcing of $Q$ and every predense subset of
    $P$~in $M$ is predense in~$Q$.)
    Note that this means that every $Q$-generic filter $G$ over~$V$
		induces a $P$-generic filter over~$M$, namely  $G^M\DEFEQ G\cap P$
    (i.e., every maximal antichain of
    $P$ in~$M$ meets $G\cap P$ in exactly one point).
    In particular, we can interpret a $P$-name $\tau$ in~$M$ as a $Q$-name.
    More exactly, there is a $Q$-name $\tau'$ such that $\tau'[G]=\tau[G^M]$
    for all $Q$-generic filters $G$. We will usually just identify $\tau$
    and $\tau'$.
  \item Analogously, if $P\in M$ and $i:P\to Q$ is a function, then
    $i$ is called an $M$-complete embedding if it preserves $\leq$
    (or at least $\leq^*$)
    and $\incomp$ and moreover: If $A\in M$ is predense in~$P$, then
    $i[A]$ is predense in $Q$.
\end{itemize}

There are several possible characterizations  of  sm (``strongly 
meager'')  and smz (``strong measure zero'') sets; we will use the following
as definitions: 

A set  $X$ is not sm if there is a measure $1$ set into
which $X$ cannot be translated; i.e., if there is a null set $Z$
such that $(X+t)\cap Z\neq\emptyset$ for all reals $t$, or, in other words,
$Z+X=2^\omega$. To summarize:
\proofclaim{eq:notsm}{$X$ is {\em not} sm iff there is a 
Lebesgue null set $Z$ such that $Z+X=2^\omega$.}

We will call such a $Z$ a ``witness'' for the fact that $X$ is not sm (or say
that $Z$ witnesses that $X$ is not sm).

The following theorem of Pawlikowski~\cite{MR1380640} is central for our
proof\footnote{We thank Tomek Bartoszy\'nski for pointing out Pawlikowski's
result to us, and for suggesting that it might be useful for our proof.}
that BC holds in our model:
\begin{Thm}\label{thm:pawlikowski}
  $X\subseteq 2^\omega$ is smz iff $X+F$ is null for every closed null set $F$.
  \\
  Moreover, for every dense $G_\delta$ set $H$ we can \emph{construct} (in
  an absolute way) a closed null set $F$ such that for every $X
  \subseteq 2^\omega$ with $X+F$ null there is $t\in 2^\omega$ with
  $t+X\subseteq H$.
\end{Thm}

In particular, we get:
\proofclaim{eq:notsmz}{$X$ is \emph{not} smz iff there is a closed null
set $F$ such that $X+F$ has positive outer Lebesgue measure.}

Again, we will say that the closed null set $F$ ``witnesses'' that 
$X$ is not smz (or call $F$ a witness for this fact).

\subsection*{Annotated contents}
\begin{list}{}{\setlength{\leftmargin}{0.5cm}\addtolength{\leftmargin}{\labelwidth}}
\item[Section~\ref{sec:ultralaver}, p. \pageref{sec:ultralaver}:]
  We introduce the family of ultralaver forcing notions 
  and prove some properties.
\item[Section~\ref{sec:janus}, p. \pageref{sec:janus}:]
  We introduce the family of Janus forcing notions 
  and prove some properties.
\item[Section~\ref{sec:iterations}, p. \pageref{sec:iterations}:]
  We define ord-transitive models and mention some basic properties.
  We define the ``almost finite'' and ``almost countable'' support
  iteration over a model.
  We show that in many respects they behave like finite
  and countable support, respectively.
\item[Section~\ref{sec:construction}, p. \pageref{sec:construction}:]
  We introduce the
  preparatory forcing notion $\prep$ which adds a generic forcing
  iteration~$\bar \BP$. 
\item[Section~\ref{sec:proof}, p. \pageref{sec:proof}:] Putting 
 everything together, we show that $\prep*\BP_{\om2}$
 forces BC+dBC, i.e., that an uncountable $X$ is neither
 smz nor sm. We show this under the assumption $X\in V$,
 and then introduce a 
 factorization
 of  $\prep*\bar \BP$ that this assumption does not result in loss
 of generality.
\item[Section~\ref{sec:alternativedefs}, p. \pageref{sec:alternativedefs}:] We briefly comment on alternative ways some notions could be defined.
\end{list}

An informal overview of the proof, including two illustrations, can be found
at~\url{http://arxiv.org/abs/1112.4424/}.

\section{Ultralaver forcing}\label{sec:ultralaver}

In this section, we define the family of  \emph{ultralaver forcings} $\bL_{\bar
D}$, variants  of Laver forcing which depend on a system $\bar D$ of
ultrafilters.  

In the rest of the paper, we will use the following properties of $\bL_{\bar
D}$.  
(And we will use \emph{only} these properties. So readers who are willing to
take these properties for granted could skip to Section~\ref{sec:janus}.)
  \begin{enumerate}
  \item
     $\bL_{\bar D}$ is $\sigma$-centered, hence ccc.\label{item:sigmacentered}
     \\
    (This is Lemma~\ref{lem:newscentered}.)
   \item $\bL_{\bar D}$ is separative. \\
           (This is Lemma~\ref{lem:LDMsep}.)  
  \item\label{item:absolutepositive} 
    \emph{Ultralaver kills smz:}
    There is a canonical $\bL_{\bar D}$-name $\bar{\n\ell}$ for a
    fast growing real in~$\omega^\omega$ called the ultralaver real. From this
    real, we can define  (in an absolute way) a closed null  set $F$
    such that
    $X+F$ is positive for all uncountable $X$ in~$V$ (and therefore
    $F$ witnesses that $X$ is not smz, according to Theorem~\ref{thm:pawlikowski}).
    \\
    (This is Corollary~\ref{cor:absolutepositive}.)
  \item 
    Whenever $X$ is uncountable, then $\bL_{\bar D} $ forces that
    $X$ is not ``thin''.
    \\
   (This is  Corollary~\ref{cor:LDnotthin}.)
  \item 
    If $(M,\in)$ is a countable model of ZFC$^*$ and if $\bL_{\bar D^M}$ is an
    ultralaver forcing in $M$,
      then for any ultrafilter system $\bar D$ extending $\bar D^M$, 
     $\bL_{\bar D^M} $ is an $M$-complete subforcing of the ultralaver forcing
      $\bL_{\bar D}$. 
   \\
   (This is Lemma~\ref{lem:LDMcomplete}.)
    \\
    Moreover, the real $\bar{\n\ell}$ of
    item~(\ref{item:absolutepositive}) is so ``canonical'' that we get:
    If (in $M$) $\bar{\n\ell}^M$ is the  $\bL_{\bar D^M}$-name for the 
    $\bL_{\bar D^M}$-generic real,
    and if (in $V$) $\bar{\n\ell}$ is the $\bL_{\bar D}$-name for the
    $\bL_{\bar D}$-generic real, and if $H$ is $\bL_{\bar D}$-generic over
    $V$ and thus $H^M\DEFEQ H\cap \bL_{\bar D^M}$ is the induced 
    $\bL_{\bar D^M}$-generic filter over $M$, then
    $\bar{\n\ell}[H]$ is  equal to 
    $ \bar{\n \ell}^M[H^M]$.
    \\
    Since the closed null set $F$ 
       is constructed from $\bar{\n\ell}$
    in an absolute way, the same holds for $F$, i.e., the 
    Borel codes  $F[H]$ and $F[H^M]$ are the same.
  \item
    Moreover, given $M$ and $\bL_{\bar D^M}$ as above, and a random real
    $r$ over~$M$, we can choose $\bar D$ extending $\bar D^M$
    such that $\bL_{\bar D}$
    forces that randomness of~$r$ is
    preserved (in a strong way that can be preserved in a 
    countable support iteration). 
    \\
    (This is Lemma~\ref{lem:extendLDtopreserverandom}.)
  \end{enumerate}

\subsection{Definition of ultralaver}

\begin{Notation} We use the following fairly standard notation:

  A \emph{tree} is a nonempty
  set $p \subseteq \omega^{<\omega}$ which is closed under initial segments
  and has no maximal elements.\footnote{Except for the 
  proof of Lemma~\ref{lem:LDMcomplete},
  where we also allow trees with maximal elements, and even empty trees.}
   The elements  (``nodes'') of a
  tree are partially ordered by $\subseteq$.

   For each sequence $s\in \omega^{<\omega}$ we write $\lh(s)$ for the
   length of $s$.

  For any tree $p \subseteq \omega^{<\omega}$ and any $s\in p$ we write 
  $\suc_p(s)$ for one of the following two sets: 
  \[ \{ k\in \omega: s^\frown k \in p \} \text{ \ \  or \ \  } \{ t\in p:
  (\exists k\in \omega)\;\,  t=s^\frown k \} \]
  and we rely on the context to help the reader decide which set we
  mean. 
  
  A \emph{branch} of $p$  is either of the following: 
  \begin{itemize}
  \item A function $f:\omega\to \omega$ with $f\on n\in p$ for all
    $n\in \omega$.
  \item A maximal chain in the partial order $(p,\subseteq)$.  (As 
    our trees do not have maximal elements, each
    such chain $C$ determines a branch $\bigcup C$ in the first sense,
    and conversely.)
  \end{itemize}
  We write $[p]$ for the set of all branches of~$p$. 
  
  For any tree $p\subseteq
  \omega^{<\omega}$ and any $s\in p$ we write $p^{[s]}$ for the set
  $\{t\in p: t \supseteq s \text{ or } t \subseteq s\}$, and we write
  $[s]$ for either of the following sets:
  \[ \{ t\in p:  s \subseteq t  \} \text{ \ \  or \ \  } \{ x \in [p]: s
  \subseteq x \}. \]

  The stem of a tree $p$ is the
  shortest $s\in p $ with $|\suc_p(s)|>1$. (The trees we
  consider will never be branches, i.e., will always have finite stems.)

\end{Notation}

\begin{Def}\label{def:LD}
  \begin{itemize}
  \item
    For trees $q,p$ we  write $q\le p$ if  $q \subseteq p$ (``$q$ is stronger
    than~$p$''), and 
    we say that  \qemph{$q$ is a pure extension of~$p$} ($q\pure p$) if
    $q\le p$ and $\stem(q)=\stem(p)$.  
  \item 
    A filter system  $\bar D$ is a family $(D_s)_{s\in
    \omega^{<\omega}}$ of filters on~$\omega$. (All our filters will 
    contain the Fr\'echet filter of cofinite sets.) We write $D_s^+$
    for the collection of $D_s$-positive sets (i.e., sets whose complement
    is not in $D_s$).
  \item  We define $\bL_{\bar D} $
    to be the set of all trees $p$ such that
      $\suc_p(t)\in D_t^+$ for all $t\in p$ above the stem.
  \item
    The generic filter is determined by the generic branch $ \bar\ell
    = (\ell_i)_{i\in \omega}\in \omega^\omega$, called the \emph{generic real}:
    $\{\bar\ell\} = \bigcap_{p\in G} [p]$ or equivalently, $ \bar\ell =
    \bigcup_{p\in G} \stem(p)$.
  \item An ultrafilter system is a filter system consisting of
    ultrafilters.  (Since all our filters contain the Fr\'echet
    filter, we only consider nonprincipal ultrafilters.)  
  \item
    An \emph{ultralaver forcing} is a forcing $\bL_{\bar D}$ defined from an
    ultrafilter system.  The generic real for an ultralaver forcing is
    also called the \emph{ultralaver real}. 
  \end{itemize}
\end{Def}

Recall that a forcing notion $(P,\le)$ is \emph{$\sigma$-centered} if 
$P = \bigcup_n P_n$, where for all $n,k\in \omega$ and for all 
$p_1,\ldots, p_k\in P_n$ there is $q\le p_1,\ldots, p_k$.

\begin{Lem}\label{lem:newscentered}
All ultralaver forcings $\bL_{\bar D}$ are $\sigma$-centered (hence ccc).
\end{Lem}

\begin{proof}
Every finite set of conditions sharing the same stem has a common lower bound. 
\end{proof}
\begin{Lem}\label{lem:LDMsep} 
    $\bL_{\bar D}$ is separative.\footnote{See page~\pageref{def:separative} for 
the definition.}
\end{Lem}
\begin{proof}  If $q\le p$, and $q\not=p$, then there is $s\in p\setminus q$. 
Now $p^{[s]} \perp q$. 
\end{proof}

If each $D_s$ is the Fr\'echet filter, then $\bL_{\bar D}$ is Laver forcing
(often just written $\bL$).

\subsection{$M$-complete embeddings}

  Note that for all  ultrafilter systems
  $\bar D$ we have:
\proofclaim{eq:compatible}{
  Two conditions in $\bL_{\bar D}$ are compatible
  if and only if their stems are comparable and moreover, the longer
  stem is an element of the condition with the shorter stem.  
}

\begin{Lem}\label{lem:LDMcomplete}
  Let $M$ be countable.\footnote{Here, 
    we can assume that $M$ is a
    countable transitive model of a sufficiently large finite 
    subset ZFC$^*$ of ZFC. Later, we will also use ord-transitive models  
    instead of transitive ones, which does not make any difference 
    as far as properties of $\bL_{\bar D}$ are concerned, as our arguments
    take place in transitive parts of such models.}
  In~$M$, let $\bL_{\bar D^M}$ be an ultralaver forcing. Let $\bar D$
  be (in $V$) a filter system extending\footnote{I.e.,
    $D_s^M \subseteq D_s$ for all $s\in \omega^{<\omega}$.}
  $\bar D^M$.
  Then $\bL_{\bar D^M} $ is an $M$-complete subforcing of  $\bL_{\bar D}$.
\end{Lem}

\begin{proof}
  For any tree\footnote{Here we also allow empty trees, and trees
  with maximal nodes.}~$T$, any filter system 
  $\bar E = (E_s)_{s\in \omega^{<\omega}}$, 
  and any ${s_0}\in T$ we define a sequence
  $(T_{\bar E,{s_0}}^\alpha)_{\alpha\in \omega_1}$
  of ``derivatives''
  (where we may abbreviate $T_{\bar E,{s_0}}^\alpha$ to $T^\alpha$)
  as follows:
  \begin{itemize}
  \item $T^0\DEFEQ  T^{[{s_0}]}$. 
  \item Given $T^\alpha$, we let
    $T^{\alpha+1}\DEFEQ  T^\alpha  \setminus 
    \bigcup \{ [s] : s\in T^\alpha , {s_0}\subseteq s,
    \suc_{T^\alpha}(s)\notin E_s^+ \}$,
    where $[s]\DEFEQ  \{t: s\subseteq t\}$. 
  \item For limit ordinals $\delta>0$
    we let $T^\delta\DEFEQ \bigcap_{\alpha<\delta} T^\alpha$.
  \end{itemize}
  Then we have
  \begin{itemize}
  \item [(a)] Each $T^\alpha$ is closed under initial segments. 
    Also: $\alpha < \beta$ implies $  T^\alpha \supseteq T^\beta$.
  \item [(b)]  There is an  $\alpha_0<\omega_1$ such that  $T^{\alpha_0} =
    T^{\alpha_0+1} = T^\beta$ for all $\beta>\alpha_0$.  We write
    $T^\infty$ or $T^\infty_{\bar E,{s_0}}$ for $T^{\alpha_0}$.
  \item[(c)] If ${s_0}\in T_{\bar E,{s_0}}^\infty$, then $T_{\bar
    E,{s_0}}^\infty\in \bL_{\bar E}$ with stem~${s_0}$.  \\ 
    Conversely, if $\stem(T)={s_0}$, and $T\in \bL_{\bar E}$, then
    $T^\infty=T$. 
  \item[(d)] If $T$ contains a tree $q\in \bL_{\bar E}$ with
    $\stem(q)={s_0}$,
    then $T^\infty$ contains $q^\infty=q$,
    so in particular ${s_0}\in T^\infty$.
  \item[(e)] Thus:  $T$ contains a condition in $\bL_{\bar E}$ with stem ${s_0}$
    iff ${s_0}\in T^\infty_{\bar E,{s_0}}$.
  \item[(f)] The computation of $T^\infty$ is absolute between any two models 
    containing $T$ and $\bar E$. (In particular, any transitive ZFC$^*$-model
    containing $T$ and $\bar E$  will also contain
    $\alpha_0$.) 
  \item[(g)] Moreover: Let $T\in M$, $\bar E\in M$, and let $\bar E'$ be a filter system 
    extending $\bar E$ such that for all ${s_0}$ and all
    $A\in \powerP(\omega)\cap M$  we have: $A\in (E_{s_0})^+$ iff $A\in
    (E_{s_0}')^+$.
    (In particular, this will be true for any $\bar E'$ extending $\bar E$,
     provided that each $E_{s_0}$ is an
     $M$-ultrafilter.)
    \\
    Then
     for each $\alpha\in M$ we have  $T^\alpha_{\bar E,{s_0}}= T^\alpha_{\bar
      E',{s_0}}$ (and hence $T^\alpha_{\bar E',{s_0}}\in M$).
      (Proved by induction on~$\alpha$.)
  \end{itemize}
  
  Now let $A  = (p_i:i\in I)\in M$ be a maximal antichain in
  $\bL_{\bar D^M}$, and assume (in $V$) that $q\in \bL_{\bar D}$.  Let
  ${s_0}\DEFEQ \stem(q)$.

  We will show that $q$ is compatible with some~$p_i$ (in $\bL_{\bar D}$).
 This is clear
  if there is some $i$ with ${s_0}\in p_i$ and $\stem(p_i)\subseteq
  {s_0}$, by~\eqref{eq:compatible}. (In this case, $p_i \cap q$
  is a condition in $\bL_{\bar D}$ with stem $s_0$.)

  So for the rest of the proof 
   we assume that this is not the case, i.e.: 
   \proofclaim{eq:not.the.case}{
  There is no $i$ with $s_0 \in p_i $ and $\stem(p_i)\subseteq s_0$.
  } 
     Let $J\DEFEQ \{ i\in I: {s_0}
  \subseteq \stem(p_i)\}$.  We claim that there is $j\in J$ with
  $\stem(p_j)\in q$ (which as above implies that $q$ and $p_j$ are compatible).
  
  Assume towards a contradiction that this is not the case. 
  Then $q$ is contained in the following tree $T$:
  \begin{align}\label{def:T}
    T \DEFEQ  
    (\omega^{<\omega})^{[{{s_0}}]}\setminus 
    \bigcup _{j\in J} [\stem(p_j)].
  \end{align}

  Note that $T\in M$.  In $V$ we have:
  \proofclaim{eq:T.contains.q}{ The tree $T$ contains a condition
    $q$ with stem ${s_0}$.}
  So by (e) (applied in $V$),  followed by (g), and again by (e) (now in $M$) we get: 
  \proofclaim{eq:T.contains.p}{
    The tree $T$ also
    contains a  condition $p\in M$ with stem ${s_0}$.}
  Now $p$ has to be
  compatible with some~$p_i$. The sequences ${s_0}=\stem(p)$ and
  $\stem(p_i)$ have to be comparable, so by~\eqref{eq:compatible}
  there are two possibilities:
  \begin{enumerate}
  \item $\stem(p_i)\subseteq \stem(p) = s_0 \in p_i$. We have excluded this
    case in our assumption \eqref{eq:not.the.case}. 
  \item $s_0 = \stem(p) \subseteq \stem(p_i)\in p$.  So $i\in J$. 
    By construction of~$T$ (see~\eqref{def:T}),
    we conclude $\stem(p_i)\notin T$, contradicting
    $\stem(p_i)\in p\subseteq T$ (see~\ref{eq:T.contains.p}).
    \qedhere
  \end{enumerate}
\end{proof}

\subsection{Ultralaver kills strong measure zero}

The following lemma appears already in \cite[Theorem 9]{MR942525}. We will give a
proof below in Lemma~\ref{lem:pure}.

\begin{Lem}\label{lem:pure.finite}
  If $A$ is a finite set, $\n \alpha$ an $\bL_{\bar D}$-name, $p\in
    \bL_{\bar D}$, and $p\forces\n \alpha\in A$, then there is
    $\beta\in A$ and a pure extension $q\pure  p $ such that $q\forces
    \n \alpha=\beta$.
\end{Lem}

\begin{Def}
  Let $\bar\ell$ be an increasing sequence of natural numbers.
  We say that $X\subseteq 2^\omega$ is
   \emph{smz with respect to~$\bar\ell$}, if there
  exists a sequence $(I_k)_{k\in\omega}$ of basic intervals
  of $2^\omega$ of measure $\leq 2^{-\ell_k}$
  (i.e.,  each $I_k$ is of the form $[s_k]$ for some
  $s_k\in  2^{\ell_k }$)  such that
  $X\subseteq\bigcap_{m\in \omega} \bigcup_{k\ge m} I_k$.
\end{Def}

\begin{Rem}
It is well known and  easy to see that the properties 
\begin{itemize}
  \item For all $\bar\ell$ there exists 
  exists a sequence $(I_k)_{k\in\omega}$ of basic intervals
  of $2^\omega$ of measure $\leq 2^{-\ell_k}$
  such that $X\subseteq \bigcup_{k\in\omega} I_k$.
  \item For all $\bar\ell$ there exists 
  exists a sequence $(I_k)_{k\in\omega}$ of basic intervals
  of $2^\omega$ of measure $\leq 2^{-\ell_k}$
  such that $X\subseteq\bigcap_{m\in \omega} \bigcup_{k\ge m } I_k$.
\end{itemize}
are equivalent.  Hence, a set $X$ is smz iff $X$ is smz with respect to all 
  $\bar\ell\in \omega^\omega$.
\end{Rem}

The following lemma is a variant of the corresponding lemma 
 (and proof)
 for Laver forcing (see for example \cite[Lemma~28.20]{MR1940513}): 
Ultralaver makes old uncountable  sets non-smz.

\begin{Lem}\label{lem:LDdestroysSMZ}
  Let $\bar D$ be a system of ultrafilters, and let $\bar{\n\ell}$ be the
  $\bL_{\bar D}$-name for the ultralaver real. Then each uncountable set $X \in V$ is forced to be non-smz (witnessed by the ultralaver real $\bar{\n\ell}$).
  
More precisely, the following holds:
  
 \begin{equation}\label{eq:my_non_smz}
 \forces_{\bL_{\bar D}} \forall X \in V  \cap [2^\omega]^{\aleph_1}\;\; \forall (x_k)_{ k \in \omega} \subseteq 2^\omega \;\; X \not\subseteq  \bigcap_{m \in \omega} \bigcup_{k \geq m} [x_k \on \n\ell_k].
 \end{equation}
\end{Lem}

We first give two technical lemmas:

\begin{Lem}\label{lem:first_technical}
Let $p \in \bL_{\bar D}$ with stem $s \in \omega^{<\omega}$, and let $\n x$ be a $\bL_{\bar D}$-name for a real in $2^\omega$. Then there exists a pure extension $q \leq_0 p$ and a real $\ur \in 2^\omega$ such that for every $n \in \omega$,
\begin{equation}\label{eq:first_technical}
\{ i \in\suc_q(s):\; q^{[s^\frown i]} \forces \n x \on n = \ur \on n \} \in D_s.
\end{equation}
\end{Lem}
\begin{proof}
For each $i \in \suc_p(s)$, let $q_i \leq_0 p^{[s^\frown i]}$ be such that $q_i$ decides 
$\n x \on i$, 
i.e., 
there is a $t_i$ of length $i$ such that 
$q_i \forces \n x \on i = t_i$ (this is possible by Lemma~\ref{lem:pure.finite}).

Now we define the real $\ur \in 2^\omega$ as the $D_s$-limit of the $t_i$'s. In more detail: For each $n \in \omega$ there is a (unique) $\ur_n \in 2^n$ such that $\{ i:\; t_i \on n = \ur_n \} \in D_s$; since $D_s$ is a filter, there is a real $\ur \in 2^\omega$ with $\ur \on n = \ur_n$ for each $n$. Finally, let 
$q \DEFEQ  \bigcup_i q_i$.
\end{proof}

\begin{Lem}\label{lem:second_technical}
Let $p \in \bL_{\bar D}$ with stem $s$, and let $(\n x_k)_{ k \in \omega}$ be a
sequence of $\bL_{\bar D}$-names for reals in $2^\omega$. Then there exists a
pure extension $q \leq_0 p$ and a family of reals $(\ur_\eta)_{ \eta \in q,\,
\eta \supseteq s} \subseteq  2^\omega$ such that for each $\eta \in q$
above~$s$, and every $n \in \omega$,
\begin{equation}\label{eq:second_technical}
\{ i \in  \suc_q(\eta):\; q^{[\eta^\frown i]} \forces \n x_{|\eta|} \on n = \ur_\eta \on n \} \in D_\eta.
\end{equation}
\end{Lem}
\begin{proof}

We apply Lemma~\ref{lem:first_technical} to each node $\eta$ in $p$  above $s$ (and
to $\n x_{|\eta|}$) separately: 
We first get a $p_1 \leq_0 p$ and a $\ur_s \in 2^\omega$;
for every immediate successor $\eta \in \suc_{p_1}(s)$, we get $q_\eta \leq_0
p_1^{[\eta]}$ and a $\ur_\eta \in 2^\omega$, and let $p_2 \DEFEQ  \bigcup_\eta
q_\eta$; in this way, we get a (fusion) sequence $(p,p_1,p_2,\ldots)$, and let
$q \DEFEQ  \bigcap_k p_k$.
\end{proof}

\begin{proof}[Proof of Lemma~\ref{lem:LDdestroysSMZ}]
We want to prove~\eqref{eq:my_non_smz}. Assume towards a contradiction that $X$ is an uncountable set in $V$, 
and that $(\n x_k)_{ k \in \omega}$ is a sequence of names for reals in $2^\omega$ and $p 
  \in \bL_{\bar D}$ such that 
\begin{equation}\label{eq:towards_smz_contra}
   p \forces X \subseteq  \bigcap_{m \in \omega} \bigcup_{k \geq m} [\n x_k \on \n\ell_k].
\end{equation}
Let $s \in \omega^{<\omega}$ be the stem of $p$.

By Lemma~\ref{lem:second_technical}, we can fix a pure extension $q \leq_0 p$ and a family $(\ur_\eta)_{\eta \in q,\, \eta \supseteq s} \subseteq 2^\omega$ such that for each $\eta \in q$ above the stem $s$ and every $n \in \omega$, condition~\eqref{eq:second_technical} holds. 

Since $X$ is (in $V$ and) uncountable, we can find a real $x^* \in X$ which is
different from each real in the countable family $(\ur_\eta)_{\eta \in q,\, \eta
\supseteq s}$; more specifically, we can pick a family of natural numbers
$(n_\eta)_{\eta \in q,\, \eta \supseteq s}$ such that $x^* \on n_\eta \neq \ur_\eta
\on n_\eta$ for any $\eta$.

We can now find $r\le_0 q$ such that:
\begin{itemize}
  \item For all $\eta\in r$ above $s$  and all $i\in \suc_r(\eta)$
          we have $i > n_\eta$. 
  \item For all  $\eta\in r$ above $s$  and all $i\in \suc_r(\eta)$
         we have $r^{[\eta^\frown i]} \forces  \n x_{|\eta|} \on n_\eta = 
	                                             \tau_\eta\on n_\eta \not= x^*\on n_\eta$.
\end{itemize}

So for all $\eta\in r$ above $s$ we have, writing $k$ for $|\eta|$, that 
$r^{[\eta^\frown i]} $ forces $x^*\notin  [ \n x_k \on n_\eta] \supseteq
 [\n x_k \on \ell_k ] $.
We conclude that 
$r$ forces $x^* \notin \bigcup_{k \ge |s|} [\n x_k \on \ell_k] $, 
contradicting 
    \eqref{eq:towards_smz_contra}. 
\end{proof}

\begin{Cor}\label{cor:LDdestroysSMZ}
  Let $(t_k)_{k\in \omega}$ be a dense subset of  $2^{\omega}$.
  
  Let $\bar D$ be a system of ultrafilters, and let $\bar{\n\ell}$ be the
  $\bL_{\bar D}$-name for the ultralaver real.  Then the set $$ \n H\DEFEQ 
  \bigcap_{m\in\omega} \bigcup _{k\ge m} [  t_k \on {\n \ell_k}] $$
  is forced to be a
  comeager set with the property that $\n H$ does not contain any
  translate of any old uncountable set. 
\end{Cor}

Pawlikowski's theorem~\ref{thm:pawlikowski} gives us:
\begin{Cor}\label{cor:absolutepositive}
  There is a canonical name $F$ for a closed null set such that
  $X+F$ is positive for all uncountable $X$ in~$V$.   
  
  In particular, no uncountable ground model set is smz in the
  ultralaver extension.
\end{Cor}

\subsection{Thin sets and strong measure zero}
\label{sec:thin}

For the notion of ``(very) thin'' set, we use an increasing 
function $\tomek(k) $ (the function we use will be described in
Corollary~\ref{cor:tomek}).   We will 
assume that $\bar\ell^*=(\ell^*_k)_{k\in\omega}$ is an
increasing sequence of natural numbers with $\ell^*_{k+1} \gg \tomek(k)$. 
  (We will later use a subsequence 
 of the ultralaver real~$\bar\ell$ as~$\bar\ell^*$, see Lemma~\ref{lem:subsequence}).

\begin{Def}\label{def:thin}
  For $X \subseteq 2^\omega$ and $k\in
  \omega$ we write $X\on [\ell^*_k,\ell^*_{k+1}) $ for the set
  $\{x\on   [\ell^*_k,\ell^*_{k+1}) : x\in X\}$.  We say that
  \begin{itemize}
    \item $X \subseteq 2^\omega$ is 
      \qemph{very thin with respect to $\bar
      \ell^*$ and~$\tomek$}, \  if there are infinitely many $k$ with $|X\on
          [\ell^*_k,\ell^*_{k+1})|\le \tomek(k) $.
    \item $X\subseteq 2^\omega$ is \qemph{thin with respect to $\bar \ell^*$ and~$\tomek$}, \  if $X$ is the union
      of countably many very thin sets.
  \end{itemize}
\end{Def}

Note that the family of thin sets  is a $\sigma$-ideal, while the family of
very thin sets is not even an ideal. Also, every very thin set is covered by a
closed very thin (in particular nowhere dense) set.  In particular, every thin
set is meager and the ideal of thin sets is a proper ideal.

\begin{Lem}\label{lem:subsequence}
  Let $\tomek$ be an increasing function. 
  Let $\bar\ell$ be an increasing sequence of natural numbers.
  We define 
  a subsequence $\bar\ell^*$ of $\bar\ell$ in the following way:
  $\ell^*_k=\ell_{n_k}$ where $n_{k+1}-n_k=\tomek({k})\cdot 2^{\ell^*_k}$.
  \\ Then we get: If $X$ is thin with respect to $\bar\ell^*$ and $\tomek$,
  then $X$ is smz with respect to~$\bar\ell$.
\end{Lem}

\begin{proof}
  Assume that $X=\bigcup_{i\in\omega} Y_i$, each $Y_i$ very thin
  with respect to~$\bar\ell^*$ and $\tomek$.
  Let $(X_j)_{j\in \omega}$ be an enumeration of $\{Y_i:i\in \omega\}$
  where each $Y_i$ appears infinitely often.   So $X \subseteq 
   \bigcap_{m\in \omega} \bigcup_{j\ge m}  X_j$.

  By induction on~$j\in\omega$, we find for all $j>0$  some $k_j>k_{j-1}$
  such that 
  \[
    |X_j\on [\ell^*_{k_j},\ell^*_{k_j+1}) |\leq \tomek({k_j})
    \quad\text{hence}\quad
    |X_j\on [0,\ell^*_{k_j+1}) |\leq \tomek({k_j})\cdot 2^{\ell^*_{k_j}}
     = n_{k_j+1}-n_{k_j}.
  \]
	So we can enumerate $X_j\on [0,\ell^*_{k_j+1}) $ as $(s_i)_{n_{k_j}\leq
i<n_{k_{j}+1}}$.  Hence $X_j$ is a subset of $\bigcup_{n_{k_j}\leq
i<n_{k_{j}+1}} [s_i]$;
  and each $s_i $ has length $\ell^*_{k_j+1}\geq \ell_i$,
  since $\ell^*_{k_j+1}=\ell_{n_{k_j+1}}$ and $i<n_{k_j+1}$.  
  This implies 
  \[ X \subseteq 
   \bigcap_{m\in \omega} \bigcup_{j\ge m}  X_j   \subseteq 
   \bigcap_{m\in \omega} \bigcup_{i\ge m}  [s_i]. \]
  Hence $X$ is smz with respect to~$\bar\ell$.
\end{proof}

Lemma~\ref{lem:LDdestroysSMZ} 
and Lemma~\ref{lem:subsequence} yield: 

\begin{Cor}\label{cor:LDnotthin}
  Let $\tomek$ be an increasing function. 
  Let $\bar D$ be  a system
  of ultrafilters, and $\n{\bar\ell}$ the name for the ultralaver real.
  Let $\n{\bar\ell}^*$ be constructed from $\tomek$ and $\n{\bar\ell}$ as in
  Lemma~\ref{lem:subsequence}.  
  \\
  Then $\bL_{\bar D}$ forces that for every uncountable~$X\subseteq 2^\omega$:
  \begin{itemize}
     \item $X$ is not smz with respect to~$\n{\bar \ell}$. 
     \item $X$ is not thin with respect
          to~$\n{\bar\ell}^*$ and~$\tomek$.\label{item:LDnotthin}
  \end{itemize}
\end{Cor}

\subsection{Ultralaver and preservation of Lebesgue positivity}\label{ss:ultralaverpositivity}

It is well known that both Laver forcing and random forcing preserve
Lebesgue positivity; in fact they satisfy a stronger  property that is preserved
under countable support iterations. 
(So in particular, a countable support iteration of Laver
and random also preserves positivity.)

Ultralaver forcing $\bL_{\bar D}$ will in general not preserve
positivity.  Indeed, if all ultrafilters $D_s$ are equal to the same
ultrafilter $D^*$,
 then the range $L\DEFEQ  \{\ell_0, \ell_1, \ldots \} \subseteq \omega $
of the ultralaver real $\bar \ell$ will diagonalize  $D^*$, so every
ground model real $x\in 2^\omega$ (viewed as a subset of $\omega$)  will either
almost contain $L$ or be almost disjoint to $L$, which implies that
the set $2^\omega\cap V$
of old reals is covered by a null  set in the
extension. 
However, later in this paper it will become clear that if we choose the
ultrafilters $D_s$ in a sufficiently generic way, then many old positive sets
will stay positive.
More specifically, in this section we will show
(Lemma~\ref{lem:extendLDtopreserverandom}): If $\bar D^M$  is an ultrafilter
system in a countable model $M$ and $r$ a random real over $M$,
then we can find an extension $\bar D$ such that $\bL_{\bar D}$ forces that
$r$ remains random over  $M[H^M]$
(where $H^M$ denotes the $\bL_{\bar D}$-name for the restriction of the
 $\bL_{\bar D}$-generic filter $H$  to $\bL_{\bar D^M}\cap M$).
Additionally, some ``side conditions''
are met, which are necessary to preserve the property in forcing iterations.

In Section~\ref{subsec:almostCS} we will see how to use this property to
preserve randoms in limits.

The setup we use for preservation of randomness is basically the notation of
``Case A''  preservation introduced in~\cite[Ch.XVIII]{MR1623206}, see also
\cite{MR1234283,MR2214624} or the textbook~\cite[6.1.B]{MR1350295}:

\begin{Def}\label{def:nullset}
We write $\CLOPEN$ for the collection of clopen sets on $2^\omega$.  
  We say that the function $Z:\omega\to \CLOPEN$
  is a code for  a null set, if
  the  measure of $Z(n)$ is at most  $ 2^{-n}$ for each~$n\in \omega $. 
  
  For such a code $Z$, 
  the set $\nullset(Z)$ coded by $Z$ is
  \[
  \nullset(Z)\DEFEQ  \bigcap_n \bigcup_{k\ge n} Z(k).
  \]
\end{Def}

The set $\nullset(Z)$ obviously is a null set, and it is well known
that every null set is contained in such a set $\nullset(Z)$.

\begin{Def}\label{def:sqsubset}
  For a real $r$ and any code $Z$, we define $Z \sqsubset_n  
  r$ by:
  \[
  (\forall k\geq n) \ r\notin Z(k).
  \]
  We write $Z \sqsubset r$ if $Z \sqsubset_n r$ holds for
  some~$n$; 
  i.e., if  $r\notin \nullset(Z)$.
\end{Def}
For later reference, we record the following trivial fact: 
\proofclaim{eq:sq.n}{
$p \forces  \n Z \sqsubset r$ iff
  there is a name
  $\n n $ for an element of $\omega$ 
  such that  $p\forces  \n Z \sqsubset_{\n n} r$. 
}

Let $P$ be a forcing notion, and $\n Z$ a $P$-name of a code for a null set.  An
interpretation of $\n Z$ below $p$ is some code $Z^*$ such that there is a
sequence $p=p_0\geq p_1\geq p_2\geq \dots$ such that $p_m$ forces $\n Z \on m=
Z^*\on m$. Usually we  demand (which allows a simpler proof of the
preservation theorem at limit stages) that the sequence $(p_0,p_1,\dots)$ is
inconsistent, i.e., $p$ forces that there is an $m$ such that $p_m\notin G$.
Note that whenever $P$ adds a new $\omega$-sequence of ordinals, we can find
such an interpretation for any~$\n Z$.

If $\n{\bar Z}=(\n Z_1,\ldots, \n Z_m)$ is a tuple of names of codes for null sets, then 
an interpretation of $\bar{\n Z}$ below $p$ is some tuple  $(Z_1^*,\ldots, Z_m^*)$ such that there is a
single sequence $p=p_0\geq p_1\geq p_2\geq \dots$ interpreting each $\n Z_i$ as $Z_i^*$.

We now turn to preservation of Lebesgue positivity:

\begin{Def}  \label{def:random.random.random}
  \begin{enumerate}
  \item 
    A forcing notion $P$ \emph{preserves Borel outer measure}, if $P$
    forces $\Leb^*(A^V)=\Leb(A^{V[G_P]})$ for every code $A$ for a Borel
    set. ($\Leb^*$ denotes the outer Lebesgue measure, and for a
    Borel code $A$ and a set-theoretic universe~$V$, $A^V$ denotes the
    Borel set coded by $A$ in~$V$.)
  \item
    $P$ \emph{strongly preserves randoms}, if the following holds: Let
    $N\esm H(\chi^*)$ be countable for a sufficiently large regular cardinal
    $\chi^*$,
    let  $P,p, \bar {\n Z} = (\n Z_1,\ldots, \n Z_m)\in N$,
    let $p\in P$ and  let $r$  be  random
    over~$N$. Assume that in~$N$, $\bar Z^* $ is an interpretation of $\n
    {\bar Z}$, and assume $Z_i^*\sqsubset_{k_i} r$ for each~$i$. Then there is an $N$-generic
    $q\le p$ forcing that $r$ is still random over~$N[G]$ and
    moreover, $\n Z_i\sqsubset_{k_i} r$ for each~$i$.
    (In particular, $P$ has to be proper.)
  \item Assume that $P$ is absolutely definable.  $P$ \emph{strongly
    preserves randoms over countable models} if (2) holds for all
    countable (transitive\footnote{Later we will introduce
    ord-transitive models, and it is easy to see that it does not make
    any difference whether we demand transitive or not; this can be seen
    using a transitive collapse.}) models $N$ of~ZFC$^*$.
  \end{enumerate}
\end{Def}
It is easy to see that these properties are increasing in strength.
(Of course (3)$\Rightarrow$(2) works only if  ZFC$^*$ is  satisfied in~$H(\chi^*)$.)  

In~\cite{MR2155272} it is shown that (1) implies (3), provided that $P$ is nep
(``non-elementary proper'', 
i.e., nicely definable and proper with respect to countable models).  In
particular, every Suslin ccc forcing notion such as random forcing, and also
many tree forcing notions including Laver forcing, are nep.  However $\bL_{\bar
D}$ is not nicely definable in this sense, as its definition uses ultrafilters
as parameters.

\begin{Lem}\label{lem:random.laver}
  Both Laver forcing and random forcing strongly preserve randoms over
  countable models.
\end{Lem}
\begin{proof}
  For random forcing, this is easy and well known (see, e.g.,
  \cite[6.3.12]{MR1350295}).
  
  For Laver forcing: By the above, it is enough to show (1).  
  This was done by
  Woodin (unpublished) and Judah-Shelah~\cite{MR1071305}. A nicer proof
  (including a variant of (2)) is given by Pawlikowski~\cite{MR1367136}.
\end{proof}

Ultralaver will generally not preserve Lebesgue positivity, let alone
randomness. However, we get the following  ``local'' variant of strong
preservation of randoms (which will be used in the preservation
theorem~\ref{lem:iterate.random}).
The rest of this section will be devoted to the proof of the following lemma.

\begin{Lem}\label{lem:extendLDtopreserverandom}
  Assume that $M$ is a countable model, 
  $\bar D^M$ an
  ultrafilter system in $M$ and 
  $r$ a random real over $M$. Then there is (in $V$) an
  ultrafilter system $\bar D$ extending%
  \footnote{This implies, by
    Lemma~\ref{lem:LDMcomplete}, that the $\bL_{\bar D}$-generic
    filter~$G$ induces an $\bL_{\bar D^M}$-generic filter over~$M$,
    which we call~$G^M$.}  
  $\bar D^M$, such that the following holds:
  \\
  \textbf{If}
  \begin{itemize}
    \item $p\in \bL_{\bar D^M}$,
    \item in $M$, $\n {\bar Z} = ( \n Z_1, \ldots , \n Z_m) $
     is a sequence of  $\bL_{\bar D^M}$-names for
      codes for  null sets,\footnote{Recall that $\nullset(\n Z)= \bigcap_n
      \bigcup_{k\ge n} \n Z(k)$ is a null set in the
      extension.} and $Z_1^*,\dots , Z_m^*$ are interpretations under~$p$,
    witnessed by a sequence $(p_n)_{n\in \omega}$ with
    strictly increasing\footnote{It is enough to assume that the lengths of the 
    stems  diverge to infinity;  any thin enough subsequence
    will then have strictly increasing stems and will still
    interpret each $\n Z_i$ as $Z_i^*$.}  stems,
    \item $Z^*_i \sqsubset_{k_i} r$ for $i=1,\dots,  m$,
  \end{itemize}
  \textbf{then} there is a $q\leq p$ in $\bL_{\bar D}$ forcing that
  \begin{itemize}
    \item $r$ is random over $M[G^M]$,
    \item $\n Z_i \sqsubset_{k_i} r$ for $i=1,\dots, m$.
  \end{itemize}
\end {Lem}

For the proof of this lemma, we will use the following concepts: 
\begin{Def} 
  Let $p\subseteq \omega^{< \omega} $ be a tree. 
  A \qemph{front name below $p$}
  is  
   a function\footnote{Instead of $\CLOPEN$ 
   we may also consider other ranges of front names, 
   such as the class of all ordinals, or the set $\omega$.}
   $h:F\to \CLOPEN$, where $F\subseteq p$ is a front (a set that
  meets every branch of~$p$ in a unique point).   (For notational simplicity 
  we also allow $h$ to be defined on elements $\notin p$; this way, 
  every front name below $p$ is also a front name below $q$ whenever
  $q\le p$.) 

  If $h$ is a front name and $\bar D$ is any filter system with $p\in \bL_{\bar D}$, 
  we define the corresponding  $\bL_{\bar D}$-name (in the sense of forcing)
  $\n z^h $ by 
      \begin{align}\label{def:n.alpha}
            \n z^h\DEFEQ   \{ ( \check y, p^{[s]}): s\in F,\ y \in h(s)\}.
      \end{align}
     (This does not depend on the $\bar D$ we use, since we set
     $\check  y\DEFEQ  \{(\check x, \omega^{<\omega} ): x \in y \}$.)

  Up to forced equality, the name $\n z^h$ is characterized
  by the fact that   $p ^{[s]} $ forces (in any ${\bL_{\bar D}}$) that 
  $  \n  z ^h  = h(s)$, for every $s$ in the domain of $h$.
\end{Def}
Note that the same object~$h$ can be viewed as a front name below $p$ with respect
to different forcings $\bL_{\bar D_1}$, $ \bL_{\bar D_2}$, as long 
as $p\in \bL_{\bar D_1}\cap  \bL_{\bar D_2}$.

\begin{Def} 
  Let $p \subseteq \omega^{<\omega}$ be a tree.
  A \qemph{continuous name below $p$}
  is either of the following: 
  \begin{itemize} 
  \item An $\omega$-sequence of front names below $p$. 
  \item A $\subseteq$-increasing function $g:p\to \CLOPEN^{<\omega}$ 
        such that $\lim_{n\to \infty } \lh(g(c\on n)) =\infty$
	for every branch $c\in [p]$. 
  \end{itemize} 
	For each $n$, the set of minimal elements 
          in $\{ s\in p:  \lh(g(s)) > n \}$ is a front, so each 
	  continuous name in the second  sense naturally 
	  defines a name 
	  in the first sense, and conversely.
  Being a continuous name below $p$ does not involve the notion of $\forces$
  nor does it depend on the filter system~$\bar D$.

  If $g$ is a continuous name and $\bar D$ is any filter system, 
  we can again 
  define the corresponding  $\bL_{\bar D}$-name $\n  Z^g $
  (in the sense of forcing);
   we leave a formal definition of $\n Z^g$ to the reader
  and content ourselves with this characterization: 
      \begin{align}\label{def:z.g}
	(\forall s\in p): p^{[s]} \forces_{\bL_{\bar D}}  g(s) \subseteq \n Z^g .
      \end{align}
\end{Def}

  Note that a continuous name below $p$ naturally corresponds
  to a  continuous function $F:[p] \to  \CLOPEN^\omega$, and $ \n Z^g$ is forced
  (by~$p$) to be the value of $F$ at the generic real $\n {\bar \ell}$.

\begin{Lem}\label{lem:pure}
  $\bL_{\bar D}$ has the following ``pure decision properties'':
  \begin{enumerate}
  \item\label{item:pure.one}   Whenever ${\n y}$ is a name for an 
    element of $\CLOPEN$, 
    $p\in \bL_{\bar D}$, then there is a pure extension $p_1\pure  p$
    such that $\n{y} = \n  z^h $ (is forced)
    for a front name $h$ below~$p_1$.
  \item\label{item:pure.omega} Whenever ${\n  Y }$ is a name for a sequence of
    elements of $\CLOPEN$, $p\in \bL_{\bar D}$, then there is a pure extension
    $q\pure  p$ such that ${\n Y } = \n Z ^g$ (is forced)  for some 
     continuous name $g$ below~$q$.
  \item\label{item:pure.finite} (This is Lemma~\ref{lem:pure.finite}.)
    If $A$ is a finite set, $\n \alpha$ a name, $p\in
    \bL_{\bar D}$, and $p$ forces $\n \alpha\in A$, then there is
    $\beta\in A$ and a pure extension $q\pure  p $ such that $q\forces
    \n \alpha=\beta$.
  \end{enumerate}
\end{Lem}
\begin{proof} 
  Let $p\in \bL_{\bar D}$, $s_0\DEFEQ \stem(p)$, $\n y$ a name for an
  element of $\CLOPEN$.
  
  We call $t\in p$ a ``good node in $p$'' if $\n y$ is a front name
  below~$p^{[t]}$ (more formally:  forced to be equal to 
  $\n z^h$ for a front name $h$).
  We can find $p_1\pure p$
  such that for all $t\in p_1$ above $s_0$:  If there is $q\pure  p_1^{[t]}$
  such that  $t$ is good in~$q$, then  $t$ is already good in~$p_1$.

  We claim that $s_0$ is now good (in~$p_1$).  Note that for any bad
  node $s$ the set $\{\,t\in \suc_{p_1}(s): \ t \text{ bad}\,\}$ is
  in~$D_s^+$.  Hence, if $s_0$ is bad, we can inductively construct
  $p_2\pure  p_1$ such that all nodes of $p_2$ are bad nodes in~$p_1$.
  Now let $q\le p_2$ decide $\n y$, $s\DEFEQ \stem(q)$.
  Then $q \pure  p_1^{[s]}$, so $s$ is good in~$p_1$, contradiction.
  This finishes the proof of  (\ref{item:pure.one}).
  
  To prove (\ref{item:pure.omega}), we first construct $p_1$ as in (\ref{item:pure.one}) with respect to $\n
  y_0$.  This gives a front $F_1\subseteq p_1$ deciding $\n
  y_0$. Above each node in $F_1$ we now repeat the construction
  from (\ref{item:pure.one}) with respect to $\n y_1$, yielding $p_2$,  etc.
  Finally, $q\DEFEQ  \bigcap_ n p_n$.

  To prove (\ref{item:pure.finite}): Similar to (\ref{item:pure.one}), we can
   find $p_1\pure  p$ such that 
  for each $t\in p_1$: If there is a pure extension of $p_1^{[t]}$
  deciding $\n\alpha$, then $p_1^{[t]}$ decides $\n \alpha$; in
  this case we again call $t$ good.  Since there are only finitely many
  possibilities for the value of $\n \alpha$, any bad node $t$ has
  $D_t^+$ many bad successors.  So if the stem of $p_1$ is bad, we
  can again reach a contradiction as in (\ref{item:pure.one}).
\end{proof}

\begin{Cor}\label{cor:obda.continuous}
Let  $\bar D$ be a filter system, and let $G\subseteq \bL_{\bar D}$
be generic.  Then every $Y \in \CLOPEN^\omega$ in $V[G]$
is the evaluation of a continuous name $\n Z^g$ by $G$.
\end{Cor}
\begin{proof} In $V$, fix a $p\in \bL_{\bar D}$ and a name $\n Y $ for an 
element of $ \CLOPEN^\omega$.
We can find $q\le_0 p$
and a continuous name $g$ below $q$ such that $q \forces \n Y = \n Z^g$.
\end{proof}

We will need the following modification of the concept of ``continuous
names''. 

\begin{Def} 
  Let $p \subseteq \omega^{<\omega}$ be a tree,  $b\in [p]$ a branch.
  An \qemph{almost continuous name below~$p$ (with respect to~$b$)}
   is 
  a $\subseteq$-increasing function $g:p\to \CLOPEN^{<\omega}$ 
        such that $\lim_{n\to \infty } \lh(g(c\on n)) =\infty$
	for every branch $c\in [p]$, except possibly for $c=b$.
\end{Def}
Note that ``except possibly for $c=b$'' is the only difference between this definition and the definition of a continuous name.  

Since for any $\bar D$ it is forced\footnote{
       This follows from our assumption 
       that all our filters contain the Fr\'echet filter.} 
 that the generic real (for
$\bL_{\bar D}$) is not equal to the exceptional branch $b$, we again get 
a name $\n Z^g$ of a function in $\CLOPEN^\omega$ satisfying: 
  \[ (\forall s\in p): p^{[s]} \forces_{\bL_{\bar D}}  g(s) \subseteq \n Z^g. \]
  An almost continuous name naturally corresponds to a 
   continuous function $F$ from $[p] \setminus \{b\}$ into
        $\CLOPEN^\omega$.

  Note that being an almost
  continuous name is a very simple combinatorial
  property of $g$ which does not depend on $\bar D$, nor does it
  involve the notion $\forces$.
  Thus, the same function $g$
  can be viewed as an almost continuous name for two different
  forcing notions $\bL_{\bar D_1}$, $\bL_{\bar D_2}$ simultaneously.

\begin{Lem} \label{lem:nicefy}
  Let $\bar D$ be a system of filters
  (not necessarily ultrafilters). 
  
  Assume that $\bar p = (p_n)_{n\in \omega}$ witnesses that $Y^*$
  is an interpretation of~$\n Y$, and that the lengths of the stems of the $p_n$
  are strictly increasing.\footnote{It is easy to see that for every
    $\bL_{\bar D}$-name $\n Y$ we can find such $\bar p$ and~$Y^*$:
    First find $\bar p$ which interprets both $\n Y$ and $\bar{\n\ell}$, 
    and then thin out to get a strictly increasing sequence of stems.}
  Then there exists a sequence $\bar q = (q_n)_{n\in \omega}$ such
  that
  \begin{enumerate}
  \item $q_0\ge q_1\ge \cdots $. 
  \item $q_n\le p_n$ for all~$n$.
  \item $\bar q$ also interprets $\n Y $ as~$Y^*$.
    (This follows from the previous two statements.)
  \item $\n Y$ is almost continuous below~$q_0$,  i.e., there is 
  an almost continuous name $g$ such that $q_0$ forces
  $\n  Y = \n Z^g $.
  \item $\n Y$ is almost continuous below~$q_n$, for all~$n$.
    (This follows from the previous statement.)
  \end{enumerate}
\end{Lem} 
\begin{proof} 
  Let $b$ be the branch described by the stems of the conditions
  $p_n$:  \[b\DEFEQ  \{ s: (\exists n)\, s \subseteq \stem(p_n)\}.\]
  
  We now construct 
  a condition~$q_0$. 
  For every $s\in b$ satisfying $\stem(p_n) \subseteq s \subsetneq
  \stem(p_{n+1})$ we set $\suc _{q_0}(s) = \suc_{p_n}(s)$, and for all
  $t\in \suc_{q_0}(s)$ except for the one in~$b$ we let $q_0^{[t]}
  \pure  p_n^{[t]} $ be such that $\n Y$ is continuous below
  $q_0^{[t]}$.  We can do this by Lemma~\ref{lem:pure}(\ref{item:pure.omega}).
  
  Now we set 
  \[ q_n\DEFEQ   p_n \cap q_0 = q_0^{[\stem(p_n)]} \le p_n. \]
  This takes care of~(1) and~(2). Now we show~(4):
  Any branch $c$ of $q_0$ not equal to $b$ must contain a
  node $s^\frown k\notin b$ with $s\in b$, so $c$ is a branch in
  $q_0^{[s^\frown k]}$, below which $\n Y $ was continuous. 
\end{proof}
The following lemmas and corollaries 
are the motivation for considering continuous and 
almost continuous names.  

\begin{Lem}  Let $\bar D$ be a system of filters
  (not necessarily ultrafilters). 
  Let $p\in \bL_{\bar D}$, let $b$ be a branch, and let
  $g:p\to \CLOPEN^{<\omega}$ be an 
  almost continuous name
  below~$p$ with respect to~$b$; write
  $\n Z^g$ for the associated $\bL_{\bar D}$-name. 
  
  Let $r\in 2^ \omega$ be a real, $n_0\in \omega$.  
  Then the following are equivalent:
  \begin{enumerate}
  \item $p\forces_{\bL_{\bar D}} r \notin \bigcup_{n\ge n_0}   \n Z^g(n)$,
    i.e., $ \n Z^g \sqsubset_{n_0} r$. 
  \item For all $n\ge n_0$ and for all $s\in p $ for which
    $g(s)$ has length $>n$ we have $r \notin g(s)(n)$.
  \end{enumerate}
\end{Lem}
Note that (2) does not mention the notion $\forces$  and 
does not depend on $\bar D$.

\begin{proof} 
  $\lnot$(2) $\Rightarrow$ $\lnot$(1):
  Assume that there is $s\in p $ for which $g(s)= (C_0,\ldots,
  C_n, \ldots, C_k)$ and $r\in C_n$.  Then $p^{[s]}$ forces that the
  generic sequence $\n Z^g = ( \n Z(0), \n Z(1), \ldots)$ starts with
  $C_0,\ldots, C_n$, so $p^{[s]}$ forces $r\in \n Z^g(n)$.
  
  $\lnot$(1) $\Rightarrow$ $\lnot$(2): Assume that $p$ does not force
  $r \notin \bigcup_{n\ge n_0} \n Z^g(n)$.  So there is a condition
  $q\le p$ and some $n\ge n_0$ such that $q \forces r\in \n Z^g(n)$.  By
  increasing the stem of~$q$, if necessary, we may assume that
  $s\DEFEQ \stem(q)$ is not on $b$ (the ``exceptional'' branch), and
  that $g(s)$ has already length~$>n$. Let $C_n\DEFEQ  g(s)(n)$ be the
  $n$-th entry of~$g(s)$.  So $p^{[s]}$ already forces $\n Z^g(n)  =
  C_n$; now $q^{[s]}\le p^{[s]}$, and $q^{[s]}$ forces the
  following statements: $r\in \n Z^g(n) $, $\n Z^g(n) = C_n$. Hence $r\in
  C_n$, so (2) fails.
\end{proof}

\begin{Cor}\label{cor:z.absolute}
  Let $\bar D_1$ and $\bar D_2$ be systems of filters, and assume that
  $p$ is in $\bL_{\bar D_1} \cap \bL_{\bar D_2}$. Let $g:p \to
  \CLOPEN^{<\omega}$ be an  
  almost continuous name of a sequence of clopen sets, and 
  let $\n Z^g_1$ and  $\n Z^g_2$ be the associated $\bL_{\bar D_1}$-name
  and $\bL_{\bar D_2}$-name, respectively.
  
  Then for any real $r$  and $n\in \omega$ we have
  \[  p \forces_{\bL_{\bar D_1}} \n Z^g_1 \sqsubset_n  r
  \ \ \Leftrightarrow \ \
  p \forces_{\bL_{\bar D_2}} \n Z^g_2 \sqsubset_n  r.
  \]
\end{Cor}
(We will use this corollary for the special case that $\bL_{\bar D_1}$
is an ultralaver forcing, and $\bL_{\bar D_2}$ is Laver forcing.) 

\begin{Lem}
  Let $\bar D_1$ and $\bar D_2$ be systems of filters, and assume that
  $p$ is in $\bL_{\bar D_1} \cap \bL_{\bar D_2}$. Let $g:p \to
  \CLOPEN^{<\omega}$ be a continuous 
     name of a
  sequence  of clopen sets, let $F \subseteq p$ be a front 
  and let $h:F\to \omega$ be a front name. 
  Again we will write $\n Z^g_1, \n Z^g_2$ for the associated names
    of codes for  null sets, 
   and we will write $\n n_1$ and $\n n_2$ for the associated
  $\bL_{\bar D_1}$- and 
  $\bL_{\bar D_2}$-names, respectively, of natural numbers. 
  
  Then for any real $r$ we have: 
    \[p \forces_{\bL_{\bar D_1}} \n Z^g_1 \sqsubset_{\n n_1} r
      \ \ \Leftrightarrow \ \
     p \forces_{\bL_{\bar D_2}} \n Z^g_2 \sqsubset_{\n n_2} r.\]
\end{Lem}
\begin{proof} 
    Assume $p \forces_{\bL_{\bar D_1}} \n Z^g_1 \sqsubset_{\n n_1} r$. 
    So for each $s\in F$  we have: 
  $p^{[s]}\forces_{\bL_{\bar D_1}} \n Z^g_1 \sqsubset_{h(s) } r$.
  By
  Corollary~\ref{cor:z.absolute}, we also have 
  $p^{[s]}\forces_{\bL_{\bar D_2}} \n Z^g_2 \sqsubset_{h(s)} r$.
   So also 
  $p^{[s]}\forces_{\bL_{\bar D_2}} \n Z^g_2 \sqsubset_{\n n_2} r$ for each $s\in F$.
  Hence 
  $p\forces_{\bL_{\bar D_2}} \n Z^g_2 \sqsubset_{\n n_2} r$.
\end{proof}
\begin{Cor}\label{cor:stays.random}
  Assume $q\in \bL$ forces in Laver forcing that 
  $ \n Z^{g_k} \sqsubset r$ for $k=1,2,\ldots$,
  where each $g_k$  is a continuous name of a code for a
  null set. 
  Then there is a Laver condition $q'\pure  q$ such that for all
  filter systems $\bar D$ we have:
  \begin{quote} If $q'\in \bL_{\bar D}$, then $q'$
  forces (in ultralaver forcing ${\bL_{\bar D}}$) that $
    \n Z^{g_k} \sqsubset r$ for all $k$.
  \end{quote} 
\end{Cor}
\begin{proof} By \eqref{eq:sq.n} we can find a sequence $(\n n _k)_{k=1}^\infty$ of
  $\bL$-names such that $q\forces  \n Z^{g_k} \sqsubset_{\n n_k}  r$ for
  each $k$. By Lemma~\ref{lem:pure}(\ref{item:pure.omega}) 
  we can find $q'\pure  q$  such that this sequence is continuous
  below $q'$.  Since each $\n n_k$ is now a front name below $q'$, we
  can apply the previous lemma.
\end{proof}

\begin{Lem}\label{lem:continuous.is.enough}
 Let $M$ be a countable model, $r\in 2^\omega$, $\bar D^M\in M$
an ultrafilter system, $\bar D $ a  filter system extending $\bar D^M$, $q\in \bL_{\bar D}$. 
For any $V$-generic filter $G\subseteq \bL_{\bar D}$ we write 
$G^M$ for the ($M$-generic, by Lemma~\ref{lem:LDMcomplete}) filter on 
$\bL_{\bar D^M}$. 

The following are equivalent: 
\begin{enumerate}
  \item $q\forces  _{\bL_{\bar D}} r $ is random over $M[G^M]$.
  \item For all names $\n Z\in M$ of codes for null sets: $q\forces_{\bL_{\bar D}}    \n Z \sqsubset r $.
  \item For all continuous names $g\in M$: 
  $q\forces_{\bL_{\bar D}}    \n Z^g  \sqsubset r $.
\end{enumerate}
\end{Lem} 
\begin{proof} (1)$\Leftrightarrow$(2) holds because every null set is contained in a set of the form $\nullset(Z)$, for some code $Z$. 

(2)$\Leftrightarrow$(3): Every code for a null set in $M[G^M]$ 
is equal to~$\n Z^g[G^M]$, for some $g\in M$, by
Corollary~\ref{cor:obda.continuous}.
\end{proof} 

The following lemma may be folklore. Nevertheless, we prove it for the convenience
of the reader.

\begin{Lem} \label{lem:random.over.mprime}
  Let $r $ be random over a  countable model $M$ and $A\in M$.  Then there is a
  countable model $M'\supseteq M$
  such that $A$ is
  countable in~$M'$, but $r$ is still random over~$M'$.
\end{Lem} 
\begin{proof}
  \def\namematrix{  
    \xymatrix@C=15mm{ M  \ar[r]^C  \ar[d]_{B_1} &
      M^C \ar[d]^{\n B_2} \\
      M^{B_1} \ar[r]_{\n P = C*\n B_2/ B_1} &
      M^ {C*\n B_2} \\
    }
  }
  \def\modelmatrix{
    \xymatrix@C=15mm{ M  \ar[r]^J  \ar[d]_{r} &
      M[J] \ar[d]^{K} \\
      M[r] \ar[r]_H &
      M[r][H] \\
    }
  }
  We will need the following  forcing notions, all defined in $M$: 
  \[\namematrix \]
  \begin{itemize}
  \item Let $C$ be the forcing that collapses the cardinality of~$A$ to
    $\omega$ with finite conditions. 
  \item Let $B_1$ be random forcing (trees $T \subseteq 2^{<\omega}$
    of positive measure).
  \item Let $\n B_2$ be the $C$-name of random forcing. 
  \item Let $i:B_1\to C*\n B_2$ be the natural complete embedding $T\mapsto
    (1_C,T)$. 
  \item Let $\n P$ be a $B_1$-name for the forcing $C*\n B_2/i[G_{B_1}]$, the
    quotient of $C*\n B_2$ by the complete subforcing $i[B_1]$.
  \end{itemize}
  
  The random real $r$ is $B_1$-generic over~$M$.  In $M[r]$ we let
  $P\DEFEQ  \n P[r]$.  Now let $H \subseteq P$ be generic over~$M[r]$.
  Then $r*H \subseteq B_1*\n P \simeq C*\n B_2$ induces
  an $M$-generic filter $J \subseteq C$ and an $M[J]$-generic filter
  $K \subseteq \n B_2[J]$; it is easy to check 
   that $K$ interprets the $\n B_2$-name of the
  canonical random real as the given random real~$r$.
  
  Hence $r$ is random over the countable model 
  $M'\DEFEQ M[J]$, and $A$ is countable
  in~$M' $.
  \[ \modelmatrix \]
\end{proof}

\begin{proof}[Proof of Lemma~\ref{lem:extendLDtopreserverandom}]
We will first describe a construction that deals with a single 
triple $ ( \bar p, \bar {\n Z}, \bar Z^ *)$ (where $\bar p$
is a sequence of conditions with strictly increasing stems which interprets
$  \bar {\n Z} $ as $ \bar Z^ *$);  this construction will 
yield a condition $q' =  q'( \bar p, \bar {\n Z}, \bar Z^ *)$.  We 
will then show how to deal with all possible triples.

  So let $p$ be a  condition, and let 
 $\bar p = (p_k)_{k\in \omega}$  be a sequence interpreting 
  $\bar {\n Z}$ as $\bar Z^*$, where the lengths of the 
  stems of $p_n$ are strictly increasing and $p_0=p$.
  It is easy to see that it is enough to deal with a single null set, i.e.,
  $m=1$, and with $k_1=0$.  We write $\n Z$ and $Z^*$ instead
  of $\n Z_1$ and $Z_1^*$.

  Using  Lemma~\ref{lem:nicefy} we may (strengthening the conditions
  in our interpretation) assume (in $M$) that the sequence 
  $(\n Z(k))_{k\in \omega}$ is almost continuous, witnessed by~$g:p\to
  \CLOPEN^{<\omega}$.  
  By Lemma~\ref{lem:random.over.mprime}, we can find a model 
  $M'\supseteq M$ such that $(2^\omega)^M$ is
  countable in~$M'$, but $r$ is still random over~$M'$.

  We now work in~$M'$.  Note that $g$ still defines an almost
  continuous name, which we again call~$\n Z$.    
  
  Each filter in $D_s^M$ is now countably
  generated; let $A_s$ be a pseudo-intersection of $D_s^M$ 
  which additionally satisfies $A_s \subseteq \suc_p(s)$ for all $s\in p$
  above the stem. 
  Let $D'_s$ be the Fr\'echet filter on $A_s$.   Let $p'\in \bL_{\bar D'}$
  be the tree with the same stem as $p$ which satisfies 
  $\suc_{p'}(s)= A_s$ for all $s\in p'$ above the stem.

   By Lemma~\ref{lem:LDMcomplete}, we know that $\bL_{\bar D^M}$ is an
   $M$-complete subforcing of $\bL_{\bar D'}$ (in $M'$ as well as in $V$).  We
   write $G^M$ for 
   the induced filter on $\bL_{\bar D^M}$.

   We now work in $V$.  Note that below the condition $p'$, the forcing
   $\bL_{\bar D'}$ is just Laver forcing $\bL$, and that $p'\le_{\bL} p$.
  Using Lemma~\ref{lem:random.laver} we can find a condition $q\le p'$
  (in Laver forcing $\bL$) such that: 
  \begin{align}
    & q \text{ is $M'$-generic}. \\
    &q\forces_\bL \text{ $r$ is random over $M'[G_{\bL}]$ 
      (hence also over $M[G^M]$)}\label{eq:r.random}.\\
    & \text{Moreover, }q \forces_\bL    \n Z \sqsubset_0 r . \label{eq:z0r}
  \end{align}

  Enumerate all continuous
  $\bL_{\bar D^M}$-names of codes for null sets from $M$ as 
  $\n Z^{g_1},  \n Z^ {g_2},  \ldots $ \ 
  Applying
  Corollary~\ref{cor:stays.random} yields a condition $q'\le q$
  such that
for all filter systems $\bar E$ satisfying $q'\in  \bL_{\bar E}$, 
 we have $q'\forces_{\bL_{\bar E}} \n Z^{g_i} \sqsubset r$ for all $i$.  
Corollary~\ref{cor:z.absolute} and Lemma~\ref{lem:continuous.is.enough} now imply: 
\proofclaim{claim:p.prime}{ For every 
   filter system $\bar E$ satisfying $q'\in  \bL_{\bar E}$, 
 $q' $ forces in ${\bL_{\bar E}}$ that $r$ is random over $M[G^M]$ and
 that $\n Z \sqsubset_0 r$. }
  By thinning out $q'$ we may assume that
  \proofclaim{eq:basdf}{For each $\nu\in \omega^\omega\cap M$ there
  is $k$ such that $\nu\on k\notin q'$. }

  We have now described a construction of $q'= q'(\bar p, \n Z, Z^*)$. 

  Let $(\bar p^n , \n Z^n, Z^{*n})$ enumerate all triples 
   $(\bar p , \n Z, Z^{*})\in M$ where $\bar p$ interprets $\n Z$ as $Z^*$
  (and consists of conditions with strictly increasing stems).   For each 
  $n$  write $\nu^n$ for $\bigcup_k \stem (p^n_k)$,
  the branch determined by the stems of the sequence $\bar p^ n$.   We now
  define by induction a sequence $q^n$ of conditions: 
  \begin{itemize}
     \item $q^0 \DEFEQ  q'(  \bar p^0 , \n Z^0, Z^{*0}) $. 
     \item Given $q^{n-1}$ and $(\bar p^n , \n Z^n, Z^{*n})$, we find $k_0$ such
     that $\nu^n\on k_0 \notin q^0 \cup \cdots \cup  q^{n-1}$
     (using~\eqref{eq:basdf}).  Let $k_1$ be such that
     $\stem(p^n_{k_1})$ has length $>k_0$.   We replace $\bar
     p^n$ by $\bar p'\DEFEQ  (p^n_{k})_{k\ge k_1}$. 
     (Obviously, $\bar p'$
     still interprets $\n Z^n$ as $Z^{*n}$.)
     Now let $q^n\DEFEQ q' (\bar p',  \n Z^n, Z^{*n})$.
  \end{itemize}
  Note that the stem of $q^n$ is at least as long as 
  the stem of $p^n_{k_1}$, and is therefore not in  $q^0 \cup
     \cdots\cup q^{n-1}$, so $\stem(q^i)$ and $\stem(q^j)$ are 
     incompatible for all $i\not=j$.   Therefore we can
     choose for each $s$ an ultrafilter $D_s$ extending
     $D^M_s$ such that
       $\stem(q^i) \subseteq s $ implies $\suc_{q^i}(s)
       \in D_s$. 

  Note that all $q^i$ are in $\bL_{\bar D}$.  Therefore,
   we can use~\eqref{claim:p.prime}. Also, $q^i\le p^i_0$.  
\end{proof}

Below, in Lemma~\ref{lem:iterate.random}, we will prove a preservation theorem
using the following ``local'' variant of ``random preservation'':


\begin{Def}\label{def:locally.random}
  Fix a countable model $M$, a real $r\in 2^\omega$ and a
  forcing notion $Q^M\in M$.
  Let $Q^M$ be an $M$-complete subforcing of $Q$.
  We say that \qemph{$Q$ locally preserves randomness of $r$ over $M$},
  if 
  there is in $M$ a 
	sequence $(D^{Q^M}_n)_{n\in\omega}$ of
	open dense subsets of $Q^M$ such that
  the following holds:\\
	{\bf Assume that   }
  \begin{itemize}
	  \item $M$ thinks that 
		  $\bar p\DEFEQ (p^n)_{n\in\omega}$ interprets 
			$(\n Z_1, \ldots, \n Z_m) $ as $(Z_1^*, \ldots, Z_m^*) $
                         (so each  $\n Z_i$ is a $Q^M$-name of a code for a null set
			and each $Z_i^*$ is a code for a null set, both in $M$);
		\item moreover, each $p^n$ is in $D^{Q^M}_n$ 
	    (we call such a sequence $(p^n)_{n\in\omega}$, or the according interpretation, \qemph{quick});
		\item $r$ is random over $M$;
                \item $Z^*_i \sqsubset_{k_i} r$ for $i=1,\dots, m$.
  \end{itemize}
  {\bf Then}
	 there is  a $q\leq_Q p^0$ forcing that
  \begin{itemize}
     \item 
	$r$ is random over $M[G^M]$;
    \item $\n Z_i \sqsubset_{k_i} r$ for $i=1,\dots, m$.
  \end{itemize}
\end{Def} 
Note that this is trivially satisfied if $r$ is not random over $M$.

For a variant of this 
definition, see Section~\ref{sec:alternativedefs}.

Setting
  $D^{Q^M}_n$ to be the set of conditions with stem of length at least $n$,
Lemma~\ref{lem:extendLDtopreserverandom} gives us:
\begin{Cor}\label{cor:ultralaverlocalpreserving}
  If $Q^M$ is an ultralaver forcing in $M$ and $r$ a real,
                 then there is an ultralaver forcing $Q$ over\footnote{``$Q$ over $Q^M$''
     just means that $Q^M$ is an $M$-complete subforcing of $Q$.} $Q^M$ locally
     preserving randomness of $r$ over~$M$.
\end{Cor}

\section{Janus forcing}\label{sec:janus}

In this section, we define a family of forcing notions that has two faces
(hence the name \qemph{Janus forcing}): Elements of this family may be countable (and therefore
equivalent to Cohen), and they may also be essentially random.

In the rest of the paper, we will use the following properties of Janus forcing
notions $\bJ$. 
(And we will use \emph{only} these properties. So readers who are willing to
take these properties for granted could skip to Section~\ref{sec:iterations}.)

  Throughout the whole paper we fix a function $\tomek:\omega\to \omega$
  given by Corollary~\ref{cor:tomek}.  
  The Janus forcings will depend on a real parameter 
  $\bar \ell^* = (\ell^*_m)_{m\in \omega}\in \omega^\omega$ which grows
  fast with respect to~$\tomek$.  (In our application, $\bar \ell^*$
  will be given by a subsequence of an ultralaver real.) 
  
  The sequence $\bar \ell^*$ and the function $\tomek$ together define 
  a notion of a ``thin set'' (see Definition~\ref{def:thin}).
  
  \begin{enumerate}
  \item \label{item:canonical.null.set} 
    There is a canonical $\bJ$-name for a (code for a) null set~$\n
    Z_\nabla$. 
    \\
    Whenever $X \subseteq 2^\omega$ is not thin, and $\bJ$
    is countable, then $\bJ$ forces that $X$ is not strongly meager,
    witnessed\footnote{in the sense of~\eqref{eq:notsm}} by~$\nullset(\n Z_\nabla)$ (the set we get when we
    evaluate the code $\n Z_\nabla$). 
    Moreover, for any $\bJ$-name~$\n Q$  of a $\sigma$-centered forcing,
     also $\bJ*\n Q$ forces that $X$ is not strongly meager, again
    witnessed by~$\nullset(\n Z_\nabla)$.
    \\
    (This is Lemma~\ref{lem:janusnotmeager}; ``thin'' is defined in Definition~\ref{def:thin}.)
  \item
    Let $M$ be a countable transitive model and $\bJ^M$ a Janus
    forcing in~$M$. Then $\bJ^M$ is a Janus forcing in $V$ as well
    (and of course countable in $V$). (Also note that trivially the forcing
    $\bJ^M$ is an $M$-complete subforcing of itself.)
    \\
    (This is Fact~\ref{fact:janus.ctblunion}.) 
  \item
    Whenever $M$ is a countable transitive model and $\bJ^M$ is a
    Janus forcing in $M$, 
    then
    there is a Janus forcing $\bJ$ 
    such that
    \begin{itemize}
    \item
      $\bJ^M$ is  an $M$-complete subforcing of $\bJ$.
    \item
      $\bJ$  is (in $V$) equivalent to random forcing 
      (actually we just need that $\bJ$ preserves Lebesgue positivity
      in a strong and iterable way).
    \end{itemize}
    (This is Lemma~\ref{lem:janusmayberandom} and Lemma~\ref{lem:janusrandompreservation}.)
   \item
    Moreover, the name $\n Z_\nabla$ referred to
    in~(\ref{item:canonical.null.set}) is so ``canonical'' that 
    it evaluates to the same code in the $\bJ$-generic extension over $V$
    as in the $\bJ^M$-generic extension over $M$.
    \\
    (This is Fact~\ref{fact:Znablaabsolute}.)
  \end{enumerate}

\subsection{Definition of Janus}

A Janus  forcing $\bJ$  will consist of:%
\footnote{We thank Andreas Blass and Jind\v{r}ich Zapletal for their comments
that led to an improved presentation of Janus forcing.}
\begin{itemize} 
\item
  A countable ``core'' (or: backbone) $\nabla$ which is defined in a
  combinatorial way from a parameter~$\bar\ell^*$.  (In our
  application, we will use a Janus forcing immediately after an
  ultralaver forcing, and $\bar\ell^*$ will be a subsequence of the
  ultralaver real.) This core is of course equivalent to Cohen
  forcing.
\item
  Some additional ``stuffing'' $\bJ\setminus \nabla$ (countable\footnote{Also
  the trivial case $\bJ=\nabla $ is allowed.}  or uncountable).  We allow
  great freedom for this, we just
  require that the core $\nabla$ is a ``sufficiently'' complete subforcing (in a
  specific combinatorial sense, see Definition~\ref{def:Janus}(\ref{item:fat})).
\end{itemize}

We will use the following 
combinatorial theorem 
from~\cite{MR2767969}:
\begin{Lem}[{\cite[Theorem 8]{MR2767969}\footnotemark}]
\footnotetext{The theorem
 in~\cite{MR2767969} actually says ``for a sufficiently large
    $I$'', but the proof shows that this should be read as ``for \emph{all}
    sufficiently large $I$''. Also, the quoted theorem only claims that $\MA_I$ will
be nonempty, but for $\varepsilon\le\frac12$ and $|I|> N_{\varepsilon,\delta}$
 it is easy to see that $\MA_I$ cannot be a singleton $\{A\}$: The set $X:= 2^I\setminus A$ has size $\ge 2^{|I|-1}\ge N_{\varepsilon,\delta}$
but satisfies $X+A\not=2^I$, as the constant sequence $\bar 0$ is  not in $X+A$.}
  \label{lem:tomek}
  For every $\varepsilon,\delta>0$ there exists
  $N_{\varepsilon,\delta}\in \omega$ such that for all sufficiently
  large finite sets $I\subseteq \omega$ there is a family
  ${\MA}_I $ with $|\MA_I|\ge 2$  consisting of sets $A \subseteq 2^I$ with
  $\dfrac{|A|}{2^{|I|}} \leq \varepsilon$ such that if $X \subseteq
  2^I$, $|X| \geq N_{\varepsilon,\delta}$ then
  \[
  \frac{|\{ A \in {\MA}_I: X+A=2^I\}|}{|{\MA}_I|} \geq 1-\delta.
  \]
  (Recall that $X+A\DEFEQ  \{x+a: x\in X, a\in A\}$.) 
\end{Lem}

Rephrasing and specializing to $\delta=\frac14$ and 
$\varepsilon = \frac{1}{2^{i}}$ we get:

\begin{Cor}\label{cor:tomek}
  For every 
  $i \in \omega$ there exists $\tomek(i)$ such that for all finite
  sets $I$ 
  with $|I| \geq \tomek(i)$
  there is a nonempty
  family ${\MA}_I$ 
  with $|\MA_I| \geq 2$
  satisfying the following:
  \begin{itemize}
  \item  ${\MA}_I$ consists of sets $A \subseteq 2^I$ with 
  $\dfrac{|A|}{2^{|I|}} \leq \dfrac{1}{2^{i}}$. 
    \item  
    For every $X \subseteq 2^I$ satisfying $|X| \geq \tomek(i) $, the 
    set $\{ A \in {{\MA}_I}: X+A=2^I \}$ has at least $\frac34 |{\MA}_I|$ elements. 
  \end{itemize}
\end{Cor}

\begin{Asm}
	We fix a sufficiently fast increasing sequence
	$\bar\ell^*=(\ell^*_i)_{i\in\omega}$ of natural numbers; more precisely, the sequence $\bar\ell^*$ will be a subsequence of
an ultralaver real $\bar\ell$, defined as in Lemma~\ref{lem:subsequence} using the function $\tomek$ from Corollary~\ref{cor:tomek}. 
Note that in this case $\ell^*_{i+1}-\ell^*_i \geq \tomek(i)$; 
	so we can fix for each $i$ a family $\MA_i \subseteq
	\powerP(2^{\Int_i})$ on
	the interval $\Int_i \DEFEQ  [\ell^*_i,\ell^*_{i+1})$ according to Corollary~\ref{cor:tomek}.
\end{Asm}

\begin{Def}\label{def:Janus.nabla}
  First we define the ``core'' $\nabla= \nabla_{\bar \ell^*}$ of our
  forcing: 
  \[ \nabla = \bigcup_{i\in \omega} \prod_{j<i} \MA_j .\]
  In other words, 
  $\sigma\in \nabla$ iff
  $\sigma= (A_0,\ldots, A_{i-1})$
  for some $i\in \omega$, $A_0\in \MA_0$, \dots, $A_{i-1}\in \MA_{i-1}$.
  We will denote the number $i$ by $\height(\sigma)$. 

  The forcing notion $\nabla$ is ordered by reverse inclusion (i.e., end extension): $\tau \leq \sigma$ 
  if 
  $\tau \supseteq \sigma$. 
\end{Def}

\begin{Def}\label{def:Janus}
Let $\bar \ell^* = (\ell^*_i)_{i\in \omega}$ be as in the assumption above.
We say that $\bJ$ is a Janus forcing based on $\bar \ell^*$ if:
\begin{enumerate}
 \item\label{item:ic} $(\nabla, \supseteq)$ is an incompatibility-preserving
  subforcing of $\bJ$.
 \item\label{item:heightsarepredense} For each $i\in \omega$ the set $\{\sigma\in \nabla:\,
       \height(\sigma)=i\}$ is predense in~$\bJ$. 
			 So in particular, 
			 $\bJ$ adds a 
			 branch through $\nabla$. The union of this branch 
			 is called $\n C^\nabla = (\n C^\nabla_0,\n C^\nabla_1,\n C^\nabla_2,\ldots)$, where $\n C^\nabla_i \subseteq 2^{\Int_i}$ with $\n C^\nabla_i \in \MA_i$. 
     \item\label{item:fat} ``Fatness'':\footnote{This is the crucial combinatorial 
property of Janus forcing.  Actually, \eqref{item:fat}~implies~\eqref{item:heightsarepredense}.}      
     For all $p \in \bJ$ and all real numbers $\varepsilon>0$ 
     there are arbitrarily large $i \in \omega$ such that there is a core condition $\sigma = (A_0,\ldots,A_{i-1})  \in \nabla$ (of length $i$) with
     
    \[ \frac
		          {| \{ A \in \MA_i: \,  \sigma^\frown A \comp_\bJ p\, \}|}
							{| { \MA_i } |}
         \geq 1-\varepsilon.
    \]
    (Recall that $p \comp_\bJ q$ means that $p$ and $q$ are compatible in $\bJ$.)
    \item  \label{item:janus.ccc} $\bJ$ is ccc. 
    \item \label{item:janus.sep}  $\bJ$ is separative.\footnote{Separative is defined on page~\pageref{def:separative}.}
   \item\label{item:janus.hc}  (To simplify some technicalities:) $\bJ \subseteq H(\aleph_1)$.
\end{enumerate}
\end{Def}

We now define 
$\n Z_\nabla$, which will be a canonical $\bJ$-name of  (a code for) a null set. We will use the sequence $\n C^\nabla$ added by $\bJ$ (see Definition~\ref{def:Janus}(\ref{item:heightsarepredense})). 

\begin{Def}\label{def:Znabla}
Each $\n C^\nabla_i$ defines a clopen set $\n Z^\nabla_i =
\{ x \in 2^\omega:\, x \on \Int_i \in \n C^\nabla_i \}$ of measure at most $\frac{1}{2^{i}}$.
The sequence 
$\n Z_\nabla = (\n Z^\nabla_0,\n Z^\nabla_1,\n Z^\nabla_2,\ldots)$
is (a name for) a code for the null set
\[
\nullset(\n Z_\nabla) = \bigcap_{n < \omega} \bigcup_{i \geq n} \n Z^\nabla_i.
\]
\end{Def}

Since $\n C^\nabla$ is defined ``canonically'' (see in particular Definition~\ref{def:Janus}(\ref{item:ic}),(\ref{item:heightsarepredense})), 
and $\n Z^\nabla$ is constructed in an absolute way from $\n C^\nabla$, we get:
\begin{Fact}\label{fact:Znablaabsolute}
  If $\bJ$ is a Janus forcing, $M$ a countable model and $\bJ^M$ a Janus forcing in $M$ which is an $M$-complete subset of $\bJ$, if $H$ is $\bJ$-generic over $V$ and
  $H^M$ the induced $\bJ^M$-generic filter over $M$, then $\n C^\nabla$
  evaluates to the same real in $M[H^M]$ as in $V[H]$, and therefore
  $\n Z^\nabla$ evaluates to the same code (but of course not to the same set of reals).
\end{Fact}

For later reference, we record the following trivial fact:
\begin{Fact}   \label{fact:janus.ctblunion}
  Being a Janus forcing is absolute. In particular, if $V\subseteq W$
  are set theoretical universes and $\bJ$ is a Janus forcing in $V$,
  then  $\bJ$ is a Janus forcing in $W$. In particular, if $M$
  is a countable model in $V$ and $\bJ\in M$ a  Janus forcing in $M$, then $\bJ$ is also a Janus forcing in $V$.
  \\
  Let $(M^n)_{n\in \omega}$ be an increasing sequence of countable
  models, and let $\bJ^n \in M^n$ be Janus forcings.
  Assume that $\bJ^n$ is $M^n$-complete in~$\bJ^{n+1}$.
  Then $\bigcup_n \bJ^n$ is a Janus forcing, and an
  $M^n$-complete extension of $\bJ^n$ for all~$n$.
\end{Fact}

\subsection{Janus and strongly meager}
Carlson~\cite{MR1139474} showed that Cohen reals make every uncountable set $X$
of the ground model not strongly meager in the extension (and that not being
strongly meager is preserved in a subsequent forcing with precaliber~$\al1$).
We show that a {\em countable} Janus forcing $\bJ$ does the same
(for a subsequent forcing that is even $\sigma$-centered, 
not just precaliber~$\al1$).
This sounds
trivial, since any (nontrivial) countable forcing is equivalent to Cohen
forcing anyway. However, we show (and will later use) that the canonical
null set $\n Z_\nabla$ defined above witnesses that $X$ is not strongly meager
(and not just some
null set that we get out of the isomorphism between $\bJ$ and Cohen forcing).
The point is that while $\nabla$ is
not a complete subforcing of~$\bJ$, the 
condition~(\ref{item:fat}) 
of the Definition~\ref{def:Janus} guarantees that Carlson's argument still
works, if we assume that $X$ is non-thin (not just uncountable).
This is enough for us, since by Corollary~\ref{cor:LDnotthin} ultralaver forcing
makes any uncountable set non-thin.

Recall that we fixed the increasing sequence $\bar \ell^* = (\ell^*_i)_{i\in
\omega}$ and~$\tomek$.  In the following, whenever we say ``(very) thin'' we mean
``(very) thin with respect to $\bar \ell^*$ and $\tomek$'' (see Definition~\ref{def:thin}). 
\begin{Lem}\label{lem:janusnotmeager}
  If $X$ is not thin, $\bJ$
  is a countable Janus forcing based on $\bar \ell^*$,
  and
  $\n R$ is a $\bJ$-name
  for a $\sigma$-centered forcing notion,
  then  
  $\bJ*\n R$ forces that $X$ is not strongly meager witnessed by
  the null set $\n Z_\nabla$.
\end{Lem}

\begin{proof}
  Let $\n c$ be a $\bJ$-name for a function $\n c:\n R\to \omega$
  witnessing that $\n R$ is $\sigma$-centered.

  Recall that ``$\n Z_\nabla$ witnesses that $X$ is not strongly meager''
  means that  $X+\n Z_\nabla = 2^\omega$. 
  Assume towards a contradiction that $(p,r) \in \bJ*\n R$ forces that $X+\n Z_\nabla \neq 2^\omega$. Then we can fix a $(\bJ*\n R)$-name $\wit$ such that 
$(p,r) \forces \wit \notin X + \n Z_\nabla$, i.e., 
$(p,r) \forces (\forall  x \in X)\,\, \wit \notin x + \n Z_\nabla$.
 By definition of $\n Z_\nabla$, we get 
\[
(p,r) \forces (\forall  x \in X)\, (\exists n \in \omega)\, (\forall i \geq n) \,\, \wit \on \Int_i \notin x \on \Int_i + \n C^\nabla_i.
\]
  For each $x\in X$ we can find $(p_x,r_x) \leq (p,r)$ and
  natural numbers 
  $n_x \in \omega$ and $m_x \in \omega$ such that 
    $p_x $ forces that $\n c(r_x) = m_x$ 
  and
\[
(p_x,r_x) \forces (\forall i \geq n_x) \,\, \wit \on \Int_i \notin x \on \Int_i + \n C^\nabla_i.
\]

  So $X = \bigcup_{p \in \bJ, m \in \omega, n \in \omega} X_{p,m,n}$, 
  where $X_{p,m,n}$ is the set of all $x$ with~$p_x=p$, $m_x=m$, $n_x=n$.
  (Note that $\bJ$ is countable, so the union is countable.)
  As $X$ is not thin, there is some $p^*, m^*, n^*$ such that $X^*\DEFEQ X_{p^*,m^*,n^*}$ is not very thin. 
  So we get for all  $x\in X^*$:
  \begin{equation}\label{eq:prx}
    (p^*,r_x) \forces (\forall i \geq n^*) \,\, \wit \on \Int_i \notin x \on \Int_i + \n C^\nabla_i.
  \end{equation}
  Since $X^*$ is not very thin, 
  there is some $i_0 \in \omega$ such that for all $i \geq i_0$ 
  \begin{equation}\label{eq:star}
\textrm{the (finite) set } X^* \on \Int_i \textrm{ has more than } \tomek(i)
\textrm{ elements.}
  \end{equation}
Due to the fact that $\bJ$ is a Janus forcing (see Definition~\ref{def:Janus}~\eqref{item:fat}), 
there are arbitrarily large $i \in \omega$ such that there is a core condition $\sigma = (A_0,\ldots,A_{i-1}) \in \nabla$ with 
\begin{equation}  \label{eq:sizeS}
\frac{ | \{ A \in \MA_i: \, \sigma^\frown A \comp_{\bJ} p^* \} | }
{ | \MA_i | } 
\geq \frac{2}{3}.
\end{equation}
Fix such an $i$ 
larger than both $i_0$ and $n^*$, and fix a condition $\sigma$ satisfying~\eqref{eq:sizeS}.

We now consider the following two subsets of $\MA_i$: 
\begin{equation}\label{eq:two_sets}
\{ A \in \MA_i: \, \sigma^\frown A \comp_{\bJ} p^* \}
\,\,
\textrm{ and }
\,\,
\{ A \in \MA_i: \, X^* \on \Int_i + A = 2^{\Int_i} \}.
\end{equation}
By~\eqref{eq:sizeS}, the relative measure (in $\MA_i$) of the left one is at least $\frac{2}{3}$; 
due to~\eqref{eq:star} and the definition of $\MA_i$ according to Corollary~\ref{cor:tomek}, the relative measure of the right one is at least $\frac{3}{4}$; 
so the two sets in~\eqref{eq:two_sets} are not disjoint, and we can pick an $A$ belonging to both.

Clearly, $\sigma^\frown A$ forces (in $\bJ$) that 
$\n C^\nabla_i$ is equal to~$A$. Fix $q \in \bJ$ witnessing $\sigma^\frown A \comp_{\bJ} p^*$. Then 
\begin{equation}\label{eq:wo_for_contradiction}
q \forces_\bJ X^* \on \Int_i + \n C^\nabla_i = X^* \on \Int_i + A = 2^{\Int_i}.
\end{equation}

Since $p^*$ forces that for each $x \in X^*$ the color $\n c(r_x) = m^*$, we can find an $r^*$ 
which is (forced by $q \leq p^*$ to be) a lower bound of the \emph{finite} set $\{r_x : \, x \in X^{**} \}$, where $X^{**} \subseteq X^*$ is any finite set with $X^{**} \on \Int_i = X^* \on \Int_i$.

By~\eqref{eq:prx}, 
\[
(q,r^*) \forces \wit \on \Int_i \notin X^{**} \on \Int_i + \n C^\nabla_i = X^* \on \Int_i + \n C^\nabla_i,
\]
contradicting~\eqref{eq:wo_for_contradiction}.
\end{proof}

Recall that by Corollary~\ref{cor:LDnotthin}, every uncountable set $X$ in $V$ will
not be thin in the $\bL_{\bar D}$-extension.  Hence we get: 

\begin{Cor}\label{cor:ultraplusjanus}
Let $X$ be uncountable. If $\bL_{\bar D}$ is any ultralaver forcing 
adding an ultralaver real $\bar \ell$, and $\bar \ell^*$ is defined from
$\bar \ell$ as in Lemma~\ref{lem:subsequence}, and if  $\n\bJ$ is a 
countable Janus forcing based on $\bar \ell^*$, $\n Q$ is any $\sigma$-centered
forcing, then $\bL_{\bar D}*\n\bJ*\n Q$ forces that $X$ is not strongly meager.
\end{Cor}

\subsection{Janus forcing and preservation of Lebesgue positivity}

We show that every Janus forcing in a countable model $M$ can be extended to
locally preserve a given  random real over $M$. (We showed the same for ultralaver
forcing in Section~\ref{ss:ultralaverpositivity}.)

We start by proving that every countable Janus forcing can be embedded into a
Janus forcing which is equivalent to random forcing, preserving the maximality
of countably many maximal antichains.  (In the following lemma, the letter
$M$  is just a label to distinguish $\bJ^M$ from $\bJ$,
 and does not necessarily refer to a model.)

\newcommand{\PPP}{{\bJ^M}}
\newcommand{\QQQ}{\bJ}

\begin{Lem}\label{lem:janusmayberandom}
   Let $\bJ^M$ be a countable Janus
   forcing (based on $\bar \ell^*$) and let
   $\{D_k:\, k\in \omega\}$ be a countable family of open dense subsets
   of $\bJ^M$.
   Then there is a Janus forcing $\bJ$ 
   (based on the same $\bar \ell^*$) such that
   \begin{itemize}
     \item $\bJ^M$ is an incompatibility-preserving subforcing of $\bJ$.
     \item Each $D_k$ is still predense in $\bJ$.
     \item $\bJ$ is forcing equivalent to random forcing. 
   \end{itemize}
\end{Lem}

\begin{proof}
    Without loss of generality assume $D_0=\bJ^M$.
    Recall that $\nabla = \nabla^{\bJ^M}$ was defined in
    Definition~\ref{def:Janus.nabla}.
Note that for each $j$ the set 
      $\{  \sigma\in \nabla: \, \height(\sigma)=j\}$ is predense in  $\PPP$, 
so the set 
\begin{align}\label{align:E.k}
      E_j\DEFEQ \{ p \in \PPP: \exists  \sigma\in \nabla: \, \height(\sigma)=j, \ p \le \sigma\}
\end{align} 
 is dense open in $\PPP$; hence without loss of generality each $E_j$ appears in our list of $D_k$'s. 

  Let $\{r^n:\, n\in \omega\} $ be an enumeration of~$\bJ^M$. 

  
  We now fix $n$ for a while (up to \eqref{def:this.is.random}). 
  We will construct  
    a finitely splitting tree 
    $S^n \subseteq \omega^{<\omega}$ and a family  $(\sigma^n_s,p^n_s,
    \tau^{*n}_s)_{ s\in S^n}$ 
    satisfying the following (suppressing the superscript~$n$):
  \begin{enumerate}[(a)]
  \item $\sigma_s\in \nabla$,  $\sigma_{\langle\rangle} = \langle\rangle$,
               $s\subseteq t$ implies $\sigma_s\subseteq \sigma_t$,
    and $s\incomp_{S^n} t$ implies $\sigma_s\incomp_\nabla \sigma_t$.
    \\ (So in particular the set $\{\sigma_t:\, t\in\suc_{S^n}(s)\}$
    is a (finite) antichain above $\sigma_s$ in~$\nabla$.)
  \item $p_s\in \PPP$,   $p_{\langle\rangle} = r^n$;
   if $s\subseteq t$ then
    $p_t\leq_\PPP p_s$ (hence  $p_t\leq r^n$); 
         $s\incomp_{S^n} t$ implies $p_s\incomp_\PPP p_t$.
    \item $p_s\leq_\PPP \sigma_s$.
  \item  $\sigma_s \subseteq \tau^*_s \in \nabla$,   and
     $\{\sigma_t:\, t\in\suc_{S^n}(s)\}$
    is  the set of all $ \tau \in \suc_\nabla(\tau^*_s)$ which are compatible
    with $p_s$. 
    \item  The set   $\{\sigma_t:\, t\in\suc_{S^n}(s)\}$
    is a subset of  $ \suc_\nabla(\tau^*_s)$
     of relative size at least
    $1-\frac 1 {\lh(s)+10}$. 
    \label{item:size.a}
  \item Each $s\in {S^n}$ has at least 2 successors (in ${S^n}$).
  \item If $k=\lh(s)$, then $p_s\in D_k$ (and therefore also
    in all $D_l$ for $l<k$).
  \end{enumerate}
  Set $\sigma_{\langle\rangle}=\langle\rangle$ and $p_{\langle\rangle}=r^n$.
  Given $s,\sigma_s$ and~$p_s$,  we construct $\suc_{S^n}(s)$ and
  $(\sigma_t,p_t)_{t\in\suc_{S^n}(s)}$:
  We apply fatness~\ref{def:Janus}(\ref{item:fat}) to $p_s$ with
  $\varepsilon=\frac 1 {\lh(s)+10}$. So we get some $\tau_s^*\in\nabla$
  of height bigger than the height of $\sigma_s$ such that 
  the set $B$ of elements of $\suc_\nabla(\tau_s^*)$ which are compatible
  with $p_s$ has relative size at least $1-\varepsilon$.
  Since $p_s\leq_\PPP \sigma_s$ we get that $\tau^*_s$ is compatible 
  with (and therefore stronger than) $\sigma_s$.
  Enumerate $B$ as $\{\tau_0,\dots,\tau_{l-1}\}$.
  Set $\suc_{S^n}(s)=\{s^\frown i:\, i<l\}$ and $\sigma_{s^\frown i}=\tau_i$.
  For $t\in\suc_{S^n}(s)$, choose $p_t\in \PPP$ stronger than both $\sigma_t$ and $p_s$ 
  (which is obviously possible since $\sigma_t$ and $p_s$ are compatible),
  and moreover $p_t\in D_{\lh(t)}$.
  This concludes the construction of the family 
     $(\sigma^n_s,p^n_s, \tau^{*n}_s)_{ s\in S^n}$.

  So $(S^n,\subseteq)$ is a finitely splitting nonempty
  tree of height $\omega$ with no maximal nodes and no isolated
  branches.
  $[S^n]$ is the (compact) set of branches of~$S^n$.
  The closed subsets of $[S^n]$ are exactly the sets of the form~$[T]$, 
  where $T\subseteq S^n$ is a subtree of $S^n$ 
  with no maximal nodes.  $[S^n]$
	carries a natural (``uniform'') probability measure~$\mu_n$,
  which is characterized by
  \[
    \mu_n((S^n)^{[t]}) = \frac{1}{|{\suc_{S^n}(s)}|}\cdot \mu_n((S^n)^{[s]})
  \]
	for all $s\in S^n$ and all $t\in \suc_{S^n}(s)$. (We just write $\mu_n(T)$
  instead of $\mu_n([T])$  to increase readability.)

  We call $T\subseteq S^n$ positive if $\mu_n(T)>0$, and we call $T$
  pruned if $\mu_n(T^{[s]})>0$ for all $s\in T$.
  (Clearly every positive tree $T$ contains a pruned tree $T'$ of the same
  measure, which can be obtained from $T$ by removing all nodes $s$ with
  $\mu_n(T^{[s]})=0$.)

  Let $T\subseteq S^n$ be a positive pruned tree and $\varepsilon>0$.
  Then on all but finitely many levels $k$ there is an $s\in T$ such that
  \begin{equation}\label{eq:lebdense}
    \suc_T(s)\subseteq \suc_{S^n}(s)\text{ has relative size }\geq
    1-\varepsilon.
  \end{equation}
  (This follows from Lebesgue's density theorem, or can easily be seen directly:
  Set $C_m=\bigcup_{t\in T,\,\lh(t)=m}{(S^n)}^{[t]}$. Then
  $C_m$ is a decreasing sequence of closed sets, each containing~$[T]$.
  If the claim fails, then $\mu_n(C_{m+1} ))\leq \mu_n(C_m)\cdot (1-\varepsilon)$
  infinitely often; so $\mu_n(T) \le \mu_n(  \bigcap_m C_m ) =0$.)

  It is well known that  the set of positive, pruned subtrees of~$S^n$, ordered
  by inclusion, is forcing equivalent to random forcing (which can
  be defined as the set of positive, pruned subtrees of $2^{<\omega}$).

  We have now constructed $S^n$ for all $n$.   Define
\begin{align}\label{def:this.is.random}
         \QQQ = \PPP \cup \bigcup_n \, \bigl\{\, (n,T) : \,  T \subseteq S^n 
         \mbox{ is a positive pruned tree}\,\bigr \}
\end{align}
  with the following partial order: 
  \begin{itemize}
    \item The order on $\QQQ$ extends the order on~$\PPP$.
    \item $(n',T')\le(n,T)$ if $n=n'$ and $T' \subseteq T$. 
    \item For $p\in \PPP$:  $(n,T) \le p$ if there is a $k$ such that 
      $p^n_t\leq p$ for all $t\in T$ of length~$k$.
      (Note that this will then be true for all bigger $k$ as well.)
    \item $p \le (n,T)$ never holds (for $p\in \PPP$). 
  \end{itemize}

  The lemma now easily follows from the following properties: 
  \begin{enumerate}
    \item The order on $\QQQ$ is transitive. 
    \item $\PPP$ is an incompatibility-preserving subforcing of  $\QQQ$.
      \\
      In particular, $\QQQ$ satisfies item~\eqref{item:ic} 
      of Definition~\ref{def:Janus} of Janus forcing.
    \item
      For all $k$: the set $\{(n,T^{[t]}):\, t\in T,\ \lh(t)=k\}$
      is a (finite) predense antichain below~$(n,T)$.
    \item\label{item:compat} $(n,T^{[t]})$ is stronger than $p^n_t$ for each $t\in T$
      (witnessed, e.g., by $k=\lh(t)$).
      Of course, $(n,T^{[t]})$ is stronger than $(n,T)$ as well.
    \item
      Since  $p^n_t\in D_k$ for $k=\lh(t)$, this implies that each 
      $D_k$ is predense below each 
      $(n,S^n)$ and therefore in~$\QQQ$.
      \\
      Also, since each set $E_j$ appeared in our list of open dense subsets  
      (see \eqref{align:E.k}), the set 
      $\{\sigma\in \nabla: \, \height(\sigma)=j\}$ is still  
      predense in $\QQQ$, i.e.,
      item~\eqref{item:heightsarepredense}
      of the Definition~\ref{def:Janus} of Janus forcing is satisfied.
    \item The condition $(n,S^n)$ is stronger than~$r^n$, so 
      $\{(n,S^n):n\in \omega \}$ is predense in $\QQQ$ and 
      $\QQQ\setminus \PPP$ is dense in~$\QQQ$.
      \\
		  Below each $(n,S^n)$, the forcing $\QQQ$ is isomorphic to random
      forcing. 
      \\
      Therefore, $\QQQ$ itself is forcing equivalent to random
      forcing. (In fact, 
      the complete Boolean algebra generated by $\QQQ$ is
      isomorphic to the standard random algebra, Borel sets modulo null sets.)
      This proves in particular  that $\bJ$ is ccc, i.e., satisfies 
       property \ref{def:Janus}(\ref{item:janus.ccc}). 
  \item 
     It is easy (but not even necessary) to check that $\QQQ$ is separative, i.e., 
      property~\ref{def:Janus}(\ref{item:janus.sep}).  In any case,
    we could replace $\le_\QQQ$ by $\le^*_\QQQ$, thus making $\QQQ$ separative without changing
    $\le_\PPP$, since $\PPP$ was already separative. 
    \item Property \ref{def:Janus}(\ref{item:janus.hc}), i.e., $\bJ\in H(\aleph_1)$,
      is obvious. 
    \item \label{item:the.last}
      The remaining item
      of the definition
      of Janus forcing, fatness~\ref{def:Janus}(\ref{item:fat}),
      is satisfied.\\
      I.e., given $(n,T)\in \QQQ$ and $\varepsilon>0$ there is 
      an arbitrarily high $\tau^*\in\nabla$ such that the relative
      size of the set $\{\tau\in\suc_\nabla(\tau^*):\, \tau\comp (n,T)\}$
      is at least $1-\varepsilon$.  (We will show $\ge (1-\varepsilon)^2$ instead, 
      to simplify the notation.)
  \end{enumerate}
  We show~(\ref{item:the.last}):  Given $(n,T)\in \QQQ$ and $\varepsilon>0$,
  we use~\eqref{eq:lebdense} to get an arbitrarily high
  $s\in T$ such that $\suc_T(s)$ is of relative size
  $\geq 1-\varepsilon$ in~$\suc_{S^n}(s)$. We may choose $s$ of length $>\frac 1 \varepsilon$.
  We claim that $\tau^*_s$ is as required:
  \begin{itemize}
    \item Let
    $B\DEFEQ
    \{\sigma_t:\, t\in\suc_{S^n}(s)\}
	   $. 
    Note that $B =
           \{\tau\in\suc_\nabla(\tau^*_s):\, \tau\comp p_s\}
    $.

      $B$ has relative size $\ge 1-\frac{1}{\lh(s)}\ge 1-\varepsilon$
      in $\suc_\nabla(\tau^*_s)$ 
      (according to property~(\ref{item:size.a}) of $S^n$). 
    \item $C\DEFEQ \{\sigma_t:\, t\in\suc_T(s)\}$ is a subset of $B$
      of relative size $\ge 1-\varepsilon$ according to our choice of~$s$.
    \item So $C$ is of relative size $(1-\varepsilon)^2$
	    in~$\suc_\nabla(\tau^*_s)$.
   \item Each $\sigma_t\in C$ is compatible with~$(n,T)$, as $(n,T^{[t]})
         \le p_t \le \sigma_t$ (see~(\ref{item:compat})). \qedhere
  \end{itemize}
\end{proof}


So in particular if $\bJ^M$ is a Janus forcing in a countable model $M$, then
we can extend it to a Janus forcing $\bJ$ which is in fact random forcing.
Since random forcing strongly preserves randoms over countable models (see
Lemma~\ref{lem:random.laver}), it is not surprising that we get local
preservation of randoms for Janus forcing, i.e., the analoga of
Lemma~\ref{lem:extendLDtopreserverandom} and
Corollary~\ref{cor:ultralaverlocalpreserving}. (Still, some additional argument
is needed, since the fact that~$\bJ$ (which is now random forcing) ``strongly preserves
randoms'' just means that a random real $r$ over $M$ is preserved with respect
to random forcing in $M$, not with respect to $\bJ^M$.)



\begin{Lem}\label{lem:janusrandompreservation}
  If $\bJ^M$ is a Janus forcing in a countable model $M$ and
  $r$ a random real over $M$, then there is a Janus forcing $\bJ$ 
  such that $\bJ^M$ is an $M$-complete subforcing of $\bJ$ and the
  following holds:
  \\
  \textbf{If}
  \begin{itemize}
    \item $p\in \bJ^M$,
    \item in $M$, $\n {\bar Z} = ( \n Z_1, \ldots , \n Z_m) $
     is a sequence of  $\bJ^M$-names for
      codes for  null sets, 
       and $Z_1^*,\dots , Z_m^*$ are interpretations under~$p$,
    witnessed by a sequence $(p_n)_{n\in \omega}$,
    \item $Z^*_i \sqsubset_{k_i} r$ for $i=1,\dots,  m$,
  \end{itemize}
  \textbf{then} there is a $q\leq p$ in $\bJ$ forcing that
  \begin{itemize}
    \item $r$ is random over $M[H^M]$,
    \item $\n Z_i \sqsubset_{k_i} r$ for $i=1,\dots, m$.
  \end{itemize}
\end{Lem}
\begin{Rem}
  In the version for ultralaver forcings, i.e.,
  Lemma~\ref{lem:extendLDtopreserverandom}, we had to assume that the stems of
  the witnessing sequence are strictly increasing. In the Janus version, we do
  not have any requirement of that kind.
\end{Rem}

\begin{proof}
  Let $\mathcal D$ be the set of dense subsets of $\bJ^M$ in $M$. 
  According to Lemma~\ref{lem:random.over.mprime}, we 
  can first find some countable $M'$ such that $r$ is still random over $M'$
  and such that in $M'$ both $\bJ^M$ and  $\mathcal D$
  are countable. According to Fact~\ref{fact:janus.ctblunion}, 
  $\bJ^M$ is a (countable) Janus forcing in $M'$, so we can apply
  Lemma~\ref{lem:janusmayberandom} to the set $\mathcal D$
  to construct a Janus forcing $\bJ^{M'}$ 
  which is equivalent to random forcing
  such that (from the point of $V$) $\bJ^{M}\lessdot_M \bJ^{M'}$.
  In $V$, let\footnote{More precisely: 
 Densely embed $\bJ^{M'}$ into 
   (Borel/null)$^{M'}$,  the complete Boolean algebra associated with 
   random forcing in $M'$, and let $\bJ:=$ (Borel/null)$^V$.  Using the 
   embedding, $\bJ^{M'}$ can now be viewed as an $M'$-complete
   subset of $\bJ$.}
 $\bJ$ be random forcing. 
  $\bJ^{M'}$ is an $M'$-complete subforcing of $\bJ$
  and therefore  $\bJ^{M}\lessdot_M \bJ$. Moreover,
  as was noted in Lemma~\ref{lem:random.laver}, we even
  know that  random forcing  strongly preserves randoms over $M'$
  (see Definition~\ref{def:locally.random}).   To show that 
  $\bJ$ is indeed a Janus forcing, we have to check the
   fatness condition~\ref{def:Janus}(\ref{item:fat});
   this follows easily from $\Pi^1_1$-absoluteness (recall that
    incompatibility of random conditions is Borel). 


  So assume that (in $M$) the sequence $(p_n)_{n\in\omega}$
  of $\bJ^M$-conditions interprets $\n {\bar Z}$ as $\bar Z^*$.
  In $M'$, $\bJ^M$-names can be reinterpreted as $\bJ^{M'}$-names,
  and  the $\bJ^{M'}$-name $\n {\bar Z}$ is interpreted as $\bar Z^*$
  by the same sequence $(p_n)_{n\in\omega}$.
  Let $k_1,\ldots, k_m$ be such that $Z_i^*\sqsubset_{k_i} r$ 
  for $i=1,\ldots, m$. 
  So by strong preservation of randoms, we can in $V$ find some
  $q\leq p_0$ forcing that $r$ is random over $M'[H^{M'}]$ (and therefore
  also over the subset $M[H^M]$), and that $\n Z_i\sqsubset_{k_i} r$ (where
  $\n Z_i$ can be evaluated in
  $M'[H^{M'}]$ or equivalently in $M[H^M]$). 
\end{proof}

So Janus forcing is locally preserving randoms (just as ultralaver forcing):
\begin{Cor}\label{cor:januslocallypreserves}
 If $Q^M$ is a Janus forcing in $M$ and $r$ a real, then there is a
 Janus forcing $Q$ over~$Q^M$ (which is in fact equivalent to random
     forcing) locally preserving randomness of~$r$ over~$M$. 
\end{Cor}

\begin{proof}
  In this case, the notion of ``quick'' interpretations is 
  trivial, i.e., $D^{Q^M}_k  = Q^M$ for all~$k$, and the claim follows from
  the previous lemma.
\end{proof}

\section{Almost finite and almost countable support iterations}\label{sec:iterations}

A main tool to construct the forcing for BC+dBC will be ``partial countable
support iterations'', more particularly ``almost finite support'' and ``almost
countable support'' iterations. A partial countable support iteration is a
forcing iteration $(P_\alpha,Q_\alpha)_{\alpha<\om2}$ such that for each limit
ordinal $\delta$ the forcing notion $P_\delta$ is a subset of the countable
support limit of $(P_\alpha,Q_\alpha)_{\alpha<\delta}$ which satisfies some
natural properties (see Definition~\ref{partial_CS}).

Instead of transitive models, we will use ord-transitive models (which are
transitive when ordinals are considered as urelements). Why do we do that?  We
want to ``approximate'' the generic iteration $\bar\BP$ of length $\omega_2$
with countable models; this can be done more naturally with ord-transitive
models (since obviously countable transitive models only see countable
ordinals). We call such an ord-transitive model a ``candidate'' (provided it
satisfies some nice properties, see Definition~\ref{def:candidate}).  A basic
point is that  forcing extensions work naturally with candidates.

In the next few paragraphs (and also in Section~\ref{sec:construction}),
$x=(M^x,\bar P^x)$ will denote a pair such that $M^x$ is a
candidate and $\bar P^x$ is (in $M^x$) a partial countable support iteration; similarly 
we write, e.g., $y= (M^y, \bar P^ y) $ or  $x_n=(M^{x_n},\bar P^{x_n})$.

We will need the following results to prove BC+dBC. (However, as opposed to the
case of the ultralaver and Janus section, the reader will probably have to read
this section to understand the construction in the next section, and not just
the following list of properties.)

  Given $x=(M^x,\bar P^x)$, we can construct 
  by induction on $\alpha$ a
  partial countable support iteration $\bar P = (P_\alpha, Q_\alpha)_{\alpha
    < \om2}$ satisfying:
  \begin{quote}
    There is a canonical $M^x$-complete embedding from $\bar P^x$ to
    $\bar P$.
  \end{quote}
  In this construction, we can use at each stage $\beta$ 
  any desired $Q_\beta$, as long as 
  $P_\beta$ forces that $Q^x_\beta$ is (evaluated as)
  an $M^x[H^x_\beta]$-complete subforcing of~$Q_\beta$ (where
  $H^x_\beta\subseteq P^x_\beta$
  is the $M^x$-generic filter induced by the generic filter~$H_\beta\subseteq P_\beta$).
   \\
   Moreover, we can demand either of 
   the following two additional properties\footnote{The $\sigma$-centered version is central for the proof of dBC; the random preserving version
     for BC.}
of the limit of this iteration~$\bar P$:
   \begin{enumerate}
   \item 
     If all $Q_\beta$ are forced to be $\sigma$-centered, 
     and $Q_\beta$ is trivial for all $\beta\notin M^x$, 
     then $P_{\omega_2}$ is $\sigma$-centered. 
   \item 
	     If $r$ is random over
     $M^x$, and all $Q_\beta$ locally preserve randomness of $r$
     over $M^x[H^x_\beta]$ (see
     Definition~\ref{def:locally.random}), then also $P_{\om2}$
     locally preserves the randomness of $r$.
   \end{enumerate}
   Actually, we need the following variant: Assume that we already 
   have $P_{\alpha_0}$ 
   for some $\alpha_0\in M^x$, and that $P^x_{\alpha_0}$ canonically
   embeds into~$P_{\alpha_0}$, and that
   the respective assumption on $Q_\beta$ holds for all $\beta\ge
   \alpha_0$.  Then we get that $P_{\alpha_0}$ forces that the quotient
   $P_{\omega_2}/P_{\alpha_0}$ satisfies the respective conclusion. 
   
   We also need:\footnote{This will give $\sigma$-closure 
     and $\al2$-cc for the preparatory forcing $\prep$.}
   \begin{enumerate}\setcounter{enumi}{2}  
   \item 
     If instead  of a single $x$ we have a sequence $x_n$
     such that each $P^{x_n}$ canonically (and $M^{x_n}$-completely)
     embeds into~$P^{x_{n+1}}$, then 
     we can find a partial countable support
     iteration $\bar P$ into which all $P^{x_n}$ embed canonically (and we can again use
     any desired $Q_\beta$, assuming that  $Q^{x_n}_\beta$
     is an $M^{x_n}[H^{x_n}_\beta]$-complete subforcing of $Q_\beta$ for all $n\in\omega$).
   \item 
     (A fact that is easy to prove but awkward to formulate.)
     If a $\Delta$-system argument produces two $x_1$, $x_2$ as in Lemma~\ref{lem:prep.is.sigma.preparation}(\ref{item:karotte6}),
     then we can find a partial countable support iteration $\bar P$
     such that $\bar P^{x_i}$ canonically (and  $M^{x_i}$-completely)
     embeds into~$\bar P$ for $i=1,2$.
   \end{enumerate}

\subsection{Ord-transitive models}\label{subsec:ordtrans}

We will use ``ord-transitive'' models, as introduced in~\cite{MR2115943} (see
also the presentation in~\cite{kellnernep}). We briefly summarize the basic
definitions and properties (restricted to the rather simple case needed in this
paper):

\begin{Def}\label{def:candidate} Fix a suitable finite subset ZFC$^*$ of ZFC (that is
  satisfied by $H(\chi^*)$ for sufficiently large regular $\chi^*$).
  \begin{enumerate}
    \item A set $M$ is called a \emph{candidate}, if
      \begin{itemize} 
        \item $M$ is countable,
        \item $(M,\in)$ is a model of ZFC$^*$,
        \item $M$ is ord-absolute: 
          $M \models \alpha\in \ON$ iff $\alpha\in \ON$,
          for all $\alpha\in M$,
        \item $M$ is ord-transitive: if $x\in M\setminus \ON$, then 
          $x\subseteq M$,
        \item $\omega+1\subseteq M$.
        \item ``$\alpha$ is a limit ordinal'' and ``$\alpha=\beta+1$''
          are both
          absolute between $M$ and $V$.

      \end{itemize} 
    \item A candidate $M$ is called \emph{nice}, if
          ``$\alpha$ has countable cofinality'' 
          and ``the countable set $A$ is cofinal in $\alpha$'' both are 
          absolute between $M$ and $V$. (So 
          if $\alpha\in M$ has countable cofinality, then
          $\alpha\cap M$ is cofinal in $\alpha$.)
          Moreover, we assume $\om1\in M$ (which implies 
          $\om1^M=\om1$) and $\om2\in M$ (but we do not require 
          $\om2^M= \om2$). 
    \item 
      Let $P^M$ be a forcing notion in a candidate $M$.  
      (To simplify notation, we can assume without loss of generality that
      $P^M\cap \ON=\emptyset$ (or at least $\subseteq \omega$)
      and that therefore $P^M\subseteq M$ and also $A\subseteq
      M$ whenever $M$ thinks that $A$ is a subset of $P^M$.)
      Recall that a subset $H^M$ of $P^M$ is $M$-generic (or: $P^M$-generic over $M$),
      if $|A\cap H^M|=1$ for all maximal antichains $A$ in $M$.
    \item Let $H^M$ be $P^M$-generic over $M$ and $\n\tau$ a $P^M$-name in $M$.
      We define the evaluation $\n\tau[H^M]^M$ to be 
      $x$ if 
      $M$ thinks that $p\forces_{P^M} \n\tau=\std x$
      for some $p\in H^M$ and $x\in M$ (or equivalently just for $x\in M\cap     \ON$), 
      and $\{\n\sigma[H^M]^M:\, (\n\sigma,p)\in\n\tau,\, 
      p\in H^M\}$ otherwise.
      Abusing notation we write $\n\tau[H^M]$ instead of $\n\tau[H^M]^M$,
      and we write $M[H^M]$ for $\{\n\tau[H^M]:\, \n\tau\text{ is a $P^M$-name
      in }M\}$.
    \item For any set $N$ (typically, an elementary submodel of some $H(\chi)$),
        the ord-collapse $k$ (or $k^N$) is a recursively defined function  with domain $N$: 
       $k(x)=x$ if $x\in  \ON$, and $k(x)=\{k(y):\,
      y\in x\cap N\}$ otherwise.
    \item  We define $\ordclos(\alpha):=\emptyset$ for all ordinals $\alpha$. 
       The ord-transitive closure of a non-ordinal $x$ is defined 
      inductively on the rank:  
      \[
        \ordclos(x)=x\cup\bigcup \{\ordclos(y):y\in x\setminus \ON\}
                   =x\cup\bigcup \{\ordclos(y):y\in x\}.
      \]
      So for $x\notin \ON$, the set 
       $\ordclos(x)$ is the smallest ord-transitive set
      containing $x$ as a subset. 
      HCON is the collection of all sets $x$ such that 
      the ord-transitive closure of $x$ is countable.
      $x$ is in HCON iff $x$ is element of some candidate.
      In particular, all reals and all ordinals are HCON.

      We write HCON$_\alpha$ for the family of all sets $x$ in HCON whose transitive
      closure 
      only contains ordinals $<\alpha$.
  \end{enumerate}
\end{Def}

The following facts can be found in~\cite{MR2115943} or~\cite{kellnernep} (they
can  be proven by rather straightforward, if tedious, inductions on the ranks of
the according objects).
\begin{Fact}\label{fact:hcon}
  \begin{enumerate}
    \item The ord-collapse of a countable elementary submodel of $H(\chi^*)$ is a nice
      candidate.
    \item Unions, intersections etc.\ are generally not absolute for
      candidates. For example, let $x\in M\setminus \ON$. In $M$ we can
      construct a set $y$ such that  $M\models y=\om1\cup\{x\}$. 
      Then $y$ is not an ordinal and therefore a subset of $M$, and
      in particular $y$ is countable and $y\not=\om1\cup \{x\}$.
    \item Let $j:M\to M'$ be the transitive collapse of a candidate $M$,
      and $f:\om1\cap M'\to \ON$ the inverse (restricted to the ordinals).
      Obviously $M'$ is a countable transitive model of ZFC$^*$;
      moreover $M$ is characterized by the pair $(M',f)$
      (we call such a pair a ``labeled transitive model'').
      Note that $f$ satisfies  $f(\alpha+1)=f(\alpha)+1$, 
      $f(\alpha)=\alpha$ for $\alpha\in \omega\cup\{\omega\}$. 
      $M\models (\alpha\text{ is a limit})$ iff $f(\alpha)$ is a limit.
      $M\models \cf(\alpha)=\omega$ iff $\cf(f(\alpha))=\omega$, and 
      in that case $f[\alpha]$ is cofinal in $f(\alpha)$.
      On the other hand, given a transitive countable model $M'$ of ZFC$^*$
      and an $f$ as above, then we can construct a (unique) candidate $M$
      corresponding to $(M',f)$.
    \item All candidates $M$ with $M\cap \ON \subseteq \omega_1$ are hereditarily countable, so their number is at most $2^{\al0}$. Similarly,  the cardinality of HCON$_\alpha$ is at most continuum 
    whenever $\alpha < \omega_2$.
    \item If $M$ is a candidate, and  if $H^M$ is $P^M$-generic over $M$, then $M[H^M]$ is a candidate
      as well and an end-extension
      of $M$ such that $M\cap\ON=M[H^M]\cap \ON$. If $M$ is nice and 
      ($M$ thinks that) $P^M$ is proper, then 
      $M[H^M]$ is nice as well.
    \item Forcing extensions commute with the transitive collapse $j$:
      If $M$ corresponds to $(M',f)$, then
      $H^M\subseteq P^M$ is $P^M$-generic over $M$ iff 
      $H'\DEFEQ {j}[H^M]$ is $P'\DEFEQ j(P^M)$-generic over $M'$,
      and in that case $M[H^M]$ corresponds to $(M'[H'],f)$.
      In particular, the forcing extension  $M[H^M]$ of $M$ satisfies
      the forcing theorem (everything that is forced is true, and
      everything true is forced).
    \item  For elementary submodels, forcing extensions commute
      with ord-collapses:
      Let $N$ be a countable elementary submodel of $H(\chi^*)$,
      $P\in N$, $k:N\to M$ the ord-collapse (so $M$ is a candidate), 
      and let $H$ be $P$-generic
      over $V$.
      Then $H$ is $P$-generic over $N$
      iff $H^M\DEFEQ k[H]$
      is $P^M\DEFEQ k(P)$-generic over $M$;
      and in that case the ord-collapse of $N[H]$
      is $M[H^M]$.
  \end{enumerate}
\end{Fact}

Assume that a nice candidate $M$ thinks that $(\bar P^M,\bar Q^M)$ is a forcing
iteration of length $\om2^V$ (we will usually write $\om2$ for the length of
the iteration, by this we will always mean  $\om2^V$ and not the possibly
different $\om2^M$). In this section, we will construct an iteration $(\bar
P,\bar Q)$ in $V$, also of length $\om2$, such that each $P^M_\alpha$
canonically and $M$-completely embeds into $P_\alpha$ for all $\alpha\in
\om2\cap M$. Once we know (by induction) that $P^M_\alpha$ $M$-completely
embeds into $P_\alpha$, we know that a $P_\alpha$-generic filter $H_\alpha$
induces a $P^M_\alpha$-generic (over $M$) filter which we call $H^M_\alpha$.
Then $M[H^M_\alpha]$ is a candidate, but nice only if $P^M_\alpha$ is proper.
We will not need that $M[H^M_\alpha]$ is nice, actually we will only
investigate sets of reals (or elements of $H(\al1)$) in $M[H^M_\alpha]$, so it
does not make any difference whether we use $M[H^M_\alpha]$ or its transitive
collapse.

\begin{Rem}\label{rem:fine.print}
In the discussion so far we omitted some details regarding the theory ZFC$^*$
(that a candidate has to satisfy). The following ``fine print'' hopefully
absolves us from any liability. (It is entirely irrelevant for the
understanding of the paper.)

We have to guarantee that each $M[H^M_\alpha]$ that we consider satisfies
enough of ZFC to make our arguments work (for example, the definitions and
basic properties of ultralaver and Janus forcings should work).  This turns out
to be easy, since (as usual) we do not need the full power set axiom for these
arguments (just the existence of, say, $\beth_5$). So it is enough that each
$M[H^M_\alpha]$ satisfies some fixed finite subset of ZFC minus power set, which
we call ZFC$^*$.

Of course we can also find a bigger (still finite) set ZFC$^{**}$ that implies:
$\beth_{10}$ exists, and each forcing extension of the universe with a forcing
of size $\le \beth_4$ satisfies ZFC$^*$.  And it is provable (in ZFC) that each
$H(\chi)$ satisfies ZFC$^{**}$ for sufficiently large regular $\chi$. 
 
We define ``candidate'' using the weaker theory ZFC$^*$, and require that nice
candidates satisfy the stronger theory ZFC$^{**}$. This guarantees that all
forcing extensions (by small forcings) of nice candidates will be candidates
(in particular, satisfy enough of ZFC such that our arguments about Janus or
ultralaver forcings work). Also, every ord-collapse of a countable elementary
submodel $N$ of $H(\chi)$ will be a nice candidate. 
\end{Rem}

\subsection{Partial countable support iterations}\label{subsec:partialCS}

We introduce the notion of ``partial countable support limit'': a
subset of the countable support (CS) limit containing the union (i.e.,
the direct limit) and satisfying some natural requirements.

Let us first describe what we mean by ``forcing iteration''.  They have to satisfy the
following requirements:
  \begin{itemize}
  \item A \qemph{topless forcing iteration} $(P_\alpha,Q_\alpha)_{\alpha<\varepsilon}$
    is a sequence of forcing notions $P_\alpha$ and $P_\alpha$-names
    $Q_\alpha$ of quasiorders with a weakest element~$1_{Q_\alpha}$.
    A \qemph{topped iteration} additionally has a final limit~$P_\varepsilon$.
    Each  $ P_\alpha$ is a set of partial functions 
    on $\alpha$ 
    (as, e.g., in~\cite{MR1234283}). More specifically, 
    if $\alpha<\beta\le \varepsilon$ and $p\in P_\beta$, then 
    $p\on\alpha\in P_\alpha$.   Also, 
    $p\on\beta\forces_{P_\beta} p(\beta)\in Q_\beta$ for all $\beta\in\dom(p)$. 
    The order on $P_\beta$ will always be the ``natural'' one: 
    $q\leq p$ iff
    $q\on\alpha$ forces (in $P_\alpha$) that
    $q^\tot(\alpha)\leq p^\tot(\alpha)$ for all $\alpha < \beta$, 
    where $r^\tot(\alpha)=r(\alpha)$
    for all $\alpha\in\dom(r)$ and $1_{Q_\alpha}$ otherwise.   $P_{\alpha+1}$ consists 
    of \emph{all} $p$ with $p\on\alpha\in P_\alpha$ and $p\on \alpha\forces p^\tot(\alpha)\in Q_\alpha$, so it is forcing equivalent to $P_\alpha*Q_\alpha$. 
		\item $P_\alpha \subseteq P_\beta$ whenever $\alpha  < \beta\le
		\varepsilon$. (In particular, the empty condition is an element of each
		$P_\beta$.)
  \item For any $p \in P_\varepsilon$ and any $q \in P_\alpha$ ($\alpha < \varepsilon$) with $q \leq p \on \alpha$, the partial function 
    $q \land p\DEFEQ q\cup p\on[\alpha,\varepsilon)$ is a condition in $P_\varepsilon$ as well (so in particular, $p \on \alpha$ is a reduction of $p$, hence $P_\alpha$ is a complete subforcing of $P_\varepsilon$; and $q\land p$ is the weakest condition
    in $P_\varepsilon$ stronger than both $q$ and $p$).
  \item Abusing notation, we usually just write $\bar P$
    for an iteration (be it topless or topped).
  \item We usually write $H_\beta$ for the
    generic filter on $P_\beta$ (which induces
    $P_\alpha$-generic filters called $H_\alpha$ for $\alpha\le\beta$).
    For topped iterations we call the filter on the final limit sometimes just $H$
    instead of $H_\varepsilon$.
  \end{itemize}

We use the following notation for quotients of iterations:
\begin{itemize}
  \item For $\alpha<\beta$, in the $P_\alpha$-extension
    $V[H_\alpha]$, we let $P_\beta/H_\alpha$ be the set
    of all $p\in P_\beta$ with $p\on\alpha\in H_\alpha$
    (ordered as in~$P_\beta$).   We may occasionally write $P_\beta/P_\alpha$ for the $P_\alpha$-name of 
    $P_\beta/H_\alpha$. 
	\item Since $P_\alpha$	is a complete subforcing of $P_\beta$,
	  this is a quotient with the usual properties, in particular
		$P_\beta$ is equivalent to $P_\alpha*(P_\beta/H_\alpha)$.
\end{itemize}

\begin{Rem}
It is well known that quotients of proper countable support iterations are
naturally equivalent to (names of) countable support iterations. In this paper,
we can restrict our attention to proper forcings, but we do not really have
countable support iterations. It turns out that it is not necessary to
investigate whether our quotients can naturally be seen as iterations of any
kind, so to avoid the subtle problems involved we will not consider the
quotient as an iteration by itself.
\end{Rem}

\begin{Def}\label{def:fullCS}
Let  $\bar P$ be a (topless) iteration of limit length $\varepsilon$.
We define three limits of $\bar P$:
\begin{itemize}
  \item The \qemph{direct limit} is the union of the  $P_\alpha$ (for $\alpha<\varepsilon$).
    So this is the smallest possible limit of the iteration.
  \item The \qemph{inverse limit} consists of \emph{all}  partial functions $p$
    with domain $\subseteq \varepsilon$ such that $p\on\alpha\in P_\alpha$
    for all $\alpha<\varepsilon$.
    This is the largest possible limit of the iteration.
  \item
    The \qemph{full countable support  limit $P^\CS_\varepsilon$} of $\bar P$
    is the inverse limit if $\cf(\varepsilon)=\omega$ and the direct limit otherwise.
\end{itemize}
We say that  $P_\varepsilon$ is a \qemph{partial CS limit}, if 
		$P_\varepsilon$ is a subset of the full CS limit and 
		the sequence $(P_\alpha)_{\alpha\le \varepsilon}$ is a topped iteration. 
		In particular, this means that $P_\varepsilon$ 
		contains the direct limit, 
		and satisfies the following for each $\alpha<\varepsilon$: 
		$P_\varepsilon$ is  closed under $p\mapsto p\on \alpha$, and
		whenever $p\in P_\varepsilon$,   $q\in P_\alpha$, $q\le p\on\alpha$,
		 then also the partial function $q\land p$ is in~$P_\varepsilon$. 
\end{Def}  

So  for a given topless $\bar P$ there is a well-defined inverse, direct and full
CS limit.  If $\cf(\varepsilon)>\omega$, then the direct and the full CS limit
coincide. If $\cf(\varepsilon)=\omega$, 
then the direct limit and the full CS limit (=inverse limit) differ.  Both of them are partial CS limits, 
but there are many more
possibilities for partial CS limits.  By definition, all of them will yield 
iterations. 

Note that the name ``CS limit'' is slightly inappropriate, as  the size of
supports of conditions is not part of the definition. To give a 
more  specific example:
Consider a topped  iteration $\bar P $ of length $\omega+\omega$ where
$P_\omega$ is the direct limit and $P_{\omega+\omega} $ is  the full
CS limit.  Let $p$ be any element of the full CS limit of $\bar P \on \omega$ which
is not in~$P_\omega$; then $p$ is not in $P_{\omega+\omega}$ either.  So not
every countable subset of $\omega+\omega$ can appear as the support of a
condition.

\begin{Def}\label{partial_CS}
A forcing iteration $\bar P$ is called a \qemph{partial CS iteration}, if
\begin{itemize}
  \item every limit is a partial CS limit, and
  \item every $Q_\alpha$ is (forced to be) separative.\footnote{The reason for this requirement is briefly discussed in Section~\ref{sec:alternativedefs}. Separativity, as well as the relations $\leq^*$ and $=^*$, are defined on page~\pageref{def:separative}.}
\end{itemize}
\end{Def}

The following fact can easily be proved by transfinite induction: 

\begin{Fact}\label{fact:eq.eqstar}
Let $\bar P$ be a partial CS iteration. Then for all $\alpha$ the forcing
notion $P_\alpha$ is separative. 
\end{Fact}

{}From now on, all iterations we consider will be partial CS
iterations.  In this paper, we will only be interested in proper
partial CS iterations, but properness is not part of the definition of
partial CS iteration.  (The reader may safely assume that all
iterations are proper.)

Note that separativity of the $Q_\alpha $ implies that all partial CS
iterations satisfy the following (trivially equivalent) properties:
\begin{Fact}\label{fact:suitable.equivalent}
Let $\bar P$ be a topped partial CS iteration of length $\varepsilon$. Then:
\begin{enumerate}
  \item Let $H$ be $ P_\varepsilon$-generic. Then $p\in H$ iff $p\on\alpha\in H_\alpha$ for all $\alpha < \varepsilon$.
  \item For all $q,p \in P_\varepsilon$: If $q\on \alpha \leq^\ast p\on \alpha$ for each $\alpha < \varepsilon$, then $q \leq^\ast p$. 
  \item \label{item:3}
For all $q,p \in P_\varepsilon$: If $q\on \alpha \leq^\ast p\on \alpha$ for each $\alpha < \varepsilon$, then $q \comp p$.
\end{enumerate}
\end{Fact}

We will be concerned with the following situation:

Assume that $M$ is a nice candidate, $\bar P^M$ is (in~$M$) a topped partial CS
iteration of length $\varepsilon$ (a limit ordinal in~$M$), and $\bar P$ is (in
$V$) a topless partial CS iteration of length
$\varepsilon'\DEFEQ \sup(\varepsilon\cap M)$.  (Recall that
``$\cf(\varepsilon)=\omega$'' is absolute between $M$ and $V$, and that
$\cf(\varepsilon)=\omega$ implies $\varepsilon'=\varepsilon$.) Moreover, assume
that we already have a system of $M$-complete coherent\footnote{I.e., they
commute with the restriction maps: $i_\alpha(p \on \alpha) = i_\beta(p) \on
\alpha$ for $\alpha < \beta$ and $p\in P^M_\beta$.} embeddings
$i_\beta:P^M_\beta\to P_\beta$ for $\beta\in \varepsilon'\cap M=\varepsilon\cap M$.
(Recall that any potential partial CS limit of $\bar P$ is a subforcing of
the full CS limit~$P^\CS_{\varepsilon'}$.)
It is easy to see that there is only one possibility for an embedding $j:
P^M_\varepsilon\to P^\CS_{\varepsilon'} $ (in fact, into any potential partial
CS limit of~$\bar P$) that extends the $i_\beta$'s naturally:

\begin{Def}\label{def:canonicalextension}
  For a topped partial CS iteration $\bar P^M$ in $M$ of length  $\varepsilon$
  and a topless one $\bar P$ in $V$ of length $\varepsilon'\DEFEQ \sup(\varepsilon\cap M)$
  together with coherent embeddings $i_\beta$, we define 
  $j: P^M_\varepsilon\to P^\CS_{\varepsilon'} $, the \qemph{canonical extension},
  in the obvious way:   
  Given $p \in P_\varepsilon^M$, take
  the sequence of restrictions to $M$-ordinals, 
  apply the functions $i_\beta$, and let
  $j(p)$ be the union of the resulting coherent sequence.
\end{Def}

We do not claim that  $j: P^M_\varepsilon\to P^\CS_{\varepsilon'} $  is
 $M$-complete.\footnote{\newcounter{myfootnote}\setcounter{myfootnote}{\value{footnote}}
  \label{cs.fs.footnote}
	For example, if $\varepsilon=\varepsilon'=\omega$ and if $P^M_\omega$ is the 
	finite support limit of a nontrivial iteration, then $j:P^M_\omega\to P^\CS_\omega$ is not
	complete: For notational simplicity, assume that all $Q^M_n$ are (forced to be) Boolean
        algebras. In $M$, let $c_n$ be (a $P^M_n$-name for) a nontrivial
	element of $Q^M_n$ (so $\lnot c_n$, the Boolean complement, is also nontrivial).
	Let $p_n$ be the $P^M_n$-condition $(c_0, \ldots, c_{n-1})$, i.e.,
        the truth value of ``$c_m\in H(m)$ for all $m<n$''. Let $q_n$ be the 
        $P^M_{n+1}$-condition $(c_0, \ldots, c_{n-1}, \lnot c_n)$, i.e., the truth
        value of 
	``$n$ is minimal with $c_n\notin H(n)$''.  In $M$, the set $A=\{q_n:\, n\in\omega\}$
	is a maximal antichain in $P^M_\omega$.  Moreover, the sequence
	$(p_n)_{n\in\omega}$ is a decreasing coherent sequence, therefore $i_n(p_n)$
	defines an element $p_\omega$ in $P^\CS_\omega$, which is clearly incompatible
	with all $j(q_n)$, hence $j[A] $ is not maximal.}
In the following, we will construct partial CS limits $P_{\varepsilon'}$ such
that $j:P^M_\varepsilon \to P_{\varepsilon'}$ is $M$-complete. (Obviously, one
requirement for such a limit  is that $j[P^M_\varepsilon]\subseteq
P_{\varepsilon'}$.) We will actually define two versions: The almost FS (``almost 
finite support'')  and the
almost CS (``almost countable support'') limit.

Note that there is only one effect that the ``top'' of $\bar P^M$ (i.e., the
forcing $P^M_\varepsilon$) has on the canonical extension~$j$: It determines the
domain of $j$. In particular it will generally depend on $P^M_\varepsilon$
whether $j$ is complete or not. Apart from that, the value of any given $j(p)$
does not depend on $P^M_\varepsilon$.

Instead of arbitrary systems of embeddings $i_\alpha$, we will only be interested in
``canonical'' ones. 
 We  assume for notational convenience that
$Q^M_\alpha$ is a subset of $Q_\alpha$ (this will naturally be the case in our application anyway). 

\begin{Def}[The canonical embedding]\label{def:canonicalembedding}
  Let $\bar P$ be a partial CS iteration in~$V$
  and $\bar P^M$  a partial CS iteration in~$M$, both topped and of length~$\varepsilon\in M$.
  We construct by induction on $\alpha\in (\varepsilon+1) \cap M$
  the canonical $M$-complete embeddings $i_\alpha:P^M_\alpha\to P_\alpha$.
  More precisely: We try to construct them, but it is possible that the 
  construction fails.  If the construction succeeds, then 
  we say that  \qemph{$\bar P^M$ (canonically) embeds into $\bar P$}, or 
	\qemph{the canonical embeddings work}, or 
  just:
	\qemph{$\bar P$ is over $\bar P^M$}, or \qemph{over $P^M_\varepsilon$}.  
  \begin{itemize}
    \item Let $\alpha=\beta+1$. By induction hypothesis,
      $i_\beta$ is $M$-complete, so a $V$-generic
      filter $H_\beta\subseteq P_\beta$  induces an $M$-generic filter $H^M_\beta 
			\DEFEQ  i_{\beta}^{-1}[H_\beta]\subseteq P^M_\beta$.
      We require that (in the $H_\beta$ extension) the set $Q^M_\beta[H^M_\beta]$
       is an $M[H^M_\beta]$-complete
      subforcing of $Q_\beta[H_\beta]$. In this case, we define $i_\alpha$  in the
      obvious way. 
    \item  For $\alpha$ limit, let $i_\alpha$ be the canonical 
      extension of the family $(i_\beta)_{\beta\in  \alpha\cap M}$. 
         We require that  $P_\alpha$ contains the 
			range of $i_\alpha$, and that $i_\alpha$ is $M$-complete; otherwise the
			construction fails.
      (If $\alpha'\DEFEQ \sup(\alpha\cap M) < \alpha$, then $i_\alpha $
      will actually be an $M$-complete map  into  $P_{\alpha'}$, assuming that the 
      requirement is fulfilled.)
  \end{itemize}
\end{Def}

In this section we try to construct a partial CS iteration $\bar P$
(over a given $\bar P^M$)  satisfying additional properties.

\begin{Rem}
What is the role of $\varepsilon'\DEFEQ  \sup( \varepsilon\cap M)$?  When our
inductive construction  of $\bar P$ arrives at~$P_\varepsilon$ where $\varepsilon'<
\varepsilon$, it would be too late\footnote{%
\label{fn:too.late}
  For example:
	Let $\varepsilon=\om1$ and $\varepsilon'=\om1\cap M$.
	Assume that $P^M_{\om1}$ is (in $M$) a (or: the unique) 
        partial CS limit of a nontrivial iteration.
	Assume that we have a topless iteration $\bar P$ of length $\varepsilon'$
	in $V$ such that the canonical embeddings work for all $\alpha\in\om1\cap M$.
	If we set $P_{\varepsilon'}$ to be the full CS limit, then
	we cannot further extend it to any iteration of length $\om1$ 
	such that the canonical embedding $i_{\om1}$ works: Let
	$p_\alpha$ and $q_\alpha$ be as in footnote~\ref{cs.fs.footnote}.  In  $M$,
	the set  
	$A=\{q_\alpha:\alpha\in\om1\}$ is a maximal antichain, and the sequence
	$(p_\alpha)_{\alpha\in\om1}$ is a decreasing coherent
	sequence.  But in $V$ there is an element  $p_{\varepsilon'}\in P^\CS_{\varepsilon'}$
         with $p_{\varepsilon'}\on \alpha = j( p_\alpha)$ for all $\alpha\in \varepsilon\cap M$. This 
         condition $p_{\varepsilon'}$ is   clearly incompatible
	with all elements of  $j[A] = \{ j(q_\alpha): \alpha \in \varepsilon\cap M\}$. 
        Hence $j[A] $ is not maximal.}
to take care of $M$-completeness of
$i_\varepsilon$ at this stage, even if all $i_\alpha$ work nicely for $\alpha\in
\varepsilon \cap M$. Note that $\varepsilon'<\varepsilon$ implies 
that $\varepsilon$ is uncountable in $M$, and that therefore 
$P^M_\varepsilon = \bigcup_{\alpha\in \varepsilon\cap M} P^M_\alpha$.
So the natural extension $j$ of the embeddings $(i_\alpha)_{\alpha\in
\varepsilon \cap M}$ has range in $ P_{\varepsilon'}$, which will be a complete
subforcing of $P_\varepsilon$. So we have to ensure
$M$-completeness already in the construction of $P_{\varepsilon'}$.
\end{Rem}

For now we just record:
\begin{Lem}\label{lem:wolfgang}
  Assume that we have topped iterations
  $\bar P^M$ (in~$M$) of length $\varepsilon$ and 
  $\bar P$ (in~$V$) of length $\varepsilon'\DEFEQ \sup(\varepsilon\cap M)$,
  and that for all $\alpha\in\varepsilon\cap M$
  the canonical embedding $i_\alpha: P^M_\alpha\to P_\alpha$ works.
  Let $i_\varepsilon: P^M_\varepsilon\to P^\CS_{\varepsilon'}$
  be the canonical extension.
  \begin{enumerate}
    \item\label{item:pathetic099}
			If  $P^M_\varepsilon$ is (in $M$)  a direct limit (which
			is always the case if $\varepsilon$ has uncountable cofinality)
			then $i_\varepsilon$
			(might not work, but at least) has range in $P_{\varepsilon'}$
			and preserves incompatibility.
    \item\label{item:suitable_implies_filter}
      If $i_\varepsilon$ has a range contained in $P_{\varepsilon'}$
      and maps predense sets $D\subseteq P^M_\varepsilon$ in $M$ 
      to predense sets $i_\varepsilon[D]\subseteq P_{\varepsilon'}$,
      then $i_\varepsilon$ preserves incompatibility (and therefore works).
  \end{enumerate}
\end{Lem}

\begin{proof} (1)
  Since $P^M_\varepsilon$ is a direct limit, the canonical extension
  $i_\varepsilon $ has range in $\bigcup_{\alpha<\varepsilon'}
  P_\alpha$, which is subset of any partial CS limit $P_{\varepsilon'}$.
  Incompatibility in $P^M_\varepsilon$ is the same as incompatibility in
  $P^M_\alpha$ for  sufficiently large $\alpha\in \varepsilon\cap M$, so 
  by assumption it is preserved by $i_\alpha$ and hence also by~$i_\varepsilon$.
  
(2)
Fix $p_1,p_2\in P^M_\varepsilon$, and assume that their images are
compatible in $P_{\varepsilon'}$; we have to show that they
are compatible in $P^M_\varepsilon$.  So fix 
a  generic filter $H\subseteq P_{\varepsilon'} $ containing $i_\varepsilon(p_1)$
 and $i_\varepsilon(p_2)$.

In $M$, we define the  following set $D$: 
\[
D \DEFEQ  \{ q \in P^M_\varepsilon: 
(q \leq p_1 \land q \leq p_2) \textrm{ or } 
(\exists \alpha < \varepsilon: q\on \alpha \incomp_{P^M_\alpha} p_1\on \alpha) \textrm{ or }
(\exists \alpha < \varepsilon: q\on \alpha \incomp_{P^M_\alpha} p_2\on \alpha) 
 \}.
\]

Using Fact~\ref{fact:suitable.equivalent}(\ref{item:3}) it is easy to check that $D$ is dense. 
Since $i_\varepsilon$ preserves predensity, 
there is $q\in D$ such that $i_\varepsilon(q)\in H$.   We claim that 
$q$ is stronger than $p_1$ and~$p_2$.   Otherwise we would have without loss
of generality 
 $ q\on \alpha \incomp_{P^M_\alpha} p_1\on \alpha$ for some
$\alpha<\varepsilon$.  But the filter  $H\on \alpha$ contains both $i_\alpha(q\on\alpha)$
and $i_\alpha(p_1\on\alpha)$, contradicting the assumption that $i_\alpha$
preserves incompatibility. 
\end{proof}

\subsection{Almost finite support iterations}

Recall Definition~\ref{def:canonicalextension} (of the canonical extension)
and
the setup that was described there: We have to find a 
subset $P_\ves$ of $P^\CS_\ves$ such that the canonical extension~$j:P^M_\varepsilon\to P_\ves$ is $M$-complete.  

We now define the almost finite support limit. (The direct  limit will in
general not do, as it may not contain the range  $j[P^M_\varepsilon]$.  The
almost finite support limit is the obvious modification of the direct limit,
and it is  the smallest partial CS limit $P _\ves$ such that $j[P^M_\varepsilon]\subseteq
P_\ves$, and it indeed  turns out to be $M$-complete as well.)

\begin{Def}\label{def:almostfs}
  Let $\varepsilon$ be a limit ordinal in~$M$,  and let 
  $\ves\DEFEQ  \sup(\varepsilon\cap M)$.  
  Let $\bar P^M$ be a topped iteration in~$M$ of length $\varepsilon$,
  and let $\bar P$ be a topless iteration in~$V$ of length $\varepsilon'$. 
  Assume that the canonical embeddings $i_\alpha$ work for all $\alpha
  \in \varepsilon \cap M = \varepsilon' \cap M$. Let $i_\varepsilon$
  be the canonical extension.
  We define the \emph{almost finite support limit of $\bar P$ over
  ${\bar P^M}$} (or: almost FS limit) as the following subforcing 
  $P_\ves$ of $P^\CS_\ves$:
   \[P_\ves\DEFEQ   \{\,  
  q \land i_\varepsilon(p) \in P^\CS_\ves :\  p\in P^M_\varepsilon \text{ and } 
      q\in P_\alpha\text{ for some } \alpha\in \varepsilon\cap M\text{ such that }
  q\le _{P_\alpha} i_\alpha(p\on \alpha) \, \}.\]
\end{Def}
Note that for $\cf(\varepsilon)>\omega$, the almost FS limit
is equal to the direct limit, as each $p\in P^M_\varepsilon$ is in fact
in $P^M_\alpha$ for some  $\alpha\in  \varepsilon\cap M$, so $i_\varepsilon(p) = 
i_\alpha(p)\in P_\alpha$.

\begin{Lem}\label{lem:afs.complete}
  Assume that $\bar P$ and $\bar P^M$ are as above and let $P_\ves$ be
  the almost FS limit.  Then $\bar P^\frown P_\ves$ is a partial CS
  iteration, and $i_\varepsilon$ works, i.e., $i_\varepsilon$ is an
  $M$-complete embedding from $P^M_\varepsilon$ to $P_\ves$.  (As
  $P_\ves $ is a complete subforcing of~$P_\varepsilon$, this also
  implies that $i_\varepsilon$ is $M$-complete from $P^M_\varepsilon$
  to $P_\varepsilon$.)
\end{Lem}
\begin{proof}
  It is easy to see that $P_\ves$ is a partial CS limit and contains
  the range $i_\varepsilon[P^M_\varepsilon]$.
    We now show preservation of
    predensity; this implies $M$-completeness by  Lemma~\ref{lem:wolfgang}.

  Let $(p_j)_{j\in J} \in M$ be a maximal antichain in
  $P^M_\varepsilon$. (Since $P^M_\varepsilon$ does not have to be ccc in $M$,
  $J$ can have any cardinality in $M$.) Let $q\land {i_{\varepsilon}}(p)$ be a
	condition in $P_\ves $.  (If $\ves<\varepsilon$, i.e.,  if
$\cf(\varepsilon)>\omega$, then we can choose $p$ to be the empty condition.) 
   Fix $\alpha \in \varepsilon \cap M $ be such
  that $q\in P_\alpha$. 
  Let $H_\alpha$ be $P_\alpha$-generic and contain $q$, 
  so $p\on\alpha$ is in $H^M_\alpha$. 
  Now in $M[H^M_\alpha]$ the set $\{p_j: j\in J, p_j\in
  P_\varepsilon^M/H_\alpha^M\}$ is predense in
  $P^M_\varepsilon/H^M_\alpha$ (since this is forced by the empty condition
  in $P^M_\alpha$). In particular, 
  $p$ is compatible with some $p_j$, witnessed 
  by $p'\le p,p_j$ in $P^M_\varepsilon /H^M_\alpha$.

  We can find $q'\le_{P_\alpha} q$ deciding $j$ and $p'$; since certainly
  $q'\le^*  {i_{\alpha}}(p'\on \alpha) $, we may assume even $\le$ without loss of generality. Now $q'\land {i_{\varepsilon}} (p') \le q\land
  {i_\varepsilon}(p)$ (since $q'\leq q$ and $p'\leq p$), 
  and $q'\land {i_\varepsilon} (p') \le
  {i_\varepsilon}(p_j)$ (since $p'\le p_j$).
\end{proof}

\begin{DefandClaim}\label{lem:kellnertheorem}
  Let $\bar P^M$ be a topped partial CS iteration in $M$
  of length $\varepsilon$.   We can construct by induction on
  $\beta\in\varepsilon+1$ an
  \emph{almost finite support iteration $\bar P$ over $\bar P^M$}
  (or: almost FS iteration) as follows:
  \begin{enumerate}
    \item As induction hypothesis we assume that the canonical   
      embedding  $i_\alpha$
      works for all $\alpha\in \beta \cap M$. (So the notation 
			  $M [H^M_\alpha]$ makes sense.) 
    \item\label{item:qwrrqw} Let $\beta=\alpha+1$. If $\alpha\in M$, then we can use any
      $Q_\alpha$ provided that (it is forced that) $Q^M_\alpha$ is an
      $M[H^M_\alpha]$-complete subforcing of $Q_\alpha$.
      (If $\alpha\notin M$, then there is no restriction on  $Q_\alpha$.)
    \item Let $\beta\in M$ and $\cf(\beta)=\omega$.
      Then $P_\beta$ is the almost FS limit of
      $(P_\alpha,Q_\alpha)_{\alpha<\beta}$ over $P^M_\beta$.
    \item Let $\beta\in M$ and $\cf(\beta)>\omega$.  Then $P_\beta$ is
      again the almost FS limit of
      $(P_\alpha,Q_\alpha)_{\alpha<\beta}$ over $P^M_\beta$ (which
      also happens to be the direct limit).
    \item  For limit ordinals not in $M$, 
       $P_\beta$ is the direct limit.
  \end{enumerate}
\end{DefandClaim}

So the claim includes that the resulting $\bar P$ is a (topped) partial CS
iteration of length $\varepsilon$ over $\bar P^M$ (i.e., the canonical embeddings
$i_\alpha$ work for all $\alpha\in (\varepsilon+1) \cap M$), where we only assume
that the $Q_\alpha$ satisfy the obvious requirement given in~(\ref{item:qwrrqw}).
(Note that we can always find some  suitable $Q_\alpha$ for $\alpha\in M$,
for example we can just take $Q^M_\alpha$ itself.)

\begin{proof}
  We have to show (by induction) that the resulting sequence $\bar P$
  is a partial CS iteration, and that $\bar P^M$ embeds into $\bar P$.
  For successor cases, there is nothing to do. So assume that $\alpha$
  is a limit.  If $P_\alpha$ is a direct limit, it is trivially a
  partial CS limit;  if $P_\alpha$ is an almost FS limit, then the
  easy part of Lemma~\ref{lem:afs.complete} shows that it is a partial
  CS limit.

  So it remains to show that for a limit  $\alpha\in M$, the (naturally defined)
  embedding $i_\alpha:P^M_\alpha\to P_\alpha$ is $M$-complete.
	This was the main claim in  Lemma~\ref{lem:afs.complete}.  
\end{proof}

The following lemma is natural and easy.   
\begin{Lem}
  Assume that we construct an almost FS
  iteration $\bar P$ over $\bar P^M$ where
  each $Q_\alpha$ is (forced to be) ccc.
  Then $P_\varepsilon$ is ccc (and in particular proper).
\end{Lem}

\begin{proof}
  We show that $P_\alpha$ is ccc by induction on
  $\alpha\leq\varepsilon$.  For successors, we use that $Q_\alpha$ is
  ccc.  For $\alpha$ of uncountable cofinality, we know that we took
  the direct limit coboundedly often (and all $P_\beta$ are ccc for
  $\beta<\alpha$), so by a result of Solovay $P_\alpha$ is again ccc.
  For $\alpha$ a limit of countable cofinality not in $M$, just use
  that all $P_\beta$ are ccc for $\beta<\alpha$, and the fact that
  $P_\alpha$ is the direct limit.  This leaves the case that
  $\alpha\in M$ has countable cofinality, i.e., the $P_\alpha$ is the
  almost FS limit. Let $A\subseteq P_\alpha$ be uncountable.  Each
  $a\in A$ has the form $q\land i_\alpha(p)$ for $p\in P^M_\alpha$ and
  $q\in \bigcup_{\gamma<\alpha} P_\gamma$.  We can thin out the set $A$ such that 
 $p$ are the same and all $q$ are in the same $P_\gamma$. So there have to be compatible
  elements in $A$.
\end{proof}
All almost FS iterations that we consider in this paper will satisfy the
countable chain condition (and hence in particular be proper).

We will need a variant of this lemma for $\sigma$-centered forcing  notions. 
\begin{Lem}\label{lem:4.17}
  Assume that we construct an almost FS
  iteration $\bar P$ over $\bar P^M$ 
  where only countably many $Q_\alpha$
  are nontrivial (e.g., only those with $\alpha\in M$) and where
  each $Q_\alpha$ is (forced to be) $\sigma$-centered.
  Then $P_\varepsilon$ is $\sigma$-centered as well.
\end{Lem}

\begin{proof}
  By induction: The direct limit of countably many 
  $\sigma$-centered forcings is $\sigma$-centered,
  as is the almost FS limit of $\sigma$-centered forcings
	(to color $q\land i_\alpha( p)$, use  $p$ itself together with the color of $q$).
\end{proof}

We will actually need two variants of the almost FS construction:
Countably many models $M^n$; and starting the almost FS iteration with some $\alpha_0$.

Firstly, we can construct an almost  FS iteration not just over one iteration
$\bar P^M$, but over an increasing chain of iterations.
 Analogously to Definition~\ref{def:almostfs} and
Lemma~\ref{lem:afs.complete}, we can show:

\begin{Lem}\label{lem:418}
 For each $n\in \omega$, let $M^n$ be a nice candidate, and let  $\bar P^n$ 
	be a topped partial CS iteration in $M^n$ of
        length\footnote{Or only: $\varepsilon\in M^{n_0}$ for some $n_0$.} $\varepsilon\in M^0$ of countable
  cofinality,
  such that $M^m\in M^n $  and 
 $M^{n}$ thinks that $\bar P^m$ canonically embeds into $\bar P^{n}$,
  for all $m<n$.  Let $\bar P$ be a topless iteration of length $\varepsilon$ into
which all $\bar P^n$ canonically embed. 

  Then we can define the almost FS limit $P_\varepsilon$ over $(\bar
  P^n)_{n\in \omega}$  as follows:
  Conditions in $P_\varepsilon$ are of the form $q\land i^n_\varepsilon(p)$ where
  $n\in\omega$, $p\in P^n_\varepsilon$,
  and $q\in P_\alpha$ for some $\alpha\in M^n\cap\varepsilon$ with  $q\le i^n_\alpha (p\on \alpha)$. 
  Then  $P_\varepsilon$
  is a partial CS limit over each $\bar P^n$. 
\end{Lem}

As before, we get the following corollary:
\begin{Cor}\label{cor:ctblmanycandidates}
  Given $M^n$ and $\bar P^n$ as above, we can construct a topped partial CS
  iteration $\bar P$ such that each $\bar P^n$ embeds $M^n$-completely into it;
  we can choose $Q_\alpha$ as we wish (subject to the obvious restriction that
  each $Q^n_\alpha$ is an $M^n[H^n_\alpha]$-complete subforcing). If we always choose
  $Q_\alpha$ to be ccc,  then $\bar P$ is ccc; this is the case if we set
  $Q_\alpha$ to be the union of the (countable) sets $Q^n_\alpha$. 
\end{Cor}
\begin{proof}
  We can define $P_\alpha$ by induction. If $\alpha\in \bigcup_{n\in
    \omega} M^n$ has countable cofinality,  then we use the almost
  FS limit as in   Lemma~\ref{lem:418}.  Otherwise we use the direct
  limit.  If $\alpha\in M^n$ has uncountable cofinality, then
  $\alpha'\DEFEQ  \sup(\alpha\cap M)$ is an element of $M^{n+1}$.  In our
  induction we have already considered $\alpha'$ and have defined 
  $P_{\alpha'}$ by   Lemma~\ref{lem:418} (applied to the sequence
  $(\bar P^{n+1}, \bar P^{n+2},\ldots)$). 
  This is sufficient to show that
  $i^n_\alpha:P^n_\alpha\to P_{\alpha'} \lessdot P_\alpha$ is
  $M^n$-complete. 
\end{proof}

Secondly, we can start the almost FS iteration after some $\alpha_0$
(i.e., $\bar P$ is already given up to $\alpha_0$, and we can continue
it as an almost FS iteration up to $\varepsilon$), and get the same
properties that we previously showed for the almost FS iteration, but
this time for the quotient  $P_\varepsilon/P_{\alpha_0}$. In
more detail:

\begin{Lem}\label{lem:almostfsstartatalpha}
  Assume that $\bar P^M$ is in $M$  a (topped) partial CS iteration of length
  $\varepsilon$, and that  $\bar P$ is in $V$ a topped partial CS iteration 
  of length  $\alpha_0$ over $\bar P^M\on {\alpha_0} $  
	for some $\alpha_0\in \varepsilon \cap M$.
  Then we can extend $\bar P$ to a (topped) partial CS iteration of length 
  $\varepsilon$ over $\bar P^M$,  as in the almost FS iteration (i.e.,
  using the almost FS limit at limit points $\beta>\alpha_0$ with $\beta\in M$
  of countable cofinality; and the direct limit everywhere else).
	We can use 
  any  $Q_\alpha$ for $\alpha\geq\alpha_0$ (provided $Q^M_\alpha$ is an
  $M[H^M_\alpha]$-complete subforcing of $Q_\alpha$). If all $Q_\alpha$ are ccc,
  then $P_{\alpha_0}$ forces that $P_{\varepsilon}/H_{\alpha_0}$ is ccc (in
  particular proper); if moreover all $Q_\alpha$ are $\sigma$-centered and only 
  countably many are nontrivial, then $P_{\alpha_0}$ forces that
  $P_{\varepsilon}/H_{\alpha_0}$ is $\sigma$-centered.
\end{Lem}

\subsection{Almost countable support iterations}\label{subsec:almostCS}
``Almost countable support iterations $\bar P$'' (over a given iteration $\bar P^M$ in a candidate~$M$)
will have the following two crucial properties:  There is a canonical 
$M$-complete embedding of~$\bar P^M $ into~$\bar P$, 
and $\bar P$ preserves a given random real (similar to the usual countable support iterations).

\begin{DefandClaim}\label{def:almost_CS_iteration_wolfgang}
  Let $\bar P^M$ be a topped partial CS iteration in $M$ of length $\varepsilon$.
  We can construct by induction on $\beta\in\varepsilon+1$
  the \emph{almost countable support iteration $\bar P$ over $\bar P^M$}
  (or:  almost CS iteration):
  \begin{enumerate}
  \item As induction hypothesis, we assume that the canonical embedding
    $i_\alpha$ works for every $\alpha\in \beta\cap M$.
    We set\footnote{
      So  for successors $\beta\in M$, we have $\delta'=\beta=\delta$.
      For $\beta\in M$ limit, $\beta= \delta$ and $\delta'$ is as
      in Definition~\ref{def:canonicalextension}.}
    \begin{equation}
      \label{eq:delta.prime} 
      \delta\DEFEQ  \min(M\setminus \beta),
      \quad
      \delta'\DEFEQ \sup(\alpha+1:\, \alpha\in \delta \cap M ).
    \end{equation}
    Note that $\delta'\le \beta \le \delta$. 
  \item Let $\beta=\alpha+1$. We can choose any desired forcing $Q_\alpha$;
    if $\beta\in M$ we of course require that
    \begin{equation}\label{eq:jehrwetewt}
      \text{$Q^M_\alpha$ is an $M[H^M_\alpha]$-complete subforcing of $Q_\alpha$.}
    \end{equation}
    This defines $P_\beta$.
  \item Let $\cf(\beta)> \omega$. Then
    $P_\beta$ is the direct limit. 
  \item Let $\cf(\beta) = \omega$ and assume that  $\beta\in M$ (so $M\cap \beta$ is cofinal 
    in $\beta$ and $\delta' = \beta=\delta$). 
    We define  $P_\beta=P_\delta$ as the union of the following two sets:
    \begin{itemize}
    \item The almost FS limit of $(P_\alpha,Q_\alpha)_{\alpha<\delta}$, see
      Definition~\ref{def:almostfs}. 
    \item The set $P_{\delta}^\gen$ of $M$-generic conditions $q \in P^\CS_{\delta}$, i.e., those which satisfy
      \begin{displaymath}
	q \forces_{P^\CS_{\delta}} i^{-1}_\delta[H_{P^\CS_{\delta}}]\subseteq P^M_\delta \textrm{ is } M\textrm{-generic.}
      \end{displaymath}
    \end{itemize}
  \item Let $\cf(\beta) = \omega$ and assume that  $\beta\notin M$ but $M\cap \beta$ is cofinal 
    in $\beta$,  so  $\delta' = \beta< \delta$. 
    We define  $P_\beta = P_{\delta'}$ as the union of the following two sets:
    \begin{itemize}
    \item The direct limit of $(P_\alpha,Q_\alpha)_{\alpha<\delta'}$.
    \item The set $P_{\delta'}^\gen$ of $M$-generic conditions $q \in P^\CS_{\delta'}$, i.e., those which satisfy
      \begin{displaymath}
	q \forces_{P^\CS_{\delta'}} i^{-1}_\delta[H_{P^\CS_{\delta'}}]\subseteq 
                 P^M_\delta    
          \textrm{ is } M\textrm{-generic.}
      \end{displaymath}
    \end{itemize}
    (Note that the $M$-generic conditions form an open subset of $P^\CS_\beta=P^\CS_{\delta'}$.)
  \item Let $\cf(\beta) = \omega$ and  $M \cap \beta$ 
    not cofinal in~$\beta$ (so $\beta\notin M$).
    Then $P_\beta$ is the full CS limit of
    $(P_\alpha,Q_\alpha)_{\alpha<\beta}$ (see Definition~\ref{def:fullCS}).
  \end{enumerate}
\end{DefandClaim}

So the claim is that for every choice of $Q_\alpha$ (with the obvious 
restriction~\eqref{eq:jehrwetewt}), this construction always results in a partial CS iteration
$\bar P$ over $\bar P^M$. 
The proof is a bit cumbersome; it is a variant of the usual proof that
properness is preserved in countable support iterations (see
e.g.~\cite{MR1234283}).

We will use the following fact in $M$ (for the iteration $\bar P^M$):
\proofclaim{eq:prelim}{Let $\bar P$ be a topped  iteration of length $\varepsilon$. 
Let $\alpha_1\le\alpha_2\le\beta\le\varepsilon$.  Let
$p_1$ be a $P_{\alpha_1}$-name for a condition in $P_\varepsilon$, and let
$D$ be an open dense set of $P_\beta$. Then there is a $P_{\alpha_2}$-name
$p_2$  for  a condition in $D$ such that the empty condition of
$P_{\alpha_2}$ forces:  $p_2\leq p_1\on\beta$ and: if $p_1$ is in
$P_\varepsilon/H_{\alpha_2}$, then the condition $p_2$ is as well.}

(Proof: Work in the $P_{\alpha_2}$-extension. We know that $p'\DEFEQ
p_1\restriction \beta$ is a $P_\beta$-condition. We now define $p_2$ as
follows: If $p'\notin P_\beta/H_{\alpha_2}$ (which is equivalent to $p_1\notin
P_\epsilon/H_{\alpha_2}$), then we choose any $p_2\leq p'$ in $D$ (which is
dense in $P_\beta$).  Otherwise (using that $D\cap P_\beta/H_{\alpha_2}$ is
dense in $P_\beta/H_{\alpha_2}$)  we can choose $p_2\leq p'$ in $D\cap
P_\beta/H_{\alpha_2}$.)

The following easy fact will also be useful:  
\proofclaim{eq:forces.forces}{Let $P$ be a subforcing of $Q$. We define
$P\on p\DEFEQ \{ r\in P: r\le p\}$.  
Assume that $p\in P$ and $P\on p = Q\on p$. 
\\Then for any $P$-name $\n x$ and any formula $\varphi(x)$
we have: 
 $p\forces_{P} \varphi(\n x) $ iff $p\forces_{Q} \varphi(\n x)$.
}

We now prove by induction on $\beta\le\varepsilon$ the following statement
(which includes that the Definition and
Claim~\ref{def:almost_CS_iteration_wolfgang} works up to $\beta$).  Let
$ \delta,  \delta'$ be as in \eqref{eq:delta.prime}.
\begin{Lem}\label{lem:inductionA}
   \begin{enumerate}[(a)]
     \item
   The topped iteration $\bar P$ of length $\beta$ is a partial CS iteration.
     \item
   The canonical embedding $i_\delta: P^M_\delta\to P_{\delta'}$ works, hence 
   also $i_\delta: P^M_\delta\to P_{\delta}$ works. 
     \item
    Moreover, assume that
      \begin{itemize}
         \item $\alpha\in M\cap \delta $,
	 \item $\n p\in M$ is a $P^M_\alpha$-name
   of a $P^M_\delta$-condition,
	\item 
       $q\in P_\alpha$ forces (in $P_\alpha$)
   that $\n p\on\alpha[H^M_\alpha]$ is in $H^M_\alpha$.
   \end{itemize}
   Then there is a $q^+\in P_{\delta'}$
   (and therefore in $P_\beta$) extending $q$ and forcing that 
   $\n p[H^M_\alpha]$ is in $H^M_\delta$.
  \end{enumerate} 
\end{Lem}

\begin{proof} 
  First let us deal with the trivial cases. 
   It is clear that we always get a partial CS iteration.
  \begin{itemize}
	\item 
  Assume that $\beta=\beta_0+1\in M$, i.e., $\delta=\delta'=\beta$.
  It is clear that $i_\beta$ works. 
  To get $q^+$, first extend $q$ to some $q'\in P_{\beta_0}$ (by induction
	hypothesis), then define $q^+$ extending $q'$ by $q^+(\beta_0)\DEFEQ \n
	p(\beta_0)$.
	\item 
  If $\beta=\beta_0+1\notin M$, there is nothing to do.
	\item 
  Assume that $\cf(\beta)>\omega$ (whether $\beta\in M$ or not).
  Then $\delta'<\beta$.
  So $i_\delta: P^M_\delta\to P_{\delta'}$ works by induction, and 
  similarly (c) follows from the inductive assumption.  (Use 
  the inductive assumption for $\beta=\delta'$; the $\delta$ that we got at that
  stage is the same as the current~$\delta$, and the $q^+$ we obtained at 
  that stage will still satisfy all requirements at the current stage.)
	\item 
  Assume that $\cf(\beta)=\omega$ and that $M\cap\beta$ is bounded in~$\beta$.
  Then the proof  is the same as in the previous case.
	\end{itemize}

  We are left with the cases corresponding to (4) and (5) of
  Definition~\ref{def:almost_CS_iteration_wolfgang}: $\cf(\beta)=\omega$ and
  $M\cap\beta$ is cofinal in~$\beta$.  So either $\beta\in M$, then
  $\delta'=\beta=\delta$, or $\beta\notin M$, then $\delta'=\beta<\delta$ and
  $\cf(\delta)>\omega$.

  We leave it to the reader to check that  $P_\beta$ is indeed a partial CS limit.  
   The  main point is to see that for all $p,q\in P_\beta$ the
   condition $q\wedge p$ is in $P_\beta$ as well, provided  $q\in P_\alpha$ and $q\le p\on \alpha$
   for some $\alpha<\beta$.
   If $p\in P^\gen_\beta$, then this follows because $P^\gen_\beta$ is open in $P^\CS_\beta$; 
   the other cases are immediate from the definition (by induction). 

  We now turn to  claim (c). Assume   $q\in P_\alpha$ and $\n p\in M$ are
  given, $\alpha\in M\cap \delta$.

  Let $(D_n)_{n \in \omega}$ enumerate all dense sets of~$P^M_\delta$
  which lie in~$M$, and let $(\alpha_n)_{n \in \omega} $ be a sequence
  of ordinals in $M$ which is cofinal in $\beta$, where $\alpha_0=\alpha$.
  
  Using \eqref{eq:prelim} in~$M$, we can find a sequence $(\n p_n)_{n
  \in \omega}$ satisfying the following in~$M$, for all $n>0$:
  \begin{itemize}
    \item 
            $\n p_0 = \n p$.
    \item 
	     $\n p_n\in M$ is a $P^M_{\alpha_n}$-name of a
	     $P^M_\delta$-condition in~$D_n$.
    \item 
            $\forces_{P^M_{\alpha_{n}}} \n p_{n}\le_{P_\delta^M} \n p_{n-1}$.
    \item 
            $\forces_{P^M_{\alpha_{n}}} $ If  $
	    \n p_{n-1}\on\alpha_{n}\in H^M_{\alpha_{n}}$,  then 
	    $\n p_{n}\on\alpha_{n}\in H^M_{\alpha_{n}}$ as well.
  \end{itemize}
  Using the inductive assumption for the $\alpha_n$'s, we can now find 
  a sequence $(q_n)_{n \in \omega}$ of conditions satisfying the following: 
  \begin{itemize}
    \item  $q_0 = q$, $q_n\in P_{\alpha_n}$.
    \item  $q_{n}\on \alpha_{n-1} = q_{n-1}$.
    \item  $q_n \forces_{P_{\alpha_n}} \n p_{n-1}\on \alpha_n \in H^M_{\alpha_n}$, 
    so also                          $ \n p_{n  }\on \alpha_n \in H^M_{\alpha_n}$.
  \end{itemize}
  Let $q^+\in P^\CS_\beta$ be the union of the $q_n$.  Then  for all $n$:
   \begin{enumerate}
    \item 
       $q_n \forces_{ P^\CS_\beta}  \n p_n\on \alpha_n \in H^M_{\alpha_n}$, so also $q^+$ forces this. 
       \\(Using induction on $n$.) 
    \item  For all $n$ and all  $m\ge n$: 
       $q^+ \forces_{ P^\CS_\beta}  \n p_m\on \alpha_m \in H^M_{\alpha_m}$, 
                 so also          $ \n p_n\on \alpha_m \in H^M_{\alpha_m}$. 
       \\(As $\n p_m \le \n p_n$.)  
    \item 
       $q^+ \forces_{ P^\CS_\beta}  \n p_n\in H^M_{\delta}$.
       \\ (Recall that  $P^\CS_\beta$ is separative, see Fact~\ref{fact:eq.eqstar}. 
	    So $i_\delta(\n p_n)\in H_\delta$ iff $i_{\alpha_n}(\n p\on \alpha_m)\in
	    H_{\alpha_m}$ for all large~$m$.)
    \end{enumerate}

    As $q^+\forces_{P^\CS_\beta }  \n p_n \in  D_n\cap H^M_\delta$, we conclude
    that $q^+\in P^\gen_\beta$ (using Lemma~\ref{lem:wolfgang},
    applied to ${P^\CS_\beta }$).   In particular,
    $P^\gen_\beta$ is dense in $P_\beta$: Let $q\wedge i_\delta(p)$
     be an element of the almost FS limit;  so $q\in P_\alpha$ for some
     $\alpha < \beta$.  Now find a generic  $q^+$  extending $q$ and stronger than $i_\delta(p)$,
     then $q^+\le q\wedge i_\delta(p)$.

    It remains to show that $i_\delta$ is $M$-complete.  
   Let $A\in M$ be a maximal antichain of $P^M_\delta$, and $p\in P_\beta$.
   Assume towards a contradiction that $p$ forces  in $P_\beta$
   that $ i^{-1}_{\delta}[H_\beta ]$  does not intersect $A$ in exactly one point.

   Since $P^\gen_\beta$ is dense in $P_\beta$, we can find some $q\leq p$ in $P^\gen_\beta$.  
     Let \[P'\DEFEQ  \{r\in P^\CS_\beta: r\le
     q\}=\{r\in P_\beta: r\le q\},   \]
  where the equality holds because $P^\gen_\beta$ is  open in $P^\CS_\beta$.

   Let $\Gamma $ be the canonical name for a $P'$-generic filter, 
	 i.e.: 
   $\Gamma\DEFEQ  \{(\check r, r): r\in P' \}$. 
     Let $R$ be either 
     $P^\CS_\beta$  or  $ P_\beta$.   We write 
   $\langle \Gamma\rangle _R$ for the filter generated
   by~$\Gamma$ in~$R$, i.e., $\langle \Gamma\rangle _R \DEFEQ \{r\in R: (\exists r'\in \Gamma ) \
   r'\le r\}$.    So 
   \begin{equation}\label{gehtsnochduemmer}
   q\forces_R H_R =\langle \Gamma\rangle_R.
   \end{equation}
   
   We now see that the following hold:
    \begin{enumerate}
     \item[--] $ q\forces_{P_\beta}   i^{-1}_{\delta}[ H _{P_\beta}  ] $
           does not intersect $A$ in exactly one point. (By
	   assumption.) 
     \item[--] $ q\forces_{P_\beta}   i^{-1}_{\delta}[ \langle
           \Gamma \rangle_{P_\beta}] $
           does not intersect $A$ in exactly one point. (By
	   \eqref{gehtsnochduemmer}.)
     \item[--] $ q\forces_{P_\beta^\CS }   i^{-1}_{\delta}[ \langle  \Gamma \rangle_{P_\beta}] $
           does not intersect $A$ in exactly one point.  (By \eqref{eq:forces.forces}.) 
     \item[--] $ q\forces_{P_\beta^\CS }   i^{-1}_{\delta}[ \langle  \Gamma
       \rangle_{P_\beta^\CS}] $ does not intersect $A$ in exactly one point.
       (Because $i_\delta$ maps $A$ into $P_\beta\subseteq P_\beta^\CS$, so
       $A\cap i^{-1}_\delta[\langle Y\rangle_{P_\beta}]  = A \cap i^{-1}_\delta[\langle
       Y\rangle_{P_\beta^\CS}]$ for all~$Y$.)
     \item[--] $ q\forces_{P_\beta^\CS }   i^{-1}_{\delta}[ H_{P_\beta^\CS}]$ 
          does not intersect $A$ in exactly one point. (Again by  \eqref{gehtsnochduemmer}.)
    \end{enumerate}
   But this, according to
   the definition of $P^\gen_\beta$, implies $q\notin
   P^\gen_\beta$, a contradiction.
\end{proof}

We can also show that the almost CS iteration
of proper forcings $Q_\alpha$ is  proper.
(We do not really need this fact, as we could allow non-proper iterations in
our preparatory forcing,  see Section~\ref{sec:7a}(\ref{item:nonproper}).  In some sense, $M$-completeness 
replaces properness, so the proof of $M$-completeness was similar to the 
``usual'' proof of properness.) 
\begin{Lem}
Assume that in Definition~\ref{def:almost_CS_iteration_wolfgang}, 
every $Q_\alpha$ is (forced to be) proper.  Then also each $P_\delta$ 
is proper. 
\end{Lem}
\begin{proof}
  By induction on $\delta\le \varepsilon$ we prove that for all $\alpha<\delta$
	the quotient $P_\delta/H_ \alpha$ is (forced to be) proper.  We use the following
	facts about properness: 
	\proofclaim{claim:proper.successor}{
	  If $P$ is proper and $P$ forces that $Q$ is proper, then $P*Q$ 
		is proper.}
	\proofclaim{claim:proper.omega}{
	  If $\bar P$ is an  iteration of length $\omega$ 
		and if each $Q_n$ is forced to be proper, then
		the inverse limit $P_\omega$ is proper, as are all quotients
		$P_\omega/H_n$.}
	\proofclaim{claim:proper.omega.1}{
		If $\bar P$ is an  iteration of length $\delta$  with
		$\cf(\delta)>\omega$, 
		and if all quotients  $P_\beta/H_\alpha$ (for $\alpha < \beta < \delta$)
		are forced to be proper, then
		the direct limit $P_\delta$ is proper, as are all quotients
		$P_\delta/H_\alpha$.}

   If  $\delta$ is a successor, then our inductive claim  easily follows from 
	 the inductive assumption together with~\eqref{claim:proper.successor}. 

	 Let  $\delta$ be  a limit of countable cofinality, say $\delta =  \sup_n
	 \delta_n$. Define an iteration $\bar P'$ of length $\omega$ 
	 with $Q'_n\DEFEQ   P_{\delta_{n+1}} / H_{\delta_n}$.  (Each $Q'_n$ is proper,
	  by inductive assumption.) There is a natural forcing 
	 equivalence between  $ P^\CS_\delta$ and $P^{\prime\CS}_\omega$, the full CS limit of $\bar
	 P'$. 

  Let $N \esm H(\chi^*)$ contain $\bar P, P_\delta, \bar P', M, \bar P^M$.
	Let $p\in P_\delta\cap N$.  Without loss of generality $p\in P_\delta^\gen$. 
  So below $p$ we can identify $P_\delta$ with $P^\CS_\delta$ and hence 
	with $P^{\prime\CS}_ \omega$; now apply~\eqref{claim:proper.omega}.

	 The case of uncountable cofinality is similar, using~\eqref{claim:proper.omega.1} instead. 
\end{proof}

Recall the definition of $\sqsubset_n$ and $\sqsubset$ from Definition~\ref{def:sqsubset}, the notion of
(quick) interpretation $Z^*$ (of a name  $\n Z$ of a code for a null set) 
 and the definition of local preservation
of randoms from Definition~\ref{def:locally.random}. 
Recall that we have seen in Corollaries~\ref{cor:ultralaverlocalpreserving} and~\ref{cor:januslocallypreserves}:
\begin{Lem}\label{lem:4.28}
  \begin{itemize}
    \item If $Q^M$ is an ultralaver forcing in $M$ and $r$ a real,
		 then there is an ultralaver forcing $Q$ over  $Q^M$ locally
     preserving randomness of $r$ over~$M$.
		\item If $Q^M$ is a Janus forcing in $M$ and $r$ a real, then there is a
		 Janus forcing $Q$ over~$Q^M$ locally preserving randomness
of~$r$ over~$M$. 
\end{itemize}
\end{Lem}

We will prove the following preservation theorem:
\begin{Lem}\label{lem:iterate.random}
  Let $\bar P$ be an almost CS iteration (of length $\varepsilon$) over~$\bar P^M$,
  $r$ random over~$M$, and $p\in P^M_\varepsilon$. Assume 
  that each $P_\alpha$ forces 
	that $Q_\alpha$ locally preserves randomness of $r$ over
	$M[H^M_\alpha]$.
  Then there is some $q\leq p$ in $P_\varepsilon$ forcing that 
  $r$ is random over~$M[H^M_\varepsilon]$.
\end{Lem}

What we will actually need is the following variant:
\begin{Lem}\label{lem:preservation.variant}
	Assume that  $\bar P^M$ is in $M$ a topped partial CS iteration of length
	$\varepsilon$, and we already have some topped partial CS iteration $\bar P$
	over~$\bar P^M\on\alpha_0$ of length $\alpha_0\in M\cap\varepsilon$.  Let $\n
	r$ be a $ P_{\alpha_0}$-name of a random real over~$M[H^M_{\alpha_0}]$.
	Assume that we extend $\bar P$ to length $\varepsilon$ as an almost CS
	iteration\footnote{Of course our official definition of almost CS iteration assumes
	that we start the construction at $0$, so we modify this definition in the
	obvious way.} using forcings $Q_\alpha$ which 
	locally preserve the randomness of $\n r$ over~$M[H^M_\alpha]$, 
	witnessed by a sequence $(D_k^{Q_\alpha^M})_{k\in \omega}$. 
	Let $p\in P^M_\varepsilon$.
	 Then we can find a $q\leq p$ in $P_{\varepsilon}$ 
	 forcing that $\n r$ is random over
	 $M[H^M_\varepsilon]$.
\end{Lem}

Actually, we will only prove the two previous lemmas under the following
additional assumption (which is enough for our application, and saves some
unpleasant work). This additional assumption is not really necessary; without
it, we could use the method of~\cite{MR2214624} for the proof.
\begin{Asm}\label{asm:quick} 
\begin{itemize} 
  \item For each $\alpha\in M\cap \varepsilon$,
    ($P^M_\alpha$ forces that) $Q^M_\alpha$ is either trivial\footnote{More specifically, $Q^M_\alpha=\{\emptyset\}$.}
    or adds a new $\omega$-sequence of ordinals.
    Note that in the latter case  we can assume without loss of generality
    that $\bigcap_{n\in\omega}D^{Q_\alpha^M}_n=\emptyset$  (and, of course, that
    the $D^{Q_\alpha^M}_n$ are decreasing).
	\item Moreover, we assume that already in $M$ there is a set $T\subseteq
		\varepsilon$ such that $P_\alpha^M$ forces: $Q_\alpha^M$ is trivial iff
    $\alpha\in T$. (So whether $Q_\alpha^M$ is trivial or not does not depend
    on the generic filter below $\alpha$, it is already decided in the 
    ground model.)  
\end{itemize}
\end{Asm}

The result will follow as a special case of the following lemma, which we prove
by induction on~$\beta$. (Note that this is a refined version of the proof of
Lemma~\ref{lem:inductionA} and similar to the proof of the preservation theorem
in~\cite[5.13]{MR1234283}.)

\begin{Def}\label{def:quick}
  Under the assumptions of 
 Lemma~\ref{lem:preservation.variant} and
Assumption~\ref{asm:quick},
let $\n Z$ be a $P_\delta$-name, $\alpha_0\le \alpha <
\delta$, and let $\bar p = (p^k)_{k\in \omega}$ be a
sequence of $P_\alpha$-names of conditions in
$P_\delta/H_\alpha$.  Let $Z^*$ be a $P_\alpha$-name. 

We say that $(\bar p, Z^ *)$ is a \emph{quick} interpretation
of $\n Z$ if $\bar p$ interprets $\n Z$ as $Z^*$ (i.e.,
$P_\alpha$ forces that $p^k$ forces $\n Z \on k  = Z^*\on
k$ for all $k$), and moreover: 
\begin{quote}
  Letting $\beta\ge \alpha$ be minimal  with $Q^M_\beta$
 nontrivial (if such $\beta$ exists):
   $P_\beta$ forces that  the sequence $(p^k(\beta))_{k\in \omega}$ is
 quick in $Q^M_\beta$, i.e., $p^k(\beta)\in D^{Q_\beta^M}_k$ for all~$k$.
\end{quote}
\end{Def}

It is easy to see that:
\proofclaim{eq:find.quick}{For every name $\n Z$ there is a 
  quick interpretation $(\bar p, Z^*)$.}

\begin{Lem} \label{lem:induktion.wirklich}
   Under the same assumptions as above,
   let $\beta$, $\delta$, $\delta'$ be as in \eqref{eq:delta.prime}
	 (so  in particular  we have $\delta'\le\beta\le\delta\le\varepsilon$).
	 \\
   {\bf Assume that }
   \begin{itemize}
     \item $\alpha\in M\cap \delta$ ($=M\cap \beta $) and $\alpha\ge \alpha_0$ (so $\alpha<\delta'$),
      \item $ p\in M$ is a $P^M_\alpha$-name of a 
	      $P^M_\delta$-condition,
		 \item $\n Z\in M$ is a $P^M_\delta$-name of a code for null set,
		 \item $Z^*\in M$ is a $P^M_\alpha$-name of a  code for a null set,
     \item  $P^M_\alpha$ forces:
		    $\bar p = (p^k)_{k\in \omega}\in M$ is a quick
		 sequence in $P^M_\delta/H^M_\alpha$ interpreting $\n
		 Z $ as~$Z^*$  (as in Definition~\ref{def:quick}),
     \item  $P^M_\alpha$ forces:
		       if $p\on
		 \alpha \in H^M_\alpha$, then $p^0\le p$,
		 \item $q\in P_\alpha$ forces $p\on \alpha \in H^M_\alpha$,
			  \item 
	        $q$ forces that $r$ is random over~$M[H^M_\alpha]$, so in particular
					there is (in $V$) a $P_\alpha$-name $ \n c_0$ below $q$  for the minimal~$c$ with $Z^*\sqsubset_{c} r$.
   \end{itemize}
   {\bf Then} there is a condition $q^+\in P_{\delta'}$, extending
	 $q$, and  forcing the following: 
	 \begin{itemize}
			  \item 
	                $p\in H^M_\delta$, 
			  \item 
	                $r$ is random over~$M[H^M_\delta]$,
			\item 
	                $\n Z\sqsubset_{\n c_0} r$.
   \end{itemize}
\end{Lem}
  We actually claim a slightly stronger version, where instead of
	$Z^*$ and $\n Z$ we have finitely many codes for null sets and names of
	codes for null sets, respectively.   We will use this stronger claim as
	inductive assumption, but for notational simplicity we  only    prove the weaker
	version; it is easy to see that the weaker version implies the stronger
	version.    

\begin{proof}
  \emph{\textbf{The nontrivial successor case:}} ${\beta=\gamma+1\in M}$.

  If $Q^M_\gamma$ is trivial, there is nothing to do.   

  Now let $\gamma_0\ge \alpha $ be minimal with $Q^M_{\gamma_0}$ nontrivial. We 
  will distinguish two cases: $\gamma=\gamma_0$ and $\gamma>\gamma_0$.

  Consider first the case that   $\gamma=\gamma_0$.  
   Work in $V[H_\gamma]$ where $q\in H_\gamma$.  Note that $M [H^M_\gamma] = 
	 M[H^M_\alpha]$.  So $r$ is random over~$ M [H^M_\gamma]$, and
	 $(p^k(\gamma))_{k\in \omega}$  quickly interprets $\n Z$ as $Z^*$ in $Q^M
	 _\gamma$.
   Now let $q^+\on \gamma= q$, and  use the fact that 
    $Q_\gamma$ locally preserves randomness to find $q^+(\gamma)\le p^0(\gamma)$.

  Next consider the case that $Q^M_\gamma$ is nontrivial and $\gamma\ge \gamma_0+1$.
	 Again work in $V[H_\gamma]$.
  Let $k^*$ be maximal with $p^{k^*}\on \gamma\in H^M_\gamma$.  (This $k^*$
  exists, since the sequence $(p^{k})_{k\in  \omega}$ was quick, so there is even 
  a $k$ with $p^{k}\on ( {\gamma_0+1}) \notin  H^M_{\gamma_0+1}$.) 
	Consider $\n Z$ as a $Q^M_\gamma$-name,  and (using~\eqref{eq:find.quick})
  find a  quick interpretation $Z'$ of $\n Z$ witnessed by a
  sequence starting with $p^{k^*}(\gamma)$.   In $M[H^M_\alpha]$, $Z'$
	is now a $P^M_\gamma/H^M_\alpha$-name.  Clearly, the sequence 
	$(p^k\on \gamma)_{k\in \omega}$ is a quick sequence interpreting $Z'$ as
  $Z^*$.  (Use the fact that $p^k\on \gamma$ forces $k^*\ge k$.)
  \\
	Using the induction hypothesis, we can first extend $q$ to a condition $q'\in
	P_\gamma$ and then (again by our assumption that $Q_\gamma$ locally preserves
	randomness) to a condition $q ^+\in P_{\gamma+1}$.

  \emph{\textbf{The nontrivial limit case:}}
	        ${M\cap \beta}$ unbounded in ${\beta}$, i.e., $\delta'=\beta$. 
					(This deals with cases~(4) and~(5) in
					Definition~\ref{def:almost_CS_iteration_wolfgang}. 
					In case (4) we have $\beta\in M$, i.e., $\beta=\delta$; 
					in case (5) we have $\beta\notin M$ and  $\beta< \delta$.)
	
  Let $\alpha=\delta_0 < \delta_1 < \cdots$ be a sequence of $M$-ordinals
	cofinal in $M\cap \delta' = M\cap \delta$.    We may assume\footnote{If from some $\gamma$ on
			all $Q^M_\zeta$ are trivial, then $P^M_\delta=P^M_\gamma$, so by
			induction there is nothing to do.   If $Q^M_\alpha$ itself is
			trivial, then we let $\delta_0\DEFEQ \min\{\zeta: Q^M_\zeta \text{ nontrivial}\}$ instead.} 
	that each $Q^M_{\delta_n}$ is nontrivial.

  Let $(\n Z_n)_{n\in\omega}$ be a list of all $P^ M_\delta$-names in $M$ 
  of codes for null sets (starting with our given null set
  $\n Z = \n Z_0$).
	Let $(E_n)_{n\in\omega}$ enumerate all open dense sets of $P^M_\delta$ from $M$, without
	loss of generality\footnote{well, if we just enumerate a basis of the open
	sets instead of all of them\dots}
	we can assume that:
	\proofclaim{claim:E.n.decides}{
	        $E_n$ decides $\n Z_0\on n $, \dots, $\n Z_n\on n$.
        }
	We write $p^k_0$ for $p^k$, and $Z_{0,0}$ for $Z^*$; as mentioned
	above, $\n Z=\n Z_0$. 

	By induction on $n$ we can now find  a sequence $\bar p_n = (p^k_n)_{k\in
	\omega}$  and  $P^M_{\delta_n }$-names $Z_{i,n}$  for $i\in \{0,\dots, n\}$
        satisfying the following:

	 \begin{enumerate}
			 \item  $P^M_{\delta_n}$ forces that 
			  $p^0_n \le p^{k}_{n-1}$ whenever
				$p^k_{n-1}\in P^M_{\delta}/H^M_{\delta_n}$. 
			\item $P_{\delta_n}^M$  forces that $p^0_n\in E_n$. (Clearly $E_n\cap
			P^M_\delta/H^M_{\delta_n}$ is a dense set.)
			\item  $\bar p_n\in M $ is a $P^M_{\delta_n}$-name for a quick sequence
			 interpreting $(\n Z_0,\ldots, \n Z_n)$ as 
			 $(Z_{0,n},\ldots, Z_{n,n})$
			 (in $P^M_\delta/H^M_{\delta_n}$),
	        so   $Z_{i,n}$  is a $P^M_{\delta_n}$-name of a code for a null set, for 
			$0\le i \le n$. 
	 \end{enumerate}

   Note that this implies that the sequence $(p^k_{n-1}\on \delta_{n})$ is 
	 (forced to be) a quick sequence interpreting
	 $(Z_{0,{n}},\ldots, Z_{n-1, {n}})$ as 
	 $(Z_{0,{n-1}},\ldots, Z_{n-1, {n-1}})$. 

   Using the induction hypothesis, we now
	 define a sequence $(q_n)_{n\in \omega}$ of conditions $q_n\in P_{\delta_n}$
	 and a sequence $(c_n)_{n\in\omega}$ (where $c_n$ is a $P_{\delta_n}$-name) 
	 such that (for $n>0$) $q_n $ extends $q_{n-1}$ and forces the following: 
	 \begin{itemize}
	    \item  $ p_{n-1}^0\on \delta_n \in H^M_{\delta_n}$. 
			\item  Therefore, $p_n^0 \le p_{n-1}^0$. 
			\item  $r$ is random over~$M[H^M_{\delta_n}]$.   
			\item Let $c_n$ be the least $c$ such that $Z_{n,n}\sqsubset_c r$.
			\item 
				 $ Z_{i,n} \sqsubset_{c_i} r$ for $i=0,\ldots, n-1$. 
	 \end{itemize}
   Now let $q = \bigcup_n q_n\in P^\CS_{\delta'}$. 
   As in Lemma~\ref{lem:inductionA}  it is easy to see that $q\in P^\gen_{\delta'} 
   \subseteq P_{\delta'}$.
   Moreover, by~\eqref{claim:E.n.decides} we get that 
         $q$ forces that $\n Z_i = \lim_n Z_{i,n}$.
	 Since each set  $C_{c,r}\DEFEQ \{x:x\sqsubset_{c} r\}$ is closed, this implies that 
         $q $ forces 
	 $\n Z_i \sqsubset_{c_i} r$, in particular $ \n Z= \n Z _0  \sqsubset_{c_0} r$.

		\emph{\textbf{The trivial  cases:}}  In all other cases, 
		  $M \cap \beta $ is bounded in $\beta$, so we already dealt with 
			everything at stage $ \beta_0\DEFEQ  \sup(\beta \cap M)$.  
			Note that $\delta_0'$ and $\delta_0$ used at stage $\beta_0$
			are the same as the current $\delta'$ and $\delta$.
\end{proof}

\section{The forcing construction}\label{sec:construction}

In this section we describe a $\sigma$-closed ``preparatory'' forcing notion
$\prep$; the generic filter will define a ``generic'' forcing iteration
$\bar \BP$, so elements of $\prep$ will be approximations to such an
iteration.   In Section~\ref{sec:proof} we will show that  the forcing
$\prep*\BP_\om2$ forces BC and dBC. 

{} From now on, we assume CH in the ground model.

\subsection{Alternating iterations, canonical embeddings and the preparatory forcing $\prep$}
The preparatory forcing $\prep$ will consist of pairs $(M,\bar P)$, where $M$
is a countable model and $\bar P\in M$ is an iteration of ultralaver and
Janus forcings.

\begin{Def}\label{def:alternating}
	An alternating iteration\footnote{See Section~\ref{sec:alternativedefs} for possible 
        variants of this definition.} is a topped partial CS iteration $\bar P$
  of length $\om2$ satisfying the following:
  \begin{itemize}
		\item Each $P_\alpha$ is proper.\footnote{This does not seem to be
      necessary, see Section~\ref{sec:alternativedefs}, but it is easy to ensure and
might be comforting to some of the readers and/or authors.} 
    \item For $\alpha$ even, either both $Q_{\alpha}$ and 
  	  $Q_{\alpha+1}$ are (forced by the empty condition to be) trivial,\footnote{%
 				For definiteness, let us agree that the trivial forcing is the
     		singleton $\{\emptyset\}$.}
			or $P_\alpha$ forces that $Q_\alpha$  is an ultralaver forcing adding the
			generic real $\bar \ell_\alpha$, and  $P_{\alpha+1}$ forces that
			$Q_{\alpha+1}$ is a Janus forcing based on $\bar \ell^*_\alpha$ (where $ \bar
			\ell^*$ is defined from $ \bar  \ell $ as in
			Lemma~\ref{lem:subsequence}).
  \end{itemize}
\end{Def}
We will call an even index an ``ultralaver position'' and an odd one a ``Janus
position''.

As in any partial CS iteration, each $P_\delta$ for $\cf(\delta)>\omega$
(and in particular $P_\om2$) is a direct limit.

Recall that in Definition~\ref{def:canonicalembedding} we have defined the notion 
``$\bar P^M$  canonically embeds into $\bar P$'' for nice candidates $M$ and 
iterations $\bar P\in V$ and $\bar P^M\in M$.
Since our iterations now  have length $\omega_2$, this 
means that the canonical embedding works up to and
including\footnote{This is stronger than to require that the canonical
embedding works for every $\alpha\in\om2\cap M$, even though both $P_\om2$ and
$P^M_{\om2}$ are just direct limits; see footnote~\ref{fn:too.late}.}  
$\om2$.

In the following, we will use pairs $x=(M^x,\bar P^x)$ as conditions in a
forcing, where $\bar P^x $ is an alternating iteration in the nice candidate $M^x$.
We will adapt our notation accordingly: Instead of writing $M$,
 $\bar P^M$, $P^M_\alpha$ 
$H_\alpha^M$ (the induced filter), $Q_\alpha^M$, etc.,
 we will write $M^x$, $\bar P^x$, $P^x_\alpha$, $H_\alpha^x$, $Q^x_\alpha$, etc. 
Instead of ``$\bar P^x$ canonically embeds into $\bar P$'' we
will say%
\footnote{Note the linguistic asymmetry here: A symmetric and more verbose variant
would say ``$x=(M^x,\bar P^x)$ canonically embeds into $(V,\bar P)$''.}
 ``$x$ canonically embeds into $\bar P$''
 or ``$(M^x, \bar P^x)$ canonically embeds into $\bar P$''
 (which is a more exact
notation anyway, since the test whether the embedding is $M^x$-complete uses 
both $M^x$ and $\bar P^x$,
not just $\bar P^x$).

The following rephrases Definition~\ref{def:canonicalembedding} of
a canonical embedding in our new notation, taking into account that:
\begin{quote} 
   $\bL_{{\bar D}^x}$ is an $M^x$-complete subforcing
of~$\bL_{\bar D}$  \   \  iff  \  \  $\bar D$
extends $\bar D^x$ 
\end{quote} (see Lemma~\ref{lem:LDMcomplete}).

\begin{Fact}\label{fact:canonical}
	  $x=(M^x,\bar P^x)$ canonically embeds into $\bar P$,
   if (inductively) for all $\beta\in \om2\cap
  M^x\cup\{\om2\}$ the following holds:
  \begin{itemize}
    \item Let $\beta=\alpha+1$ for $\alpha$ even (i.e., an ultralaver position).
      Then either $Q^x_\alpha$ is trivial (and $Q_\alpha$ can be trivial or not), 
      or we require that
      ($P_\alpha$ forces that) the $V[H_\alpha]$-ultrafilter system $\bar D$ used for $Q_\alpha$
      extends the $M^x[H^x_\alpha]$-ultrafilter system $\bar D^x$
      used for $Q^x_\alpha$.
    \item Let $\beta=\alpha+1$ for $\alpha$ odd (i.e., a Janus position).
      Then either  $Q^x_\alpha$ is trivial, 
      or we require that ($P_\alpha$ forces that) 
      the Janus forcing $Q^x_\alpha$ is an $M^x[H^x_\alpha]$-complete 
      subforcing of the Janus forcing $Q_\alpha$.
		\item Let $\beta$ be a limit. Then the canonical extension
      $i_\beta:P^x_\beta\to P_\beta$ is $M^x$-complete. (The canonical extension 
       was defined in Definition~\ref{def:canonicalextension}.)
  \end{itemize}
\end{Fact}

Fix a sufficiently large regular cardinal $\chi^*$ (see Remark~\ref{rem:fine.print}).
\begin{Def}\label{def:prep}
  The \qemph{preparatory forcing} $\prep$ consists of 
  pairs $x=(M^x,\bar P^x)$ such that
  $M^x\in H(\chi^*)$ is a nice candidate (containing $\om2$), and
  $\bar P^x$ is in $M^x$ an alternating iteration (in particular topped and
  of length $\om2$).
  \\
  We define $y$ to be stronger than $x$ (in symbols: $y\leq_{\prep} x$),
  if the following holds: either $x=y$, or:
  \begin{itemize}
    \item $M^x\in M^y$ and $M^x$ is countable in $M^y$.   
    \item $M^y$ thinks that $(M^x,\bar P^x)$ canonically embeds into $\bar P^y$.
  \end{itemize}
\end{Def}
  Note that this order on $\prep$ is transitive.

We will sometimes write $i_{x,y}$ for
 the canonical embedding (in $M^y$) from  $P^x_{\om2}$ to $P^y_{\om2}$.

There are several variants of this definition which result in equivalent
forcing notions. We will briefly come back to this in
Section~\ref{sec:alternativedefs}.

The following is trivial by elementarity:
\begin{Fact}\label{fact:esmV}
		Assume that $\bar P$ is an alternating iteration (in $V$), that $x=(M^x,\bar P^x) \in
		\prep$ canonically embeds into $\bar P$, and that $N \esm H( \chi^*)$
		contains $x$ and $\bar P$.  Let $y=(M^y,  \bar P^y)$ be the ord-collapse of
		$(N, \bar P)$.  Then $y\in\prep$ and $y\le x$.
\end{Fact}

This fact will be used,  for example, to get from the following Lemma~\ref{lem:trivialexample}
to Corollary~\ref{cor:gurke3}.

\begin{Lem}\label{lem:trivialexample}
  Given $x\in\prep$, there is an alternating iteration $\bar P$
  such that $x$ canonically embeds into $\bar P$.
\end{Lem}

\begin{proof}
For the proof, we use either of the partial CS constructions 
introduced in the previous chapter (i.e., an almost CS iteration or an almost
FS iteration over $\bar P^x$). The only
thing we have to check is that we can indeed choose $Q_\alpha$ that satisfy the
definition of an alternating iteration (i.e., as ultralaver or Janus forcings) and
such that $Q^x_\alpha$ is $M^x$-complete in $Q_\alpha$.

In the ultralaver case we arbitrarily extend  $\bar D^x$ to an ultrafilter
system $\bar D$, which is justified  by Lemma~\ref{lem:LDMcomplete}.

In the Janus case, we take $Q_\alpha\DEFEQ  Q_\alpha^x$ (this works by  Fact~\ref{fact:janus.ctblunion}). Alternatively, we could 
extend $Q_\alpha^x$ to a random forcing (using  Lemma~\ref{lem:janusrandompreservation}).
\end{proof}

\begin{Cor}\label{cor:gurke3}
  Given $x\in\prep$ and an HCON object $b\in H(\chi^*)$ (e.g., a real or an ordinal), there is a
  $y\leq x$ such that $b\in M^y$.
\end{Cor}
What we will actually need are the following three variants:
\begin{Lem}\label{lem:prep.is.sigma.preparation}
  \begin{enumerate}
  \item
  Given $x\in\prep$ there is a $\sigma$-centered
  alternating iteration $\bar P$ above $x$.
  \item\label{item:karotte2}
   Given a decreasing sequence $\bar x=(x_n)_{n\in \omega}$ in
  $\prep$, there is an alternating iteration $\bar P$ 
  such that each $x_n$ embeds into $\bar P$. Moreover,   
  we can assume that for all Janus positions $\beta$, the
  Janus \footnote{If all $Q^{x_n}_\beta$ are trivial, then we 
  may also set $Q_\beta$ to be the trivial forcing, which is formally 
  not a Janus forcing.} 
  forcing
  $Q_\beta$ is (forced to be) the union of the
  $Q^{x_n}_\beta$, and that for all limits $\alpha$, the forcing
	$P_\alpha$ is the  almost FS limit 
  over~$(x_n)_{n\in\omega}$
  (as in Corollary~\ref{cor:ctblmanycandidates}).
	\item\label{item:karotte6}
  Let $x,y\in \prep$. Let $j^x$ be  the transitive collapse of 
  $M^x$, and define $j^y$ analogously.
  Assume that $j^x{}[M^x]=j^y {}[ M^y]$, that $j^x(\bar P^x)=j^y(\bar P^y)$ and
  that there are $\alpha_0\leq\alpha_1<\om2$ such that:
  \begin{itemize}
  \item $M^x\cap \alpha_0=M^y\cap \alpha_0$ (and thus $j^x\on \alpha_0=j^y\on\alpha_0$).
  \item $M^x\cap [\alpha_0, \omega_2) \subseteq [\alpha_0, \alpha_1)$.
  \item $M^y\cap [\alpha_0, \omega_2) \subseteq [\alpha_1, \om2)$.
  \end{itemize}
  Then there is an alternating iteration $\bar P$ 
  such that both $x$ and $y$ canonically embed into it.
  \end{enumerate}
\end{Lem}

\begin{proof}
 For (1), use an almost FS iteration. 
 We only use the coordinates in $M^x$, and use  the (countable!)
 Janus forcings $Q_\alpha\DEFEQ  Q^x_\alpha$ for all Janus positions
 $\alpha\in M^x$ (see Fact~\ref{fact:janus.ctblunion}). 
 Ultralaver forcings are $\sigma$-centered anyway, so 
 $P_\varepsilon$ will be $\sigma$-centered,  by Lemma~\ref{lem:4.17}.

 For (2), use the almost FS iteration over the sequence $(x_n)_{n\in \omega}$
  as in Corollary~\ref{cor:ctblmanycandidates}, 
  and at Janus positions $\alpha$ set $Q_\alpha$ to be the 
   union  of the   $Q^{x_n}_\alpha$.   (By Fact~\ref{fact:janus.ctblunion},
   $Q^{x_n}_\alpha$ is $M^{x_n}$-complete in $Q_\alpha$, so
   Corollary~\ref{cor:ctblmanycandidates} can be applied here.)

	For (3), we again use an almost FS  construction.  This
	time we start with an almost FS construction over $x$ up 
	to $\alpha_1$, and then continue with an almost FS
	construction over $y$. 
\end{proof}
As above, Fact~\ref{fact:esmV} gives us the following consequences:

\begin{Cor}\label{cor:bigcor}
  \begin{enumerate}
    \item\label{item:gurke0}
      $\prep$ is $\sigma$-closed.  Hence $\prep$ does not add new HCON objects (and in particular: no new reals). 
	  \item\label{item:gurke1} $\prep$ forces that the generic filter $G\subseteq \prep$ is
	  	$\sigma$-directed, i.e., for every countable subset $B$
  		of $G$ there is a $y\in G$ stronger than each element of $B$.
    \item\label{item:gurke2} 
    $\prep$ forces CH. (Since we assume CH in $V$.)
		\item\label{item:martin}  
      Given a decreasing sequence $\bar x=(x_n)_{n\in \omega}$ in 
			$\prep$ and any HCON object $b\in H(\chi^*)$, there is a
			$y\in \prep$ such that
			\begin{itemize}
			   \item $y\leq x_n$ for all $n$,
			   \item $M^y$ contains $b$ and the sequence $\bar x$,
				 \item for all Janus positions $\beta$, $M^y $ thinks that the
			Janus forcing $Q^y_\beta$ is (forced to be) the union of
			the~$Q^{x_n}_\beta$,
			   \item for all limits $\alpha$, $M^y $ thinks that $P^y_\alpha$
					 is the almost FS limit\footnote{constructed in Lemma~\ref{lem:418}}
					   over $(x_n)_{n\in\omega}$
					   (of~$(P^y_\beta)_{\beta<\alpha}$).
			\end{itemize}
  \end{enumerate}
\end{Cor}
\begin{proof}
Item~(\ref{item:martin}) directly follows from
Lemma~\ref{lem:prep.is.sigma.preparation}(\ref{item:karotte2})   and
Fact~\ref{fact:esmV}.
Item~(\ref{item:gurke0}) is a  special case
of~(\ref{item:martin}), and~(\ref{item:gurke1}) and~(\ref{item:gurke2}) are
trivial consequences of~(\ref{item:gurke0}).
\end{proof}

Another consequence of Lemma~\ref{lem:prep.is.sigma.preparation} is: 
\begin{Lem}\label{lem:al2cc}
   The forcing notion  $\prep$ is $\al2$-cc.  
\end{Lem}
\begin{proof}
      Recall that we assume that $V$ (and hence $V[G])$
      satisfies CH. 

	Assume towards a contradiction that $(x_i:i< \omega_2)$ is an antichain.
	Using CH we may without loss of generality assume that for each
	$i\in\omega_2$ the transitive collapse of $(M^{x_i},\bar P^{x_i})$ is the
	same.  Set $L_i\DEFEQ  M^{x_i}\cap\om2$.  Using the $\Delta$-lemma we 
  find some uncountable $I\subseteq \om2$ such 
	that the $L_i$ for $i\in I$ form a $\Delta$-system with root~$L$.
  Set $\alpha_0=\sup(L)+ 3$.
	Moreover, we may assume $\sup(L_i)<\min(L_j\setminus \alpha_0)$ for all
	$i<j$.

	Now take any $i,j\in I$, set $x\DEFEQ x_i$ and $y\DEFEQ x_j$, and use
	Lemma~\ref{lem:prep.is.sigma.preparation}(\ref{item:karotte6}).  
	Finally, use Fact~\ref{fact:esmV} to find $z\le x_i, x_j$.
\end{proof}

\subsection{The generic forcing $\BP'$}

Let $G$ be $\prep$-generic.  Obviously $G$ is a $\le_{\prep}$-directed system.
Using the canonical embeddings, we can construct in $V[G]$ a direct
limit $\BP'_{\om2} $ of the directed system~$G$:
Formally, we set 
\[\BP'_{\om2}\DEFEQ \{(x,p):\, x\in G \text{ and }  p\in P^x_{\om2}\},\] 
  and  we
set $(y,q)\leq (x,p)$ if $y\leq_{\prep} x$ and $q$ is (in $y$) stronger than
$i_{x,y}(p)$ (where $i_{x,y}:P^x_{\om2} \to P^y_{\om2} $ is the canonical embedding). 
Similarly, we define for each $\alpha$ 
 \[\BP'_{\alpha}\DEFEQ \{(x,p):\, x\in G,\, \alpha\in
M^x\text{ and } p\in P^x_\alpha\}\] with the same order.   

To summarize:
\begin{Def}\label{def:BPstrich} For $\alpha\leq\om2$, 
  the direct limit of the $P^x_\alpha$ with $x\in G$ is called
  $\BP'_{\alpha}$. 
\end{Def}

Formally, elements of $\BP'_{\om2}$ are defined as pairs $(x,p)$.  However, the $x$
does not really contribute any information. 
 In particular: 
\begin{Fact}\label{facts:trivial66}
  \begin{enumerate}
    \item Assume that $(x,p^x)$ and $(y,p^y)$ are in $\BP'_{\om2}$,
      that $y\leq x$, and that the canonical embedding $i_{x,y}$ witnessing~$y\le x$
       maps $p^x$ to $p^y$. Then $(x,p^x)=^*(y,p^y)$.
    \item $(y,q)$ is in $\BP'_{\om2}$ stronger than $(x,p)$ iff
      for some (or equivalently: for any)
      $z\leq x,y$ in $G$ the canonically embedded $q$
      is in $P^z_{\om2}$ stronger than  the canonically embedded $p$.
      The same holds if ``stronger than'' is replaced by ``compatible
      with'' or by ``incompatible with''. 
    \item\label{item:bla3} If $(x,p)\in\BP'_\alpha$, and if $y$ is such that
      $M^y=M^x$ and $\bar P^y\on\alpha=\bar P^x\on\alpha$,
      then $(y,p)=^*(x,p)$.
  \end{enumerate}
\end{Fact}
In the following, we will therefore  often abuse notation and just write $p$ instead of
$(x,p)$ for an element of~$\BP'_\alpha$.

We can define a natural restriction map from $\BP'_{\om2}$ to $\BP'_\alpha$, by
mapping $(x,p)$ to $(x,p\on \alpha)$.  Note that by the fact above,
we can assume without loss of generality that $\alpha\in M^x$.
More exactly: There is a $y\leq x$ in~$G$ such that $\alpha\in M^y$ (according
to Corollary~\ref{cor:gurke3}).  Then in $\BP'_{\om2}$ we have $(x,p)=^*(y,p)$.

\begin{Fact} \label{fact:5.12} 
  The following is forced by $\prep$:
  \begin{itemize}
    \item $\BP'_\beta$ is  completely embedded into $ \BP'_\alpha$
      for $\beta< \alpha\le \om2$
      (witnessed by the natural restriction map).
    \item If $x\in G$, then $P^x_\alpha$ is $M^x$-completely embedded
      into $\BP'_\alpha$ for $\alpha\leq\om2$ 
      (by the identity map
     $p\mapsto (x,p)$). 
    \item If $\cf(\alpha)>\omega$, then $\BP'_\alpha$ is the union of the
      $\BP'_\beta$ for $\beta<\alpha$. 
    \item By definition, $\BP'_{\om2}$ is a subset of $V$.
  \end{itemize}
\end{Fact}

$G$ will always denote an $\prep$-generic filter, while the $\BP'_{\om2}$-generic
filter over $V[G]$ will be denoted by $H'_{\om2}$ (and the induced
$\BP'_\alpha$-generic by $H'_\alpha$).  Recall that for each $x\in G$, the map
$p\mapsto (x,p)$ is an $M^x$-complete embedding of $P^x_{\om2}$ into $\BP'_{\om2}$ (and of
$P^x_\alpha$ into $\BP'_\alpha$). This way $H'_\alpha\subseteq \BP'_\alpha$ induces an $M^x$-generic
filter $H^x_\alpha \subseteq P^x_\alpha$.

So $x\in \prep$ forces that $\BP'_\alpha$ is approximated by 
$P^x_\alpha$. In particular we get:
\begin{Lem}\label{lem:pathetic0}
  Assume that $x\in \prep$, that $\alpha\leq \om2$ in $M^x$, that $p\in P^x_\alpha$,  that $\varphi(t)$ is a first order formula of the language $\{\in\}$ with 
  one free variable $t$
  and that $\dot \tau$ is a $P^x_\alpha$-name in $M^x$.
  Then $M^x\models p\forces_{P^x_\alpha} \varphi(\dot\tau)$ iff 
  $x\forces_\prep (x,p)\forces_{\BP'_\alpha} M^x[H^x_\alpha]\models \varphi(\dot\tau[H^x_\alpha])$.
\end{Lem}
\begin{proof} 
  ``$\Rightarrow$'' is clear. So assume that $\varphi(\dot\tau)$ is not forced 
  in $M^x$. Then some $q\leq_{P^x_\alpha} p$ forces the negation.
  Now $x$ forces that $(x,q)\leq (x,p)$ in $\BP'_\alpha$; but 
  the conditions $(x,p)$ and  $(x,q)$ force contradictory statements. 
\end{proof}

\subsection{The inductive proof of ccc} \label{sec:ccc}
We will now prove by induction on~$\alpha$
 that $\BP'_\alpha$ is (forced to be) ccc and
(equivalent to) an alternating iteration.  Once we know this, we can prove
Lemma~\ref{lem:elemsub}, which easily implies all the lemmas in this section. So
in particular these lemmas will only be needed to prove
ccc and not for anything else (and they will probably not
aid the understanding of the construction).

In this section, we try to stick to the following notation: $\prep$-names are
denoted with a tilde underneath (e.g., $\n \tau$), while $P^x_\alpha$-names or
$\BP'_\alpha$-names (for any $\alpha\le\om2$) are denoted with a dot accent
(e.g., $\dot\tau$).  We use both accents when we deal with $\prep$-names for
$\BP'_\alpha$-names (e.g., $\nd\tau$).

We first prove a few lemmas that are easy generalizations of the following
straightforward observation:

Assume that $x \forces_\prep(\n z,\n p)\in\BP'_\alpha$.  In particular,
$x\forces \n z\in G$.  We first strengthen $x$ to some $x_1$ that decides $\n
z$ and $\n p$ to be $z^*$ and $p^*$. Then $x_1\leq^* z^*$ (the order $\leq^*$
is defined on page~\pageref{def:starorder}), so we can further strengthen
$x_1$ to some $y\leq z^*$.  By definition, this means that $z^*$ is canonically
embedded into $\bar P^y$; so (by Fact~\ref{facts:trivial66}) 
the $P^{z^*}_\alpha$-condition $p^*$ can be interpreted
as  a $P^y_\alpha$-condition as well. So we end up with some $y\leq x$ and a
$P^y_\alpha$-condition $p^*$ such that $y\forces_\prep (\n z,\n p)=^*(y, p^*)$.

Since $\prep$ is $\sigma$-closed, we can immediately generalize this to countably 
many ($\prep$-names for) $\BP'_{\alpha}$-conditions:
\begin{Fact}\label{fact:pathetic1}
  Assume that $x\forces_{\prep} \n p_n\in \BP'_\alpha$ for all $n\in\omega$.
  Then there is a $y\leq x$ and there are $p_n^*\in P^y_\alpha$ such that
  $y\forces_{\prep} \n p_n=^*p_n^*$ for all $n\in\omega$.
\end{Fact}
Recall that more formally we should write: $x\forces_{\prep} (\n z_n,\n p_n)\in
\BP'_\alpha$; and $y\forces_{\prep} (\n z_n,\n p_n)=^*(y,p_n^*)$.

We will need a 
variant  of the previous fact:
\begin{Lem}\label{lem:pathetic2}
  Assume that $\BP'_\beta$ is forced to be ccc, and assume that 
  $x$ forces (in ${\prep}$) that $\nd r_n$ is a $\BP'_\beta$-name 
  for a real (or an HCON object) for every $n\in\omega$.
  Then there is a $y\leq x$ and there are $P^y_\beta$-names $\dot{r}^*_n$ 
  in $M^y$ such that
  $y\forces_{\prep}   (  \forces _{\BP'_\beta} \nd r_n=\dot{r}^*_n)$ for all $n$.
 
\end{Lem}
(Of course, we mean: $\nd r_n$ is evaluated by $G*H'_\beta$, while $\dot{r}^*_n$
is evaluated by $H_\beta^y$.)
\begin{proof}  
The proof is an obvious consequence of the previous fact,  
since names of reals in a
ccc forcing can be viewed as a countable sequence of conditions. 

In more detail: 
  For notational simplicity assume all $\nd r_n$ are names for
	elements
	of $2^\omega$.
  Working in $V$, we can find for each $n,m\in\omega$
   names for a  maximal antichain $\n A_{n,m}$ and for a function $\n f_{n,m}:\n A_{n,m}\to 2$
  such that $x$ forces that ($\BP'_\beta$ forces that) $\nd r_n(m)=\n f_{n,m}(a)$
  for the unique $a\in \n A_{n,m}\cap H'_\beta$.
  Since $\BP'_\beta$ is ccc, each $\n A_{n,m}$ is countable, and
	since ${\prep}$ is $\sigma$-closed, it is forced 
  that the sequence $\n\Xi=(\n A_{n,m},\n
  f_{n,m})_{n,m\in\omega}$ is in $V$.

  In $V$, we strengthen $x$ to $x_1$ to decide $\n\Xi$ 
  to be some $\Xi^*$. We can also assume that 
  $\Xi^*\in M^{x_1}$ (see Corollary~\ref{cor:gurke3}).
  Each $A^*_{n,m}$ consists of countably many $a$ such that 
  $x_1$ forces $a\in\BP'_\beta$. Using Fact~\ref{fact:pathetic1}
  iteratively (and again the fact that ${\prep}$ is $\sigma$-closed)
  we get some $y\leq x_1$ such that each such $a$ is actually 
  an element of $P^y_\beta$. So in $M^y$, we can use
  $( A^*_{n,m}, f^*_{n,m})_{n,m\in \omega}$ to construct $P^y_\beta$-names $\dot{r}^*_n$
  in the obvious way.

  Now assume that $y\in G$ and that $H'_\beta$ is $\BP'_\beta$-generic
  over $V[G]$. Fix any $a\in A^*_{n,m}=\n A_{n,m}$.
  Since $a\in P^y_\beta$, we get $a \in H^y_\beta$ iff $a\in H'_\beta$.
  So there is a unique element $a$ of $A^*_{n,m}\cap H^y_\beta$, and
  $\dot{r}^*_n(m)=f^*_{n,m}(a)=\n f_{n,m}(a)=\nd r_n(m)$.
\end{proof}

We will also need the following modification:
\begin{Lem}\label{lem:pathetic3}
  (Same assumptions as in the previous lemma.)
  In $V[G][H'_\beta]$, let $\BQ_\beta$ be the union of $Q^z_\beta[H^z_\beta]$
  for all $z\in G$. 
  In $V$, assume that $x$ forces that each $\nd r_n$ is a name for 
  an element of $\BQ_\beta$. Then there is a $y\leq x$ and there is in $M^y$
  a sequence $(\dot r^*_n)_{n\in\omega}$ of
  $P^y_\beta$-names for elements of $Q^y_\beta$ such that 
  $y$ forces $\nd r_n=\dot r^*_n$ for all $n$.
\end{Lem}
So the difference to the previous lemma is: We additionally assume that $\nd
r_n$ is in $\bigcup_{z\in G}Q^z_\beta$, and we additionally get that $\dot r^*_n$ is
a name for an element of $Q^y_\beta$.

\begin{proof}
 Assume $x\in G$ and work in $V[G]$. Fix $n$.
 $\BP'_\beta$ forces that there is some 
 $y_n\in G$ and some 
 $P^{y_n}_\beta$-name $\tau_n\in M^{y_n}$ of an element of $Q^{y_n} _\beta$ such that
 $\nd r_n$ (evaluated by $H'_\beta$) is the same as
 $\tau_n$ (evaluated by $H^{y_n} _\beta$).
 Since we assume that $\BP'_\beta$ is ccc, we can 
 find a countable set $Y_n\subseteq G$ of the possible $y_n$,
 i.e., the empty condition of $\BP'_\beta$ forces $y_n\in Y_n$.
 (As $\prep$ is $\sigma$-closed and $Y_n\subseteq \prep \subseteq V$,
  we must have $Y_n\in V$.)

 So in $V$, there is (for each $n$) an $\prep$-name $\n Y_n$ for this countable
 set. Since $\prep$ is $\sigma$-closed, we can find some $z_0 \leq x$ deciding
 each $\n Y_n$ to be some countable set  $Y_n^* \subseteq \prep $.
 In particular,
 for each $y\in Y_n^*$ we know that $z_0 \forces_\prep y\in G$, i.e.,
 $z_0 \leq^* y$;
 so using once again that ${\prep}$ is $\sigma$-closed we can find
 some $z$ stronger
 than $z_0 $ and all the $y\in \bigcup_{n\in\omega} Y^*_n$. 
 Let $X$ contain all $\tau\in M^y$ such that for some $y\in \bigcup_{n\in\omega} Y^*_n$,
 $\tau$ is a $P^y_\beta$-name for a $Q^y_\beta$-element.
 Since $z\leq y$, 
 each $\tau\in X$ is actually\footnote{%
   Here we use two consequences of
   $z\leq y$: Every $P^y_\beta$-name in $M^y$ can be canonically interpreted
   as a  $P^z_\beta$-name in $M^z$, and $Q^y_\beta$ is (forced to be) a subset
   of $Q^z_\beta$.}
 a $P^z_\beta$-name
 for an element of $Q^z_\beta$. 

 So $X$ is a set of $P^z_\beta$-names for $Q^z_\beta$-elements;
 we can assume that $X\in M^z$.
 Also,  $z$ forces that $\nd r_n\in X$ for all~$n$.
 Using Lemma~\ref{lem:pathetic2}, we can additionally assume that there
 are names  $P^z_\beta$-name $\dot r^*_n$  in $M^z$ such that 
 $z$ forces that  $\nd r_n =  \dot r^*_n$ is forced for each $n$.
 By  Lemma~\ref{lem:pathetic0}, we know that $M^z$
 thinks that $P^z_\beta$ forces that $\dot r^*_n\in X$.  Therefore 
 $\dot r^*_n$ is a $P^z_\beta$-name for a $Q^z_\beta$-element.
\end{proof}

We now prove by induction on $\alpha$ that $\BP'_\alpha$ is equivalent to a ccc
alternating iteration:
\begin{Lem}\label{lem:halbfett}
The following holds in $V[G]$ for $\alpha<\om2$:
\begin{enumerate}
      \item\label{item:iteration}
        $\BP'_\alpha$  is equivalent to
        an alternating iteration. More formally:
        There is an 
        iteration 
        $(\BP_\beta,\BQ_\beta)_{\beta<\alpha}$ with limit $\BP_\alpha$
        that satisfies the definition of alternating iteration 
        (up to $\alpha$), and there is 
        a naturally defined
        dense embedding 
        $j_\alpha:\BP'_\alpha\to \BP_\alpha$,
				such that
        for $\beta < \alpha$ we have $j_\beta \subseteq j_\alpha$, and
        the embeddings commute with the restrictions.\footnote{I.e.,  
				$j_\beta(x,p\on \beta) = j_\alpha(x,p\on \beta) = j_\alpha(x,p) \on
        \beta$.}  
          Each $\BQ_\alpha$ is the union of  all $Q^x_\alpha$
               with~$x\in G$.
        For $x\in G$ with $\alpha\in M^x$,
				the function $i_{x,\alpha}: P^x_\alpha\to \BP_\alpha$
        that maps $p$ to $j_\alpha(x,p)$ is the canonical $M^x$-complete embedding.
     \item In particular, a $\BP'_\alpha$-generic filter $H'_\alpha$
        can be translated into a $\BP_\alpha$-generic filter which we call
        $H_\alpha$ (and vice versa).
      \item\label{item:a1} 
        $\BP_\alpha$ has a dense subset of size~$\al1$.
      \item\label{item:ccc}
        $\BP_\alpha$ is ccc.
			\item\label{item:ch}
        $\BP_\alpha$ forces CH.
\end{enumerate}
\end{Lem}

\begin{proof}
$\alpha=0$ is trivial (since $\BP_0 $ and $\BP'_0$ both 
 are trivial:  $\BP_0$ is a singleton, and $\BP'_0$ consists of 
 pairwise compatible elements).

So assume that all items hold for all $\beta<\alpha$.

\proofsection{Proof of (\ref{item:iteration})}  

	\emph{\textbf{Ultralaver successor case:}} Let $\alpha=\beta+1$ with $\beta$
	an ultralaver position.
  Let $H_\beta$ be  $\BP_\beta$-generic over $V[G]$. Work in $V[G][H_\beta]$.
	By induction, for every $x\in G$ the canonical embedding $i_{x,\beta}$
  defines a $P^x_\beta$-generic filter over~$M^x$ called~$H^x_\beta$. 

  \emph{Definition of $\BQ_\beta$ (and thus of $\BP_{\alpha}$):}
	In $M^x[H^x_\beta]$, the forcing notion $Q^x_\beta$ is defined as $\bL_{\bar D^x}$
  for some system of ultrafilters $\bar D^x$ in $M^x[H^x_\beta]$.
  Fix some $s\in\omega^{{<}\omega}$.
  If $y\leq x$ in $G$, then $D_s^y$ extends $D_s^x$.
  Let $D_s$ be the union of all  $D_s^x$ with $x\in G$.
  So $D_s$ is a proper filter. It is even an ultrafilter:
  Let $r$ be a $\BP_\beta$-name for a real. 
  Using Lemma~\ref{lem:pathetic2}, we know that there is 
  some $y\in G$ and some 
  $P^y_\beta$-name $\n r^y\in M^y$ 
  such that (in $V[G][H_\beta]$)  we have
   $\n r^y[H^y_\beta]=r$.
  So $r\in M^y[H^y_\beta]$, hence 
  either $r$ or its complement is in $D_s^y$ and therefore in $D_s$.
  So all filters in the family $\bar D = (D_s)_{s\in\omega^{{<}\omega}}$ 
  are ultrafilters.

  Now work again in $V[G]$. We set $\BQ_\beta$ to be the
  $\BP_\beta$-name for $\bL_{\bar D}$.
  (Note that $\BP_\beta$ forces that $\BQ_\beta$ literally is the 
  union of the $Q^x_\beta[H^x_\beta]$ for $x\in G$,
  again by Lemma~\ref{lem:pathetic2}.) 

  \emph{Definition of $j_\alpha$:} 
  Let $(x,p)$ be in $\BP'_\alpha$. 
  If $p\in P^x_\beta$, then we set $j_\alpha(x,p)=j_\beta(x,p)$,
  i.e., $j_\alpha$ will extend $j_\beta$. If $p=(p\on\beta,p(\beta))$
  is in $P^x_\alpha$ but not in $P^x_\beta$, we set
  $j_\alpha(x,p)=(r,s)\in \BP_\beta*\BQ_\beta$ where
  $r=j_\beta(x,p\on\beta)$ and 
  $s$ is the ($\BP_\alpha$-name for) $p(\beta)$ as evaluated in
  $M^x[H^x_\beta]$.  From  $\BQ_\beta = \bigcup_{x\in G} Q^x_\beta[H^x_\beta]$
  we conclude that 
  this embedding is dense.

  \emph{The canonical embedding:}
  By induction we know that $i_{x,\beta}$ which 
  maps $p\in P^x_\beta$ to $j_\beta(x,p)$ is 
  (the restriction to $P^x_\beta$ of) the canonical 
  embedding of $x$ into $\BP_{\om2}$. So we have to extend the 
  canonical embedding to $i_{x,\alpha}:P^x_\alpha\to \BP_\alpha$.
  By definition of ``canonical embedding'', $i_{x,\alpha}$ maps 
  $p\in P^x_\alpha$  to the pair
  $(i_{x,\beta}(p\on\beta), p(\beta))$.
  This is the same as $j_\alpha(x,p)$.
	We already know that $D^x_s$ is (forced to be)
  an $M^x[H^x_\beta]$-ultrafilter that is extended by~$D_s$.

  \emph{\textbf{Janus successor case:}} This is similar, but simpler than
  the previous case: Here, $\BQ_\beta$ is just defined as the union of
  all $Q^x_\beta[H^x_\beta]$ for~$x\in G$.
  We will show below that this union satisfies the ccc; 
  just as in Fact~\ref{fact:janus.ctblunion}, 
  it is then easy to see that this union is again a Janus forcing.   

  In particular, $\BQ_\beta$ consists of hereditarily countable objects 
  (since it is the union of Janus forcings, which by definition 
  consist of hereditarily countable objects).
  So since $\BP_\beta$ forces CH, 
  $\BQ_\beta$ is forced to have size~$\al1$.
  Also note that since all Janus forcings involved are separative,
  the union (which is a limit of an in\-com\-patibility-preserving
  directed system) 
  is trivially separative as well.

  \emph{\textbf{Limit case:}} Let $\alpha$ be a limit ordinal.

  \emph{Definition of $\BP_\alpha$ and $j_\alpha$:}
  First we define $j_\alpha: \BP_\alpha' \to \BP^\CS_\alpha$:  For each $(x,p)\in \BP'_\alpha$,
  let $j_\alpha(x,p)\in \BP^\CS_\alpha$
  be the union of all $j_\beta(x,p\on \beta)$ (for $\beta\in \alpha\cap M^x$).
  (Note that $\beta_1<\beta_2$ implies that $j_{\beta_1}(x,p\on \beta_1)$ is
  a restriction of $ j_{\beta_2}(x,p\on \beta_2)$, so this union is indeed
  an element of $\BP^\CS_\alpha$.)

  $\BP_\alpha$ is the set of all $q\wedge p$, where $p\in j_\alpha[\BP'_\alpha]$,
  $q\in \BP_\beta$ for some $\beta < \alpha$, and $q \le p\on \beta$.
  
	It is easy  
	to  check that $\BP_\alpha$ actually is a partial countable
  support limit, and that $j_\alpha$ is dense.
  We will show below that $\BP_\alpha$ satisfies the ccc, so in 
  particular it is proper. 

  \emph{The canonical embedding:}
  To see that $i_{x,\alpha}$ is the (restriction of the) canonical embedding,
  we just have to check that $i_{x,\alpha}$ is $M^x$-complete.
  This is the case since
  $\BP'_\alpha$ is the direct limit of all $P^y_\alpha$
  for $y\in G$ (without loss of generality $y\le x$),
  and each $i_{x,y}$ is $M^x$-complete (see Fact~\ref{fact:5.12}). 

\proofsection{Proof of (\ref{item:a1})}

  Recall that we assume CH in the ground model. 

  The successor case, $\alpha=\beta+1$, follows easily from
  (\ref{item:a1})--(\ref{item:ch}) for $\BP_\beta$
  (since $\BP_\beta$
  forces that $\BQ_\beta$ has size $2^{\al0}=\aleph_1 = \aleph_1^V$). 

  If $\cf(\alpha)>\omega$, 
  then $\BP_\alpha=\bigcup_{\beta<\alpha} \BP_\beta$, so the proof is easy.

  So let $\cf(\alpha)= \omega $.  The following straightforward argument works for 
  any ccc partial CS iteration where all iterands $\BQ_\beta$ are of size $\le \aleph_1$. 

  For notational simplicity we assume $\forces_{\BP_\beta } \BQ_\beta \subseteq
  \omega_1$ for all $\beta<\alpha$ (this is justified by inductive assumption~(\ref{item:ch})).  
  By induction, we can assume that for all $\beta<\alpha$ there is a dense
   $\BP^*_\beta\subseteq \BP_\beta$ of size~$\al1$ and that every  $\BP^*_\beta$ is ccc.
  For each $p\in \BP_\alpha$ and 
  all $\beta\in \dom(p)$
  we can find a maximal antichain 
  $A^p_\beta\subseteq \BP_\beta^*$ such that each element $a\in A^p_\beta$
  decides the value of $p(\beta)$, say $a \forces_{\BP_\beta}
  p(\beta)=\gamma^p_\beta(a)$.   Writing\footnote{Since $\le $ is separative, $p\sim q$ iff $p=^*q$, but
  this fact is not used here.} $p\sim q$ if $p\le q $ and $q\le p$,
  the map $p\mapsto (A^p_\beta, \gamma^p_\beta)_{\beta\in\dom(p)}$ is 1-1
  modulo $\sim$. Since each $A^p_\beta$ is countable, 
  there are only $\al1$ many possible values, therefore there are  only $\al1$
  many $\sim $-equivalence classes.   Any set of representatives will be dense.

  Alternatively, we can prove~(\ref{item:a1}) directly for $\BP'_\alpha$.
  I.e.,
  we can find a $\le^*$-dense subset $\BP'' \subseteq \BP'_\alpha$ of
  cardinality~$\aleph_1$.  Note that a condition $(x,p)\in \BP'_\alpha$
  essentially depends only on $p$ (cf.~Fact~\ref{facts:trivial66}).  More
  specifically, given $(x,p)$ we can ``transitively\footnote{%
  In more detail: We define a function $f:M^x\to V$ by induction as follows: If
  $\beta\in M^x\cap \alpha+1$ or if $\beta=\om2$, then $f(\beta)=\beta$.
  Otherwise, if $\beta\in M^x\cap \ON$, then $f(\beta)$ is the smallest ordinal
  above $f[\beta]$. If $a\in M^x\setminus\ON$, then $f(a)=\{f(b):\, b\in a\cap
  M^x\}$. It is easy to see that $f$ is an isomorphism from $M^x$ to 
  $M^{x'}\DEFEQ f[M^x]$ and that $M^{x'}$ is a candidate. Moreover,
  the ordinals that occur in $M^{x'}$ are subsets of $\alpha+\om1$ together
  with the interval $[\om2,\om2+\om1]$; i.e., there are $\al1$ many
  ordinals that can possibly occur in $M^{x'}$, and therefore there are
  $2^\al0$ many possible such candidates. Moreover, setting $p'\DEFEQ f(p)$,
  it is easy to check that $(x,p)=^*(x',p')$ (similarly to Fact~\ref{facts:trivial66}).
  }
  collapse $x$ above
  $\alpha$'', resulting in a $=^*$-equivalent condition $(x',p')$.  Since
  $|\alpha|=\al1$, there are only $\al1^{\al0}=2^{\al0}$ many such candidates
  $x'$ and since each $x'$ is countable and $p'\in x'$, there are only
  $2^{\al0}$ many pairs $(x',p')$.

\proofsection{Proof of (\ref{item:ccc})}

  \emph{\textbf{Ultralaver successor case:}} Let $\alpha=\beta+1$ with $\beta$ an ultralaver 
  position. 
  We already know that $\BP_\alpha=\BP_\beta*\BQ_\beta$ where $\BQ_\beta$ is an  ultralaver
  forcing, which in particular is ccc, so by induction $\BP_\alpha$
  is ccc.

  {\bf\em Janus successor case:}   
  As above it suffices to show that $\BQ_\beta$,
  the union of the Janus forcings 
  $Q^x_\beta[H^x_\beta]$ for $x\in G$, is (forced to be) ccc.

	Assume towards a contradiction that this is not the case, i.e., that we have
	an uncountable antichain in~$\BQ_\beta$.  We already know that $\BQ_\beta$
  has size $\al1$ and therefore the uncountable antichain has size $\al1$. So,
	working in $V$, we assume towards a contradiction that 
  \begin{equation}\label{eq:ijqprjqr0999}
    x_0\forces_{\prep} p_0\forces_{\BP_\beta} 
       \{ \nd a_i:i\in \omega_1\}\text{ is a maximal (uncountable) antichain in }\BQ_\beta.
  \end{equation}

	We construct by induction on $n\in\omega$ a
	decreasing sequence of conditions such that $x_{n+1}$ satisfies the
  following:
  \begin{enumerate}
    \item[(i)] For all $i\in\om1\cap M^{x_n}$ there is (in $M^{x_{n+1}}$)
      a $P^{x_{n+1}}_\beta$-name $\dot{a}_i^*$ for 
      a $Q^{x_{n+1}}_\beta$-condition 
			such that 
			\[
			  x_{n+1}\forces_{\prep} p_0\forces_{\BP_\beta}\nd a_i=\dot{a}_i^*.
		  \]
			Why can we get that? Just use Lemma~\ref{lem:pathetic3}.
    \item[(ii)] If $\tau$ is in $M^{x_n}$ a $P^{x_n}_\beta$-name for 
      an element of $Q^{x_n}_\beta$, then there is $k^*(\tau)\in\om1$
      such that 
			\[ 
			  x_{n+1}\forces_{\prep} p_0 \forces_{\BP_\beta}\,
			  (\exists i<k^*(\tau))\ \nd a_i \comp_{\BQ_\beta} \tau.
			\]
			Also, all these $k^*(\tau)$ are in $M^{x_{n+1}}$.
      \\
      Why can we get that?
			First note that $x_n\forces p_0\forces (\exists i\in\om1) \ \nd a_i
			\comp \tau $.  Since
      $\BP_\beta$ is ccc, $x_n$ forces that there is some bound $\n k(\tau)$
      for $i$. So it suffices that $x_{n+1}$ determines $\n k(\tau)$ to be 
      $k^*(\tau)$ (for all the countably many $\tau$).
  \end{enumerate}
  Set $\delta^*\DEFEQ \om1\cap\bigcup_{n\in\omega} M^{x_n}$.
  By Corollary~\ref{cor:bigcor}(\ref{item:martin}), there is some $y$ such that
  \begin{itemize}
    \item $y \le x_n$ for all $n\in\omega$,
    \item $(x_n)_{n\in\omega}$ and $(\dot{a}_i^*)_{i\in\delta^*}$ are in $M^y$,
    \item ($M^y$ thinks that) $P^y_\beta$ forces that 
      $Q^y_\beta$ is the union of $Q^{x_n}_\beta$, i.e., as a formula:
      $M^y\models P^y_\beta\forces Q^y_\beta=\bigcup_{n\in\omega} Q^{x_n}_\beta$.
  \end{itemize}
  Let $G$ be ${\prep}$-generic (over $V$) containing $y$, and let $H_\beta$
  be $\BP_\beta$-generic (over $V[G]$) containing $p_0$.

  Set $A^*\DEFEQ \{\dot{a}^*_i[H^y_\beta]:\, i<\delta^*\}$.
  Note that $A^*$ is in $M^y[H^y_\beta]$. We claim 
  \begin{equation}\label{eq:pijqr9}
    A^*\subseteq Q^y_\beta[H^y_\beta]\text{ is predense.}
  \end{equation}
  Pick any $q_0\in Q^y_\beta$.
	So there is some $n\in\omega$ and 
	some $\tau$ which is in
  $M^{x_n}$ a $P^{x_n}_\beta$-name of a $Q^{x_n}_\beta$-condition,
  such that $q_0=\tau[H^{x_n}_\beta]$. By (ii) above,
  $x_{n+1}$ and therefore $y$ forces (in $\prep$)  that for some $i<k^*(\tau)$
	(and therefore some $i< \delta^*$) the condition $p_0$ forces
	the following (in $\BP_\beta$): 
	\begin{quote}
	The conditions 
	$\nd a_i$ and $\tau$ are compatible in~$\BQ_\beta$. Also, 
	$\nd a_i=\dot a^*_i$ and  $\tau $
	both are in $Q^y_\beta$,
        and 	
	$Q^y_\beta$ is an incompatibility-preserving subforcing of $\BQ_\beta$.
	 Therefore $M^y[H^y_\beta]$
	thinks that $\dot a^*_i$ and  $\tau$  are compatible. 
	\end{quote}
	This proves~\eqref{eq:pijqr9}.

  Since $Q^y_\beta[H^y_\beta]$ is $M^y[H^y_\beta]$-complete in $\BQ_\beta[H_\beta]$,
  and since $A^*\in M^y[H^y_\beta]$, 
  this implies  
  (as $\dot{a}^*_i[H^y_\beta]=\nd a_i[G*H_\beta]$ for all
  $i<\delta^*$)
  that $\{\nd a_i[G*H_\beta]:\, i<\delta^*\}$
  already is predense, a contradiction to~\eqref{eq:ijqprjqr0999}.

  {\bf\em Limit case:}
  We work with $\BP'_\alpha$, which by definition 
  only contains HCON objects.

  Assume towards a contradiction that $\BP'_\alpha$ has an
	uncountable antichain. We already know that $\BP'_\alpha$ has 
	a dense subset of size~$\al1$ (modulo $=^*$), so the antichain has size $\al1$.

  Again, work in $V$.  We assume towards a contradiction that 
	\begin{equation}\label{eq:lkjwtoi}
		x_0\forces_{\prep} \{\n a_i:\, i\in\om1\} \text{ is a maximal (uncountable)
		antichain in }\BP'_\alpha.
  \end{equation}
  So each $\n a_i$ is an ${\prep}$-name for an HCON object $(x,p)$ in $V$. 

  To lighten the notation we will abbreviate elements $(x,p)\in \BP'_\alpha$
by~$p$; this is justified by Fact~\ref{facts:trivial66}. 

  Fix any HCON object $p$ and $\beta<\alpha$. 
  We will now define 
	the $({\prep}*\BP' _\beta)$-names $\nd\iota(\beta,p)$ and $\nd r(\beta,p)$:
	Let $G$ be ${\prep}$-generic and containing $x_0$, and 
	$H'_\beta$ be $\BP'_\beta$-generic.
	Let $R$ be the quotient $\BP'_\alpha /  H'_\beta $.
	If $p$ is not in $R$, set $\nd\iota(\beta, p)=\nd r(\beta,p)=0$.
	Otherwise, let $\nd\iota(\beta, p)$ be the minimal $i$ such that 
	$\n a_i\in R$ and $\n a_i$ and $p$ are compatible (in $R$),
	and set $\nd r(\beta, p)\in R$ to be a witness of this compatibility.
	Since $\BP'_\beta$ is (forced to be) ccc, we can find 
	(in~$V[G]$) a countable set $\n X^\iota(\beta, p)\subseteq \omega_1$ 
	containing all possibilities for $\nd\iota(\beta, p)$
	and similarly $\n X^r(\beta, p)$ consisting of HCON objects for $\nd r(\beta, p)$.

	To summarize: 
       For every $\beta<\alpha$
	and every HCON object $p$, we can define (in~$V$)  the ${\prep}$-names
        $\n X^\iota(\beta, p)$ and $\n X^r(\beta, p)$ such that
	\begin{equation} 
    x_0\forces_{\prep} \  \forces_{\BP'_\beta} \biggl(p\in \BP'_\alpha/H'_\beta \ \rightarrow
		 \ (\exists i\in \n X^\iota(\beta,p))\, 
		 (\exists r\in \n X^r(\beta, p))\  r\leq_{\BP'_\alpha/H'_\beta}
		  p,\n a_i\biggr).
	\end{equation}

  Similarly to the Janus successor case, we define by induction on
	$n\in\omega$ a decreasing sequence of conditions such that 
	$x_{n+1}$ satisfies the following:
	For all $\beta\in\alpha\cap M^{x_n}$ and $p\in P^{x_n}_\alpha$,
		  $x_{n+1}$ decides $\n X^{\iota}(\beta,p)$ 
			and $\n X^{r}(\beta,p)$ to be some $X^{\iota*}(\beta,p)$ and $X^{r*}(\beta,p)$.
	    For all $i\in\om1\cap M^{x_{n}}$, 
		  $x_{n+1}$ decides $\n a_i$ to be some $a^*_i\in P^{x_{n+1}}_\alpha$.
			Moreover, each 
			such $X^{\iota*}$  and $X^{r*}$ is in $M^{x_{n+1}}$,
			and every $r\in X^{r*}(\beta,p)$ is in
			$P^{x_{n+1}}_\alpha$.
			(For this, we just use Fact~\ref{fact:pathetic1}
			and Lemma~\ref{lem:pathetic2}.)

	Set $\delta^*\DEFEQ \om1\cap\bigcup_{n\in\omega} M^{x_{n}}$,
	and set $A^*\DEFEQ \{a^*_i:\, i\in\delta^*\}$.
	By  Corollary~\ref{cor:bigcor}(\ref{item:martin}), there is some $y$ such that  
  \begin{gather}
	  \text{$y\leq x_n$ for all $n\in \omega$},\\
	  \text{$\bar x\DEFEQ (x_n)_{n\in\omega}$ and $A^*$ are in $M^y$},\\
	  \label{eq:gqetwet}\text{($M^y$ thinks that) $P^y_\alpha$ is defined as 
	    the almost FS limit  over $\bar x$}.
  \end{gather}
  We claim  that $y$ forces 
	\begin{equation}\label{eq:khweqt}
	  A^*\text{ is predense in } P^y_\alpha.
	\end{equation}
  Since
            $P^y_\alpha$ is $M^y$-completely embedded into $\BP'_\alpha$,
         and  since $A^*\in M^y $ (and since 
	 $\n a_i=a^*_i$ for all $i\in\delta^*$) we get that
	 $\{\n a_i:\, i\in\delta^*\}$ is predense, a contradiction
   to~\eqref{eq:lkjwtoi}.

        So it remains to show~\eqref{eq:khweqt}.  Let $G$ be ${\prep}$-generic containing $y$.
        Let 
	 $r$ be a condition  in $P^y_\alpha$; we will find $i<\delta^*$ such that $r$ 
         is compatible with $a^*_i$.
	Since $P^y_\alpha$ is the almost FS limit over $\bar x$, there 
	is some $n\in\omega$ and $\beta\in \alpha\cap
	M^{x_n}$ such that $r$
	has the form $q\land p$ with 
	$p$ in $P^{x_n}_\alpha$,  $q\in P^y_\beta$ and $q\le p\on \beta$.

        Now  let $H'_\beta$
	be $\BP'_\beta$-generic containing $q$.
	Work in $V[G][H'_\beta]$. Since $q\leq p\on\beta$,
	we get $p\in \BP'_\alpha/H'_\beta$.  
         Let $\iota^*$ be the evaluation by $G*H'_\beta$ of $\nd\iota(\beta,p)$, 
        and let $r^*$ be the evaluation of $\nd r(\beta,p)$.  
        Note that $\iota^*  < \delta^*$ and $r^*\in P^y_\alpha$. 
       So we know that
	$a^*_{\iota^*}$ and $p$ are compatible in $\BP'_\alpha/H'_\beta$
	witnessed by $r^*$. 
        Find $q'\in H'_\beta$ forcing
	$r^*\le_{\BP'_\alpha/H'_\beta} p, a^*_{\iota^*}$.  We may find $q'\le q$.
         Now
	$q'\land r^*$ witnesses that $q\land p$ and $
	a^*_{\iota^*}$ are compatible in~$P^y_\alpha$.
	
  To summarize: The crucial point in proving the ccc is that ``densely'' we
  choose (a variant of) a finite support iteration, see~\eqref{eq:gqetwet}.
  Still, it is a bit surprising that we get the ccc, since we can also argue
  that densely we use (a variant of) a countable support iteration.
  But this does not prevent the ccc, it only prevents 
  the generic iteration from having direct limits in 
  stages
   of countable cofinality.%
   \footnote{Assume that $x$ forces that $\BP'_\alpha$ is the union
   of the $\BP'_\beta$ for $\beta<\alpha$; then we can find a stronger
   $y$ that uses an almost CS iteration over~$x$.  This almost CS      
    iteration contains a  condition $p$ with unbounded support. (Take 
    any condition in the generic part of the almost CS limit; if this
    condition has bounded domain, we can extend it to have unbounded domain, see
    Definition~\ref{def:almost_CS_iteration_wolfgang}.)  Now $p$ will be
    in $\BP'_\alpha$ and have unbounded domain.}

\proofsection{Proof of (\ref{item:ch})}

  This follows from  (\ref{item:a1}) and  (\ref{item:ccc}).
\end{proof}

\subsection{The generic alternating iteration $\mathaccent "7016{\BP}$}
In Lemma~\ref{lem:halbfett} we have seen:
\begin{Cor}\label{cor:summary} 
  Let $G$ be ${\prep}$-generic. Then we can construct\footnote{in an ``absolute 
  way'': Given $G$, we first define $\BP'_\om2$ to be the direct limit of $G$,
  and then inductively construct the $\BP_\alpha$'s from $\BP'_\om2$.}
  (in $V[G]$)
	an alternating iteration $\bar \BP$ such that the following holds:
  \begin{itemize}
    \item $\bar \BP$ is ccc.
		\item If $x\in G$, then $x$ canonically
      embeds into $\bar \BP$. 
      (In particular, a  $\BP_\om2$-generic filter $H_\om2$
      induces a $P^x_\om2$-generic filter over $M^x$, called $H^x_\om2$.)
    \item Each $\BQ_\alpha$ is the union of  all $Q^x_\alpha[H^x_\alpha]$
      with~$x\in G$.
    \item $ \BP_\om2$ is equivalent to the direct limit $\BP'_\om2$ of $G$:
			There is a dense embedding $j:\BP'_\om2\to \BP_\om2$, and
			for each $x\in G$ the function 
      $p\mapsto j(x,p)$ is the canonical embedding.  
  \end{itemize}
\end{Cor}

\begin{Lem}\label{lem:weiothowet}
  Let $x\in {\prep}$. Then ${\prep}$ forces the following: 
   $x\in G$ iff $x$ canonically embeds into $\bar \BP$.
\end{Lem}

\begin{proof}
	If $x\in G$, then we already know  that $x$ canonically embeds into
	$\bar \BP$. 

  So assume (towards a contradiction)
  that $y$ forces that $x$ embeds, but  $y\forces x\notin G$.
  Work in $V[G]$ where $y\in G$.
  Both  $x$ (by assumption) and $y\in G$ canonically embed into $\bar \BP$.
  Let $N$ be an elementary submodel 
  of $H^{V[G]}(\chi^*)$ containing $x,y,\bar \BP$; let $z = (M^z, \bar P^z)$ be the 
  ord-collapse of $(N, \bar \BP)$. Then $z\in V$ (as $\prep$ is $\sigma$-closed)
  and $z\in \prep$,
  and (by elementarity) $z\leq x,y$. This shows that $x\comp_\prep y$,
  i.e., $y$ cannot force $x\notin G$, a contradiction.
\end{proof}

	Using ccc, we can now prove a lemma that is in fact stronger than
the lemmas in the previous Section~\ref{sec:ccc}:
\begin{Lem}\label{lem:elemsub}
  The following is forced by ${\prep}$: Let $N \esm H^{V[G]}(\chi^*)$ be
  countable, and let  $y$ be the ord-collapse of $(N,\bar\BP)$.
	Then $y\in G$. Moreover, if $x\in G\cap N$, then $y \le x$.
\end{Lem}

\begin{proof}
  Work in $V[G]$ with $x\in G$. Pick an elementary
  submodel $N$ containing $x$ and $\bar \BP$. Let $y$ be the 
  ord-collapse of $(N,\bar \BP)$ via a collapsing map $k$. 
  As above, it is clear that $y\in{\prep}$
  and $y\leq x$.  
  To show $y\in G$, it is (by the previous lemma) 
  enough to show that $y$ canonically embeds. 
  We claim that $k^{-1}$ is the canonical embedding of $y$ into $\bar\BP$. 
  The crucial point is to show $M^y$-completeness.  Let $B\in M^y$ be 
  a maximal antichain of $P^y_{\om2}$, say $B=k(A)$ where $A\in N$ is a maximal
  antichain of $\BP_{\om2}$.   So (by ccc) $A$ is countable, hence $A\subseteq N$. 
  So not only $A=k^{-1}(B)$ but even 
           $A=k^{-1}[B]$. 
  Hence $k^{-1}$ is an $M^y$-complete embedding. 
\end{proof}

\begin{Rem} We used the ccc of $\BP_{\om2}$ to prove Lemma~\ref{lem:elemsub};
this use was essential in the sense that we can in turn easily prove the ccc of
$\BP_{\om2}$ if we assume that Lemma~\ref{lem:elemsub} holds. In fact
Lemma~\ref{lem:elemsub} easily implies all other lemmas in
Section~\ref{sec:ccc} as well.
\end{Rem}

\section{The proof of \textup{BC+dBC}}\label{sec:proof}
We first%
\footnote{Note that for this weak version, it would be enough to
produce a generic iteration of length 2 only, i.e., $\BQ_0*\BQ_1$, where
$\BQ_0$ is an ultralaver forcing and $\BQ_1$ a corresponding Janus forcing.}
 prove that no uncountable $X$ in $V$ will be smz or sm in the
final extension $V[G*H]$.
Then we show how to modify the argument to work for all  uncountable sets 
in $V[G*H]$.

\subsection{\textup{BC+dBC} for ground model sets.}\label{sec:groundmodel}


\begin{Lem}\label{lem:6.1}
  Let $X \in V$ be an uncountable set of reals. 
  Then $\prep*\BP_\om2$ forces that $X$ is not smz. 
\end{Lem}
\begin{proof}\ 
\begin{enumerate}
\item
	Fix any even $\alpha< \om2$ (i.e., an
	ultralaver position) in our iteration. 
	The ultralaver forcing $\BQ_\alpha$ adds a
	(canonically defined code for a) 
	closed null set $\dot F$ constructed from the 
	ultralaver real $\bar \ell_\alpha$.
	(Recall Corollary~\ref{cor:absolutepositive}.)
	In the following, when 
	 we consider various ultralaver forcings $\BQ_\alpha$, $Q_\alpha$, $Q^x_\alpha$,
	 we treat
	 $\dot F$ not as an actual name, but rather as a definition 
	 which depends on the forcing used. 
\item
	According to Theorem~\ref{thm:pawlikowski}, it is enough to show that $X+\dot
	F$ is non-null in the $\prep*\BP_{\om2}$-extension, or equivalently, in every
	$\prep*\BP_{\beta}$-extension ($\alpha < \beta<\om2$).
	So assume towards a contradiction  that there is a  $\beta > \alpha$
	and an $\prep*\BP_{\beta}$-name
	$\nd Z$ of a (code for a) Borel null set such that some
	$(x,p)\in \prep*\BP_\om2$ forces that 
	$X + \dot F \subseteq \nd Z$.   
\item   Using the dense embedding  $j_\om2:\BP'_\om2\to \BP_\om2$, we may
        replace $(x,p)$ by a condition $(x,p')\in \prep*\BP'_\om2$. 
 According to
	Fact~\ref{fact:pathetic1} (recall that we now know that $\BP_\om2$ 
        satisfies ccc)
         and Lemma~\ref{lem:pathetic2}
	we can assume 
       that $p'$ is already a $P^x_\beta$-condition $p^x$
       and that
       $\nd Z $ is (forced by $x$ to be the same as)
	a $P^x_\beta$-name $\dot Z^x$ in $M^x$.
\item
	We construct (in $V$) an iteration $\bar P$ in the following way: 
	\begin{enumerate}
  \item[(a)]
  	Up to $\alpha$, we take an arbitrary alternating iteration
    into which $x$ embeds.
	  In particular, $P_\alpha$ will be proper and hence 
	   force that $X$ is still uncountable.
	   \item[(b)] Let $Q_\alpha$ be any ultralaver forcing (over $Q^x_\alpha$
       in case $\alpha\in M^x$). 
       So according 
		   to Corollary~\ref{cor:absolutepositive}, we know that 
       $Q_\alpha$ forces that $X+\dot F$ is not null.

		   Therefore  we can pick (in $V[H_{\alpha+1}]$) 
		   some $\dot r$ in $X+\dot F$ which is random over
       (the countable model) 
			 $M^x[H^x_{\alpha+1}]$, where $H^x_{\alpha+1}$ is induced
			 by $H_{\alpha+1}$. 
		 \item[(c)] In the rest of the construction, we preserve
		   randomness of $\dot r$ over $M^x[H^x_{\zeta}]$ for each $\zeta\le \om2$.
			 We can do this using an almost CS iteration
			 over~$x$
       where at each Janus position we use a random version of Janus forcing and
       at each ultralaver position we use a suitable ultralaver forcing;
       this is possible by Lemma~\ref{lem:4.28}. 
       By Lemma~\ref{lem:preservation.variant}, this iteration
       will preserve the randomness of $\dot r$.
	 	 \item[(d)] So we get $\bar P$ over $x$ 
		  (with canonical embedding $i_x$) and $q\leq_{P_\om2} i_x(p^x)$ 
       such that  $q\on\beta$ forces (in $P_\beta$)
       that $\dot r$ is random over $M^x[H^x_{\beta}]$, in particular that
       $\dot r\notin \dot Z^x$.
	\end{enumerate}
	We now pick a countable $N\esm H(\chi^*)$ containing
  everything and ord-collapse  $(N,\bar P)$ to $y\leq x$. (See Fact~\ref{fact:esmV}.)
	Set $X^y\DEFEQ X\cap M^y$ (the image of $X$ under the collapse).
	By elementarity, $M^y$ thinks that (a)--(d) above holds for $\bar P^y$ 
  and that $X^y$ is uncountable.  Note that  $X^y\subseteq X$.
	\item
  This gives a contradiction in the obvious way:
  Let $G$ be $\prep$-generic over $V$ and contain $y$,
	and let $H_\beta$ be $\BP_\beta$-generic over $V[G]$ and contain $q\on\beta$.
  So $M^y[H^y_\beta]$ thinks that $r\notin \dot Z^x$ (which is 
	absolute) and that $r=x+f$ for some $x\in X^y\subseteq X$
	and $f\in F$ (actually even in $F$ as evaluated in $M^y[H^y_{\alpha+1}]$).
	So in $V[G][H_\beta]$, $r$ is the sum of an element of $X$
	and an element of $F$. So $(y,q)\leq (x,p')$ forces that $\dot r\in (X+\dot
	F)\setminus \nd Z$, a contradiction to~(2).
	\qedhere
\end{enumerate}
\end{proof}

Of course, we need this result not just for ground model sets $X$, but for
$\prep*\BP_{\om2}$-names $\nd X=(\nd x_i:i\in\om1)$ of uncountable sets.  It is
easy to see that it is enough to deal with $\prep*\BP_{\beta}$-names for (all)
$\beta<\om2$. So given $\nd X$, we can (in the proof) pick $\alpha$ such that
$\nd X$ is actually an $\prep*\BP_{\alpha}$-name.  We can try to repeat the same
proof; however, the problem is the following: When constructing $\bar P$
in~(4), it is not clear how to simultaneously make all the uncountably many
names $(\nd x_i)$ into $\bar P$-names in a sufficiently ``absolute'' way.  In
other words: It is not clear how to end up with some $M^y$ and $\dot X^y$
uncountable in $M^y$ such that it is guaranteed that $\dot X^y$ (evaluated in
$M^y[H^y_{\alpha}]$) will be a subset of $\nd X$ (evaluated in
$V[G][H_\alpha]$). We will solve this problem in the next section by factoring
$\prep$.

Let us now give the proof of the corresponding weak version of dBC:

\begin{Lem}\label{lem:6.2}
  Let $X \in V$ be an uncountable set of reals. 
  Then $\prep*\BP_\om2$ forces that $X$ is not strongly meager. 
\end{Lem}
\begin{proof}
The proof is parallel to the previous one:
\begin{enumerate}
\item
	Fix any even $\alpha< \om2$ (i.e., an
	ultralaver position) in our iteration. 
	The Janus forcing $\BQ_{\alpha+1}$ 
	adds a (canonically defined code for a) 
	null set $\dot Z_\nabla$.
	(See Definition~\ref{def:Znabla} and Fact~\ref{fact:Znablaabsolute}.)
\item
	According to~\eqref{eq:notsm}, it is enough to show that
	$X+\dot Z_\nabla=2^\omega$ in the $\prep*\BP_{\om2}$-extension, or equivalently, in every
	$\prep*\BP_{\beta}$-extension  ($\alpha<\beta<\om2$).
	(For every real $r$, the statement
	$r\in X+\dot Z_\nabla$, i.e., $(\exists x\in X)\ x+r\in\dot Z_\nabla$, is absolute.)
	So assume towards a contradiction  that there is a  $\beta > \alpha$
	and an $\prep*\BP_{\beta}$-name
	$\nd r$ of a real such that some
	$(x,p)\in \prep*\BP_\om2$ forces that 
	$\nd r\notin X + \dot Z_\nabla$.
\item Again, we can assume that $\nd r $ is a $P^x_\beta$-name $\dot r^x$ in $M^x$. 
\item
  We construct (in $V$) an iteration $\bar P$  in the following way: 
	\begin{enumerate}
  \item[(a)]
	Up to $\alpha$, we take an arbitrary alternating
	iteration into which $x$ embeds.
	In particular, $P_\alpha$ again forces that $X$ is still uncountable.
	   \item[(b1)] Let $Q_\alpha$ be any ultralaver forcing (over $Q_\alpha^x$). Then
       $Q_\alpha$ forces that $X$ is not thin
		   (see Corollary~\ref{cor:LDnotthin}).
		 \item[(b2)] Let $Q_{\alpha+1}$ be a countable Janus forcing. 
       So $Q_{\alpha+1}$ forces $X+\dot Z_\nabla=2^\omega$. (See  Lemma~\ref{lem:janusnotmeager}.)
     \item[(c)]  We continue the iteration in a $\sigma$-centered way.
       I.e., we use an almost FS iteration over $x$ of
       ultralaver forcings and countable Janus forcings, 
       using trivial $Q_\zeta$ for all $\zeta\notin M^x$; see
       Lemma~\ref{lem:4.17}.
	 	 \item[(d)] So $P_\beta$ still
			 forces that $X+\dot Z_\nabla=2^\omega$, and in particular that $\dot
       r^x\in X+\dot Z_\nabla$.
    	 (Again by  Lemma~\ref{lem:janusnotmeager}.)
	\end{enumerate}
	Again, by collapsing some $N$ as in the previous proof,
	we get $y\le x $ and $X^y\subseteq X$.
	\item
  This again gives the obvious contradiction:
  Let $G$ be $\prep$-generic over $V$ and contain $y$,
	and let $H_\beta$ be $\BP_\beta$-generic over $V[G]$ and contain $p$.
  So $M^y[H^y_\beta]$ thinks that
	$r=x+ z$ for some $x\in X^y\subseteq X$
	and $z \in Z_\nabla$ (this time, $\dot Z_\nabla$ is evaluated in $M^y[H^y_{\beta}]$),
	contradicting~(2).
	\qedhere
\end{enumerate}
\end{proof}

\subsection{A factor lemma}\label{sec:factor}

We can restrict $\prep$ to any ${\alpha^*}<\om2$ 
in the obvious way: Conditions are pairs
$x=(M^x,\bar P^x)$ of nice candidates $M^x$ (containing ${\alpha^*}$) and alternating
iterations $\bar P^x$, but now $M^x$ thinks that $\bar P^x$ has length ${\alpha^*}$
(and not $\om2$).  We call this variant $\prep\on{\alpha^*}$.

Note that all results of Section~\ref{sec:construction} about $\prep$ are still true for
$\prep\on{\alpha^*}$. In particular, whenever $G\subseteq \prep\on\alpha^*$ is generic, 
it will define a direct limit (which we call $\BP^{\prime*}$), and an
alternating iteration of length $\alpha^*$ (called  $\bar \BP^*$);  again
we will have that $x\in G$ iff $x$ canonically embeds into $\bar
\BP^*$.

There is a natural projection map from $\prep$ (more exactly:  from the dense
subset of those $x$ which satisfy ${\alpha^*}\in M^x$) into $\prep\on\alpha^*$,
 mapping $x=(M^x,\bar P^x)$ to
$x\on{\alpha^*}\DEFEQ (M^x,\bar P^x\on{\alpha^*})$.  (It is obvious that this projection is
dense and preserves $\leq$.)

There is also a natural embedding $\varphi$ from $\prep\on{\alpha^*}$ to $\prep$: We
can just continue an alternating iteration of length ${\alpha^*}$ by appending
trivial forcings.

$\varphi$ is complete: It preserves $\leq$ and $\incomp$. (Assume that $z\leq
\varphi(x),\varphi(y)$. Then $z\on{\alpha^*}\leq x,y$.) Also, the projection is a
reduction: If $y\leq x\on{\alpha^*}$ in $\prep\on{\alpha^*}$, then let $M^z$ be a model 
containing both $x$ and $y$.  In $M^z$, we can first  construct 
 an alternating iteration of length $\alpha^*$ over $y$
(using almost FS over $y$, or almost CS ---
 this does not matter here).   We then continue this iteration  $\bar P^z$
using almost FS or almost CS  over $x$. 
So $x$ and $y$ both  embed into $\bar P^z$, hence 
  $z=(M^z,\bar P^z)\leq x,y$.

So according to the general factor lemma of forcing theory, we know that
$\prep$ is forcing equivalent to $\prep\on{\alpha^*} * (\prep/\prep\on{\alpha^*})$,
where $\prep/\prep\on{\alpha^*}$ is the quotient of $\prep$ and $\prep\on{\alpha^*}$,
i.e.,
the ($\prep\on{\alpha^*}$-name for the) 
set of $x\in\prep$ which  are compatible (in $\prep$) with all $\varphi(y)$ for
$y\in G\on{\alpha^*}$ (the generic filter for $\prep\on{\alpha^*}$),
or equivalently, the set of $x\in\prep$ such that $x\on{\alpha^*}\in
G\on{\alpha^*}$. So  Lemma~\ref{lem:weiothowet} (relativized to
$\prep\on\alpha^*$) implies:
\proofclaim{eq:oetwji}{$\prep/\prep\on{\alpha^*}$ is the set of $x\in\prep$ that
canonically embed (up to ${\alpha^*}$) into $\BP_{\alpha^*}$.}

\begin{Setup}
Fix some $\alpha^*<\om2$ of uncountable cofinality.\footnote{Probably the cofinality is completely irrelevant, but the picture is clearer this way.}
Let $G\on{\alpha^*}$ be 
$\prep\on{\alpha^*}$-generic over $V$ and work in $V^*\DEFEQ V[G\on{\alpha^*}]$.
Set $\bar\BP^*=(\BP^*_\beta)_{\beta<\alpha^*}$, the generic alternating
iteration added by $\prep\on{\alpha^*}$. 
Let $\prep^*$ be the quotient
$\prep/\prep\on\alpha^*$.
\end{Setup}

We claim that $\prep^*$ satisfies (in $V^*$) all the properties
that we proved in Section~\ref{sec:construction}  for $\prep$ (in $V$), with
the obvious modifications.  In particular:

\begin{enumerate}[(A)$_{\alpha^*}$]
  \item $\prep^*$ is $ \al2$-cc, since it is
    the quotient of an  $\al2$-cc forcing. 
  \item $\prep^*$ does not add new reals (and more generally, no new
    HCON objects), since it is the quotient of a $\sigma$-closed forcing.\footnote{It is 
    easy to see that $\prep^*$ is even $\sigma$-closed, by ``relativizing'' the
    proof for $\prep$, but we will not need this.} 
  \item 
    Let $G^*$ be $\prep^*$-generic over $V^*$. Then $G^*$ is $\prep$-generic
    over $V$, and therefore
		Corollary~\ref{cor:summary} holds for~$G^*$. 
    (Note that $\BP'_\om2$ and then $\BP_\om2$ is constructed from~$G^*$.)
   Moreover, it is easy to see%
 \footnote{%
     For $\beta \le  \alpha^*$, let 
      $\BP^{\prime*}_\beta$ be the direct limit of $(G\on\alpha^*)\on \beta$
      and $\BP^{\prime}_\beta$ the direct limit of $G^*\on\beta$.
      The function  $k_\beta: \BP^{\prime*}_\beta\to \BP^{\prime}_\beta$
      that maps $(x,p)$ to $(\varphi(x),p)$ 
      preserves $\leq$ and $\incomp$ and is surjective
      modulo $=^*$,
      see Fact~\ref{facts:trivial66}(\ref{item:bla3}).
      So it is clear that 
      defining $\bar\BP^*\on\beta$ by induction
      from $\BP^{\prime*}_\beta$ yields the same result as 
      defining $\bar\BP\on\beta$ from $\BP_{\beta}'$.
    }  
    that $\bar\BP$ starts with $\bar\BP^*$.
 \item  In particular,  we get a variant of Lemma~\ref{lem:elemsub}:
     The following is forced by ${\prep^*}$: Let $N \esm H^{V[G^*]}(\chi^*)$ be
  countable, and let  $y$ be the ord-collapse of $(N,\bar\BP)$.
  Then $y\in G^*$. Moreover:  If $x\in G^*\cap N$, then $y \le x$.
\end{enumerate}

We can use the last item to prove the $\prep^*$-version of
Fact~\ref{fact:pathetic1}: 
\begin{Cor}\label{cor:slkjte}
In $V^*$, the following holds: 
\begin{enumerate}
\item \label{item:fangen.a}
  Assume that $x\in\prep^*$ forces that 
  $p\in \BP_\om2$. Then there is a $y\leq x$
  and a $p^y\in P^y_\om2$ such that $y$ forces
  $p^y=^*p$. 
\item \label{item:fangen.b}
  Assume that $x\in\prep^*$ forces that 
  $\nd r$ is a $\BP_{\om2}$-name of a real.   Then there is a $y\leq x$
  and a $P^y_\om2$-name $\dot r^y$  such that $y$ forces
  that $\dot r^y $ and $\nd r $ are equivalent as  $\BP_{\om2}$-names. 
\end{enumerate}
\end{Cor}

\begin{proof} We only prove (\ref{item:fangen.a}), the proof of (\ref{item:fangen.b})
is similar.

  Let $G^*$ contain $x$. In $V[G^*]$, pick an elementary
  submodel $N$ containing $x,p,\bar\BP$ and let $(M^{z},\bar P^{z},p^{z})$
  be the ord-collapse of $(N,\bar\BP,p)$.
  Then $z\in G^*$. 
  This whole situation is forced by some $y\leq z\leq x\in G^*$.
  So $y$ and $p^y$ is as required, where 
  $p^y\in P^y_\om2$ is the canonical image of $p^z$.
\end{proof}

We also get the following analogue of Fact~\ref{fact:esmV}: 
\proofclaim{claim:44}{
In $V^*$ we have:    Let $x\in \prep^*$.  
		Assume that $\bar P$ is an alternating iteration that extends
    $ \bar \BP\on \alpha^*$ and 
      that $x=(M^x,\bar P^x) \in
		\prep$ canonically embeds into $\bar P$, and that $N \esm H( \chi^*)$
		contains $x$ and $\bar P$.  Let $y=(M^y,  \bar P^y)$ be the ord-collapse of
		$(N, \bar P)$.  Then $y\in\prep^*$ and  $y\le x$. 
}

We now claim that $\prep*\BP_\om2$ forces BC+dBC.
We know that $\prep$ is forcing equivalent to $\prep\on{\alpha^*} *
\prep^*$.  Obviously we have
\[
  \prep*\BP_\om2=\prep\on{\alpha^*}* \prep^**\BP_{\alpha^*} * \BP_{{\alpha^*},\,\om2}
\] (where $\BP_{{\alpha^*},\,\om2}$ is the quotient of
$\BP_\om2$ and $\BP_{\alpha^*}$).
Note that $\BP_{\alpha^*}$ is already determined by $\prep\on{\alpha^*}$,
so $\prep^**\BP_{\alpha^*}$ is (forced by $\prep\on{\alpha^*}$ to be) 
a product $\prep^*\times \BP_{\alpha^*}=\BP_{\alpha^*}\times \prep^*$.

But note that this is not the same as $\BP_{\alpha^*} * \prep^*$, where we
evaluate the definition of~$\prep^*$ in the $\BP_{\alpha^*}$-extension of
$V[G\on{\alpha^*}]$: We would get new candidates and therefore new conditions
in~$\prep^*$ 
after forcing with~$\BP_{\alpha^*}$.  In other words, we can 
\emph{not} just argue as follows:
\begin{wrongproof}
   $\prep*\BP_\om2$ is the same as
   $(\prep\on{\alpha^*}* \BP_{\alpha^*})* (\prep^**\BP_{{\alpha^*},\om2})$;
   so given an $\prep*\BP_\om2$-name $X$ of a set of reals of size~$\al1$,
   we can choose $\alpha^*$ large enough so that 
   $X$ is an $(\prep\on{\alpha^*}* \BP_{\alpha^*})$-name. Then,
   working in the $(\prep\on{\alpha^*}* \BP_{\alpha^*})$-extension,
   we just apply Lemmas~\ref{lem:6.1} and~\ref{lem:6.2}.
\end{wrongproof}

So what do we do instead?
Assume that $\nd X=\{\nd\xi_i:\, i\in\om1\}$ is an $\prep*\BP_{\om2}$-name for
a set of reals of size~$\aleph_1$. So there is a $\beta<\om2$ such that 
$\nd X$ is added by $\prep*\BP_\beta$.  
In the $\prep$-extension, $\BP_{\beta}$ is ccc, therefore we can assume
that each $\nd \xi_i$ is a system of countably many countable
antichains $\n A^m_i$ of~$\BP_\beta$,
together with functions $\n f^m_i:\n A^m_i\to\{0,1\}$. 
For the following argument,
we prefer to work with the equivalent $\BP_\beta'$ instead of~$\BP_\beta$. 
  We can assume that each of the sequences $B_i\DEFEQ (\n
A^m_i,\n f^m_i)_{m\in\omega}$ is an element of~$V$ (since $\BP'_\beta$ is a
subset of~$V$ and since $\prep$ is $\sigma$-closed).
So each $B_i$ is decided by a maximal antichain~$Z_i$ of~$\prep$.
Since $\prep$ is $\al2$-cc, these $\al1$ many antichains all
are contained in some $\prep\on {\alpha^*}$ with  ${\alpha^*}\geq \beta$.

 So in the $\prep\on {\alpha^*}$-extension $V^*$ we 
have the following situation:
Each $\xi_i$ is a very ``absolute\footnote{or: ``nice'' in the sense of~\cite[5.11]{MR597342}}''
$\prep^* * \BP_{\alpha^*}$-name
(or equivalently, $\prep^* \times \BP_{\alpha^*}$-name),
in fact they are already determined by antichains that are in $\BP_{\alpha^*}$
and do not depend on $\prep^*$. So we can interpret them as
$\BP_{\alpha^*}$-names.

Note that: 
\proofclaim{claim:xi.i}{ The $\xi_i$ are forced
 (by $\prep^**\BP_{\alpha^*}$)
 to be pairwise different,
and therefore already by $\BP_{\alpha^*}$.}

Now we are finally ready to prove that $\prep* \BP_{\om2}$ forces that every uncountable $X$ 
is neither smz nor sm.  It is enough to show that for every name $\nd X$
of an uncountable set of reals of size $\al1$ the forcing  $\prep* \BP_{\om2}$
forces that $\nd X$ is neither smz nor sm.   For the rest of the proof
we fix such a name $\nd X$, the corresponding $\nd \xi_i$'s (for $i\in \omega_1$),
and the appropriate $\alpha^*$ as above.  From now on, we 
 work in the $\prep\on\alpha^*$-extension~$V^*$.  

So we have to show that $\prep^* * \BP_{\om2}$ forces that $\nd X$ 
is neither smz nor sm.

After all our preparations, we can just repeat the proofs of BC (Lemma~\ref{lem:6.1}) 
and dBC (Lemma~\ref{lem:6.2}) 
 of Section~\ref{sec:groundmodel}, with the
following modifications.   The modifications are the same for both proofs; 
for better readability we describe the results of the change only for the proof of dBC.

\begin{enumerate}
  \item Change:  
Instead of an arbitrary ultralaver position $\alpha<\om2$, we obviously 
have to choose $\alpha\geq \alpha^*$.
\\ For the dBC: We choose $\alpha\ge \alpha^*$ an arbitrary ultralaver position.   
 The Janus forcing $\BQ_{\alpha+1}$
        adds a (canonically defined code for a)
        null set $\dot Z_\nabla$.
  \item
   Change: No change here.
(Of course we now have an $\prep^**\BP_{\alpha^*}$-name $\nd X$ instead 
of a ground model set.)\\
 For the dBC: 
        It is enough to show that
        $\nd X+\dot Z_\nabla=2^\omega$ in the $\prep^**\BP_{\om2}$-extension of~$V^*$,
         or equivalently, in every
        $\prep^**\BP_{\beta}$-extension  ($\alpha<\beta<\om2$).
        So assume towards a contradiction  that there is a  $\beta > \alpha$
        and an $\prep^**\BP_{\beta}$-name
        $\nd r$ of a real such that some
        $(x,p)\in \prep^**\BP_\om2$ forces that
        $\nd r\notin \nd X + \dot Z_\nabla$.
  \item Change: no change.  (But we  use Corollary~\ref{cor:slkjte}
      instead of Lemma~\ref{lem:pathetic2}.)\\
      For dBC:  
        Using  Corollary~\ref{cor:slkjte}(\ref{item:fangen.b}), without loss of generality $x$ forces $p^x=^* p $
           and  there is a $P^x_\beta$-name $\dot  r^x$ in $M^x$ such that $\dot r^x=\nd r$ is forced. 
  \item Change: 
The  iteration obviously has to start with the $\prep\on\alpha^*$-generic iteration
$\bar\BP^*$  (which is ccc),
the rest is the same.  
\\ For dBC: 
  In $V^*$ we construct an iteration $\bar P$ in the following way:
	\begin{enumerate}
  \item[(a1)]
  	Up to $\alpha^*$, we use the iteration $\bar\BP^* $ (which  
        already lives in our  current universe $V^*$).  As explained
       above in the paragraph preceding~\eqref{claim:xi.i}, 
	  $\nd X$ can be interpreted as a $\BP_{\alpha^*}$-name $\dot X$, and by
	  \eqref{claim:xi.i}, $\dot X $ is forced to be uncountable. 
  \item[(a2)]
  	We continue the iteration from $\alpha^*$ to $\alpha$ in a way that
	 embeds $x$ and such that $P_\alpha$ is proper.   
	  So  $P_{\alpha}$ will  force 
	   that $\dot X$ is still uncountable.
	   \item[(b1)] Let $Q_\alpha$ be any ultralaver forcing (over $Q_\alpha^x$). Then
       $Q_\alpha$ forces that $\dot X$ is not thin.
		 \item[(b2)] Let $Q_{\alpha+1}$ be a countable Janus forcing. 
       So $Q_{\alpha+1}$ forces $\dot X+\dot Z_\nabla=2^\omega$. 
     \item[(c)]  We continue the iteration in a $\sigma$-centered way.
       I.e., we use an almost FS iteration over $x$ of
       ultralaver forcings and countable Janus forcings, 
       using trivial $Q_\zeta$ for all $\zeta\notin M^x$. 
	 	 \item[(d)] So $P_\beta$ still
			 forces that $\dot X+\dot Z_\nabla=2^\omega$, and in particular that $\dot
       r^x\in \dot  X+\dot Z_\nabla$.
	\end{enumerate}  
	  We now pick (in $V^*$)  a countable $N\esm H(\chi^*)$ containing
  everything and ord-collapse  $(N,\bar P)$ to $y\leq x$, by \eqref{claim:44}. 
  The HCON object $y$ is of course in $V$ (and even in $\prep$), but we can say more:  Since the iteration $\bar P$ starts with the $(\prep\on\alpha^*)$-generic iteration $\bar\BP^*$, the condition $y$ will be in the quotient forcing $\prep^*$.\\
	Set $\dot X^y\DEFEQ \dot X\cap M^y$ (which is 
the image of $\dot X$ under the collapse, since we view $\dot X$ as a set of 
  HCON-names).
	By elementarity, $M^y$ thinks that (a)--(d) above holds for $\bar P^y$ 
  and that $\dot X^y$ is forced to be  uncountable.  
  Note that  $\dot X^y\subseteq \dot X$ in the following sense:  
   Whenever   $G^**H$ is $\prep^**\BP_{\om2}$-generic over $V^*$, and $y\in G^*$, then the evaluation of  $\dot X^y $ in $M^y[H^y]$ is a subset of the evaluation of $\dot X$ in $V^*[G^**H]$.

  \item Change:  No change here.\\
    For dBC: We get our desired contradiction as follows:\\
  Let $G^*$ be $\prep^*$-generic over $V^*$ and contain $y$.
	Let $H_\beta$ be $\BP_\beta$-generic over $V^*[G^*]$ and contain $p$.
  So $M^y[H^y_\beta]$ thinks that
	$ r=x+ z$ for some $x\in   X^y\subseteq  X$
	and\footnote{Note
 that we get the same Borel code, whether we evaluate $\dot Z_\nabla$ in $M^y[H^y_{\beta}]$ or in $V^*[G^**H_\beta]$. Accordingly, 
the actual Borel set of reals coded by $Z_\nabla$ in the smaller universe
is a subset of the corresponding Borel set in the larger universe.} 
 $z \in  Z_\nabla$, contradicting~(2).
\end{enumerate}

\section{A word on variants of the definitions}\label{sec:alternativedefs}
The following is not needed for understanding the paper, we just
briefly comment on alternative ways some notions could be defined.

\subsection{Regarding \qemph{alternating iterations}}  \label{sec:7a}

  We  call the set of $\alpha\in\om2$ such that $Q_\alpha$ is (forced to be) nontrivial 
	the \qemph{true domain} of $\bar P$ (we use this notation in this 
	remark only). Obviously $\bar P$ is naturally isomorphic
  to an iteration whose length is the order type of its true domain.
	In  Definitions~\ref{def:alternating} and~\ref{def:prep}, we could have
 	imposed 
	the following additional requirements. All these variants lead
	to equivalent forcing notions. 
  \begin{enumerate}
    \item $M^x$ is (an ord-collapse of) an
      {\em elementary} submodel of $H(\chi^*)$.
			\\
			This is equivalent, as conditions coming from
			elementary submodels are dense in our $\prep$, by 
			Fact~\ref{fact:esmV}. 
			\\
      While this definition looks much simpler and therefore 
			nicer (we could replace ord-transitive models
			by the better understood elementary models), it would not make 
      things easier and just ``hides'' the point of the 
      construction: For example, we use models $M^x$ that
      are (an ord-collapse of) an
      elementary submodel of~$H^{V'}(\chi^*)$ for some forcing extension
      $V'$ of~$V$.
    \item Require that ($M^x$ thinks that) the true domain of $\bar P^x$
		  is $\om2$.
			\\
		This is equivalent for the same reason as (1) (and this
		requirement is compatible with (1)).
			\\
		This definition would allow to drop the ``trivial'' option
		from the definition.  The whole proof would still work with 
                minor modifications --- in particular, because of the following
                fact:
                \footnote{We are grateful to Stefan
                Geschke and Andreas Blass for pointing out
                this fact. The only references we are aware
                of are \cite[proof of Lemma 2]{MR1179593} and \cite{MO84129}.}
                \proofclaim{eq:blass}{The finite support iteration of
                  $\sigma$-centered forcing notions of
                  length $<(2^{\aleph_0})^+$ is again
                  $\sigma$-centered.}
  We chose our version for two reasons: first, it seems
  more flexible, and second, we were initially not aware of 
  \eqref{eq:blass}.  
	  \item Alternatively, require that ($M^x$ thinks that) the true  
  		domain of $\bar P^x$ is countable.
			\\
			Again, equivalence can be seen as in~(1), again (3) is compatible with~(1)
			but obviously not with~(2).
			\\
			This requirement would not make the definition easier, 
			so there is no reason to adopt it. It would have the slight 
			inconvenience 
			that instead of using ord-collapses as in Fact~\ref{fact:esmV},
			we would have to put another model on top to make the iteration
			countable. Also, it would have the (purely aesthetic) disadvantage that
			the generic iteration itself does not satisfy this
			requirement. 
    \item \label{item:nonproper}
      Also, we could have dropped the requirement that the iteration 
      is proper. It is never directly used, and ``densely'' 
      $\bar P$ is proper anyway. (E.g., in Lemma~\ref{lem:6.1}(4)(a), 
      we would just construct $\bar P$ up to $\alpha$ to
      be proper or even ccc, so that $X$ remains uncountable.)
  \end{enumerate}

\subsection{Regarding \qemph{almost CS iterations and separative iterands}}
\label{sec:7b}

  Recall that in Definition~\ref{partial_CS} we 
  required that each  iterand $Q_\alpha$ in a partial CS iteration
  is separative. This implies the property (actually: the three equivalent
  properties) from  Fact~\ref{fact:suitable.equivalent}.
  Let us call this property \qemph{suitability} for now.
  Suitability is a property of the limit
  $P_\varepsilon$ of $\bar P$.  Suitability always holds for finite support
  iterations and for countable support iterations.  However, if we do not assume
  that each $Q_\alpha$ is separative, then  suitability may fail 
	for partial CS iterations. 
  We could drop the separativity assumption, and instead
  add suitability as an additional natural requirement to the definition of partial CS limit.

  The disadvantage of this approach is that we would have to check in all 
  constructions of partial CS iterations that suitability is 
  indeed satisfied
	(which we found to be straightforward but rather cumbersome, in
  particular in the case of the almost CS iteration).

  In contrast, the disadvantage of assuming that $Q_\alpha$ is separative
  is minimal and purely cosmetic: It is well known that every quasiorder $Q$
  can be made into a  separative one
  which is  forcing equivalent to the original~$Q$   (e.g., by just 
  redefining the order to be $\leq^*_Q$).

\subsection{Regarding \qemph{preservation of random and quick sequences}}
  Recall Definition~\ref{def:locally.random}  of local preservation of random reals
  and Lemma~\ref{lem:4.28}.

    In some respect the dense sets $D_n$ are unnecessary.  For
    ultralaver forcing $\bL_{\bar D}$,  the notion of a ``quick'' sequence
    refers to the sets $D_n$ of conditions with stem of length at
    least~$n$. 

    We could define a new partial order on $\bL_{\bar D}$ as
    follows: 
     \[ q\le' p \ \Leftrightarrow \ (q=p) \ \text{or} \ (q\le p \text{ and the
     stem of $q$ is strictly  longer than the stem of~$p$}).
     \]
    Then $(\bL_{\bar D}, \le)$ and 
         $(\bL_{\bar D}, \le')$ are forcing equivalent, and any
	 $\le'$-interpretation of a new real will automatically be quick.

    Note however that  $(\bL_{\bar D}, \le')$ is now not separative any more.  
    Therefore we chose not to take this approach, since losing separativity
    causes technical inconvenience, as described in \ref{sec:7b}. 

\bibliographystyle{alpha}     
\bibliography{BCdBC}

\end{document}